\documentclass[oneside]{amsart}
\usepackage{esint,amssymb,bm,enumitem}

\usepackage[draft=false,bookmarksopenlevel=2]{hyperref}

\usepackage{tikz}
\usetikzlibrary{calc,patterns}

\hypersetup{%
pdfborder = {0 0 0}}

\newtheorem{theorem}[equation]{Theorem}
\newtheorem{lemma}[equation]{Lemma}
\newtheorem{corollary}[equation]{Corollary}
\theoremstyle{definition}
\newtheorem{definition}[equation]{Definition}
\newtheorem{remark}[equation]{Remark}

\newcommand{\RR}{\mathbb{R}}\let\R\RR
\newcommand{\NN}{\mathbb{N}}
\newcommand{\CC}{\mathbb{C}}
\newcommand{\es}{\text{ess sup}}
\newcommand{\rn}[1]{\MakeUppercase{\romannumeral #1}}
\DeclareMathOperator{\Div}{div}
\DeclareMathOperator{\diam}{diam}
\DeclareMathOperator{\dist}{dist}
\DeclareMathOperator{\supp}{supp}
\newcommand\mat{\bm} % Matrices
\newcommand{\dmn}{d}
\newcommand{\pdmn}{d}
\newcommand{\dmnMinusOne}{{d-1}}
\newcommand{\arr}[1]{\mathbf{\dot{#1}}}
\newcommand{\abs}[1]{|#1|}
\newcommand{\doublebar}[1]{\|#1\|}

\title[Gradient estimates and the fundamental solution]{Gradient Estimates And The Fundamental Solution For Higher-Order Elliptic Systems With Lower-Order Terms}

\author{Ariel Barton}
\address{Ariel Barton, Department of Mathematical Sciences,
			309 SCEN,
			University of Ar\-kan\-sas,
			Fayetteville, AR 72701}
\email{aeb019@uark.edu}
\author{Michael J. Duffy Jr.}
\address{Michael J. Duffy, Department of Mathematical Sciences,
			309 SCEN,
			University of Ar\-kan\-sas,
			Fayetteville, AR 72701}
\email{mjduffy@uark.edu}


\begin{document}

\begin{abstract}
We establish the Caccioppoli inequality, a reverse H\"older inequality in the spirit of the classic estimate of Meyers, and construct the fundamental solution for linear elliptic differential equations of order $2m$ with certain lower order terms.
\end{abstract}

\maketitle

\tableofcontents

\section{Introduction}\label{intro}

There is at present a very extensive theory for second order linear elliptic differential operators without lower order terms. Such an operator~$L$ may be written as
\begin{equation}
\label{eqn:intro:classic}
(L\vec{u})_j=-\sum_{k=1}^N\sum_{a=1}^d\sum_{b=1}^d
\partial_{a}(A^{j,k}_{a,b}\partial_{b} u_k)\end{equation}
where $\vec u$ is a function defined on a subset of~$\R^d$.
Two important generalizations are higher order operators
\begin{equation}\label{eqn:intro:higher}
(L\vec{u})_j=\sum_{k=1}^N\sum_{|\alpha|=|\beta|=m}(-1)^{m} \partial^\alpha(A^{j,k}_{\alpha,\beta}\partial^\beta u_k)
\end{equation}
and operators with lower order terms
\begin{align}\label{eqn:intro:lower}
(L\vec{u})_j
&=\sum_{k=1}^N
\Bigl(A^{j,k}_{0,0} u_k
+\sum_{b=1}^d A^{j,k}_{0,b}\partial_{b} u_k
-\sum_{a=1}^d \partial_{a}(A^{j,k}_{a,0} u_k)
-\sum_{a=1}^d\sum_{b=1}^d \partial_{a}(A^{j,k}_{a,b}\partial_{b} u_k)\Bigr)
\\\nonumber&=
\sum_{k=1}^N
\sum_{\substack{0\leq|\alpha|\leq 1\\ 0\leq|\beta|\leq 1}}
(-1)^{|\alpha|}\partial^\alpha(A^{j,k}_{\alpha,\beta}\partial^\beta u_k)
\end{align}
where $\alpha$ and $\beta$ denote multiindices.

Operators of higher order \eqref{eqn:intro:higher} with variable coefficients $A_{\alpha,\beta}^{j,k}$ have been investigated in many recent papers, including \cite{%AusHMT01,BluK04,Aus04,Agr07,
MazMS10,CaoMY16,CaoMY17,Tol18,NiuZ18,NiuX19,Zat20,XuN21}, and the first author's papers with Hofmann and Mayboroda \cite{Bar16,Bar17,BarHM17,BarHM19A,BarHM19B,BarHM18,BarHM20,Bar19p,Bar20p}. (The theory of higher order operators with constant coefficients is older and more developed; we refer the interested reader to the references in the above papers or to the survey paper \cite{BarM16B} for more details.) Harmonic analysis of second order operators with general lower order terms \eqref{eqn:intro:lower} has been done in a number of recent papers, including \cite{%
CaoMY17,%
DavHM18,MayP19,KimS19,% Fundamental
Sak19,%Wan20,% BVP (Wang for constant coefficients)
Dav20,DavW20,% Landis's conjecture
Bai21,BaiMR21,% Riesz transforms
Sak21,% Fundamental
Mou19p,BorHLMP20p,DavI22p%
}.

% Just Schrodinger: \cite{YanCYF18}
% Biharmonic: \cite{Zhu20}
% Classic: \cite{AusE20}

In this paper we will combine the two approaches and investigate operators $L$ of order $2m\geq 2$ with certain lower order terms
\begin{equation}\label{eqn:L:intro}
(L\vec{u})_j=\sum_{k=1}^N\sum_{\substack{\mathfrak{a}\leq |\alpha|\leq m
\\
\mathfrak{b}\leq|\beta|\leq m}} (-1)^{|\alpha|}\partial^\alpha(A^{j,k}_{\alpha,\beta}\partial^\beta u_k).\end{equation}
Specifically, three of the foundational results of the theory of elliptic operators of the form~\eqref{eqn:intro:classic}, which have all received considerable study in the cases of operators of the forms~\eqref{eqn:intro:higher} and~\eqref{eqn:intro:lower}, are Caccioppoli's inequality, Meyers's reverse H\"older inequality for gradients, and the fundamental solution. In this paper we investigate these three topics in the case of operators of the form~\eqref{eqn:L:intro} under certain assumptions on the coefficients.

For operators \eqref{eqn:intro:classic} or \eqref{eqn:intro:higher} without lower order terms, it is usual to require that all coefficients be bounded. Applying H\"older's inequality yields the bound
\[|\langle L\vec u,\vec \varphi\rangle|
=
\biggl|\sum_{j,k=1}^N\sum_{\substack{|\alpha|=m\\|\beta|=m}}\int_{\RR^d}{\partial^\alpha\varphi_j} \,\overline{A_{\alpha,\beta}^{j,k}\,\partial^\beta u_k}
\biggr|
\leq \|\mat A\|_{L^\infty(\R^d)} \|\nabla^m\vec{\varphi}\|_{L^{p'}(\RR^d)} \|\nabla^m\vec{u}\|_{L^{p}(\RR^d)}
\]
for any $1\leq p\leq \infty$.
Thus, under these assumptions $L$ is a bounded linear operator from the Sobolev space $\dot W^{m,p}(\R^d)$ (with norm $\|\vec u\|_{\dot W^{m,p}(\R^d)}=\|\nabla^m\vec  u\|_{L^p(\R^d)}$) to the dual space $\dot W^{-m,p}(\R^d)=(\dot W^{m,p'}(\R^d))^*$ for any $1\leq p\leq \infty$. This is a useful property we would like to preserve.

Observe that elements of $\dot W^{m,p}(\R^d)$ are, strictly speaking, equivalence classes of functions with the same $m$th order gradient. Their lower order derivatives may differ by polynomials. In investigating operators with lower order terms \eqref{eqn:intro:lower} and~\eqref{eqn:L:intro}, the spaces $\dot W^{m,p}(\R^d)$ are not satisfactory; we will need the lower order derivatives of functions in the domain of $L$ to be well defined.

The Gagliardo-Nirenberg-Sobolev inequality gives a natural normalization condition on $\dot W^{1,p}(\R^d)$ if $p<d$. Specifically, if $p<d$ then every element (equivalence class of functions) in $\dot W^{1,p}(\R^d)$ contains a representative that lies in a Lebesgue space $L^{p^*}(\R^d)$ for a certain $p^*$ with $p< p^*<\infty$. This representative is unique as a $L^{p^*}$ function (that is, up to sets of measure zero).

An induction argument shows that, if $u\in \dot W^{m,p}(\R^d)$, then there is a representative of $u$ such that $\partial^\alpha u$ lies in a Lebesgue space for all $\alpha$ with $m-d/p<|\alpha|\leq m$. This representative is unique (as a locally integrable function) up to adding polynomials of degree at most $m-d/p$. (Specifically, $\partial^\alpha u\in L^{p_{m,d,\alpha}}(\R^d)$, where $p_{m,d,\alpha}$ is given by formula~\eqref{pk} below.)

We define the $Y^{m,p}(\R^d)$ norm by
\begin{equation*}
\|u\|_{Y^{m,p}(\R^d)}:=\sum_{m-d/p<|\alpha|\leq m}\|\partial^\alpha u\|_{L^{p_{m,d,\alpha}}(\R^d)}.\end{equation*}
$Y^{m,p}(\R^d)$ is thus a space of equivalence classes of functions up to adding polynomials of degree at most $m-d/p$. The Gagliardo-Nirenberg-Sobolev inequality gives a natural isomorphism between $Y^{m,p}(\R^d)$ and the space $\dot W^{m,p}(\R^d)$.

We will consider operators that satisfy, for all suitable test functions $\vec\psi$ and~$\vec\varphi$, the Gårding inequality (or ellipticity or coercivity condition)
\begin{equation}\label{gi}
\text{Re}\sum_{j,k=1}^N
\sum_{\substack{\mathfrak{a}\leq |\alpha|\leq m
\\
\mathfrak{b}\leq|\beta|\leq m}} \int_{\RR^d}{\partial^\alpha\varphi_j} \,\overline{A_{\alpha,\beta}^{j,k}\,\partial^\beta\varphi_k}
\geq \lambda\|\vec{\varphi}\|^2_{Y^{m,2}(\RR^d)}
\end{equation}
and the bound
\begin{equation}\label{eqn:elliptic:bound}
\int_{\RR^d}
\biggl|
\sum_{j,k=1}^N
\sum_{\substack{\mathfrak{a}\leq |\alpha|\leq m
\\
\mathfrak{b}\leq|\beta|\leq m}}
{\partial^\alpha\varphi_j} \,\overline{A_{\alpha,\beta}^{j,k}\,\partial^\beta\psi_k}
\biggr|
\leq \Lambda(p)\,
\|\vec{\varphi}\|_{Y^{m,p'}(\R^d)} \,\|\vec{\psi}\|_{Y^{m,p}(\R^d)}
\end{equation}
for a range of $p$ near~$2$.

(In Section~\ref{CacSec}, following \cite{AusQ00}, we will consider operators satisfying a slightly weaker form~\eqref{wgi} of the Gårding inequality~\eqref{gi}.)

Note that if $d=2$ and $p\geq 2$, then $m-d/p\geq m-1$ and so $\|u\|_{Y^{m,p}(\R^d)}=\|u\|_{\dot W^{m,p}(\R^d)}$. In this case the Gagliardo-Nirenberg-Sobolev inequality provides no normalization and so the bound~\eqref{eqn:elliptic:bound}, for $p=2$, can only be expected to hold if $\mathfrak{a}=\mathfrak{b}=m$. Thus, in dimension~$2$, the results of the present paper do not represent a generalization of previous results such as \cite{Bar16}. We will include the case $d=2$ in our results, but only for the sake of completeness and ease of reference.

There are many possible conditions that can be imposed on the coefficients~$A_{\alpha,\beta}^{j,k}$ that yield the bound~\eqref{eqn:elliptic:bound}.
Following (or modifying) \cite{%
%CaoMY17,% bounded
DavHM18,% In Lp for bigger p
%MayP19, % Reverse holder madness
KimS19,% Fundamental
Sak19,%
%Wan20,% BVP (Wang for bounded coefficients)
%Dav20,
DavW20,% Landis's conjecture
%Bai21,
BaiMR21,% Riesz transforms
Sak21,% Fundamental
BorHLMP20p%
}, we will focus our attention on the case
\begin{equation}\label{eqn:intro:bound}
\max_{\substack{\mathfrak{a}\leq |\alpha|\leq m\\\mathfrak{b}\leq |\beta|\leq m}}
\|A^{j,k}_{\alpha,\beta}\|_{ L^{2_{\alpha,\beta}}(\R^d)} \leq \Lambda,
\qquad \mathfrak{a},\mathfrak{b}>m-\frac{d}{2},
\qquad
2_{\alpha,\beta}=\frac{d}{2m-|\alpha|-|\beta|}.
\end{equation}
For all $p$ in a certain range including~$2$, the bound~\eqref{eqn:elliptic:bound} follows immediately from the bound~\eqref{eqn:intro:bound}, H\"older's inequality, and the Gagliardo-Nirenberg-Sobolev inequality.
See Lemma~\ref{lem:p:+} for further discussion. Note that if $2m=2$ and $d\geq 3$, the condition $\mathfrak{a}$, $\mathfrak{b}>m-\frac{d}{2}$ holds for $\mathfrak{a}=\mathfrak{b}=0$ and so we may ignore this condition.

We will also consider coefficients satisfying Bochner norm estimates
\begin{equation}
\label{eqn:intro:bound:Bochner}
\max_{\substack{\mathfrak{a}\leq |\alpha|\leq m\\\mathfrak{b}\leq |\beta|\leq m}}
\|A^{j,k}_{\alpha,\beta}\|_{L_t^\infty L_x^{\widetilde 2_{\alpha,\beta}}(\R^d)}\leq \Lambda,
\quad \mathfrak{a},\mathfrak{b}>m-\frac{d-1}{2},
\quad
\widetilde 2_{\alpha,\beta}=\frac{d-1}{2m-|\alpha|-|\beta|}.
\end{equation}
Again, for second order operators ($2m=2$), if $d\geq 4$ then we may take $\mathfrak{a}=\mathfrak{b}=0$.
For example, this includes the case where coefficients are constant in a specified direction, that is, where $A^{j,k}_{\alpha,\beta}(x,t)=a^{j,k}_{\alpha,\beta}(x)$ for all $x\in\R^{d-1}$, $t\in\R$, and some function $a^{j,k}_{\alpha,\beta}\in L^{\widetilde 2_{\alpha,\beta}}(\R^{d-1})$. This is the case studied in \cite{BorHLMP20p}. Operators of the form \eqref{eqn:intro:classic} and~\eqref{eqn:intro:lower} that satisfy $A^{j,k}_{\alpha,\beta}(x,t)=a^{j,k}_{\alpha,\beta}(x)$ (for $|\alpha|=|\beta|=m$) have been studied in the higher order case in \cite{BarHM17,BarHM19A,BarHM19B,BarHM18,BarHM20,Bar19p,Bar20p}, and in the second order case in many papers, including but not limited to \cite{JerK81A,%FabJK84,
KenP93,KenKPT00,Rul07,AusAH08,KenR09, AusAM10A,Axe10,
AlfAAHK11,Ros13,%Bar13,
AusM14, HofKMP15A,HofMitMor15,HofKMP15B,AusS16,  BarM16A, MaeM17,AmeA18,AusM19,AusE20,HofZ21}.
Nontrivial coefficients constant in a specified direction cannot lie in $L^p(\R^d)$ for any $p<\infty$, but can easily lie in Bochner spaces.

Like the condition~\eqref{eqn:intro:bound}, the condition and~\eqref{eqn:intro:bound:Bochner} implies the bound~\eqref{eqn:elliptic:bound} for a range of~$p$ including~$2$; see Lemma~\ref{lem:p:+} below.

We note that the conditions \eqref{eqn:intro:bound} and~\eqref{eqn:intro:bound:Bochner} differ from those of \cite{CaoMY17,Wan20}, in which the authors investigate the system \eqref{eqn:intro:lower} or~\eqref{eqn:L:intro} for coefficients $A_{\alpha,\beta}^{j,k}\in L^\infty(\R^d)$ for all $\alpha$ and $\beta$. (Our conditions imply $A_{\alpha,\beta}^{j,k}\in L^\infty(\R^d)$ only for $|\alpha|=|\beta|=m$.)

\subsection{The Caccioppoli inequality and Meyers's reverse H\"older inequality}

The Caccioppoli inequality (established in the early twentieth century)
is valid for all operators $L$ of the form~\eqref{eqn:intro:classic} where the coefficients $A^{j,k}_{a,b}$ are bounded and satisfy the Gårding inequality~\eqref{gi}, and is often written as
\begin{equation*}\int_{B(X_0,r)} |\nabla\vec u|^2\leq \frac{C}{r^2}\int_{B(X_0,2r)} |\vec u|^2
\quad\text{whenever }L\vec u=0\text{ in }B(X_0,2r)
.\end{equation*}
It can be generalized to the case $L\vec u\neq 0$ by adding an appropriate term on the right hand side; a very general form is
\begin{equation*}\int_{B(X_0,r)} |\nabla\vec u|^2\leq \frac{C}{r^2}\int_{B(X_0,2r)} |\vec u|^2
+C\|L\vec u\|_{\dot W^{-1,2}(B(X_0,2r))}
\end{equation*}
where $\dot W^{-1,2}(B(X_0,2r))$ is the dual space to $\dot W^{1,2}_0(B(X_0,2r))$, the closure in $\dot W^{1,2}(B(X_0,2r))$ of the set of smooth functions compactly supported in $B(X_0,2r)$. By the Poincar\'e or Gagliardo-Nirenberg-Sobolev inequality, this is equal (with equivalence of norms) to the closure in $ W^{1,2}_0(B(X_0,2r))$ or $ Y^{1,2}_0(B(X_0,2r))$.

\begin{remark}It is common to formulate the Caccioppoli inequality (and Meyers's reverse H\"older inequality below) for solutions to $L\vec u=\vec f-\Div \arr F$ (that is, $(L\vec u)_j=f_j-\sum_{a=1}^d \partial_a F_{a,j}$). This is equivalent to our formulation in terms of operator norms of $L\vec u$ if appropriate norms on $\vec f$, $\vec F$ are used.

Specifically, if $L\vec u=-\Div \arr F$, then $|\langle L\vec u,\vec \varphi\rangle| = |\langle \arr F, \nabla \vec\varphi\rangle|$ for all test functions $\vec\varphi\in \dot W^{1,2}_0(B(X_0,2r))$, and so by H\"older's inequality, $\|L\vec u\|_{\dot W^{-1,2}(B(X_0,2r))}\leq \|\arr F\|_{L^2(B(X_0,2r))}$. By the Gagliardo-Nirenberg-Sobolev inequality, if $d\geq 3$ and $p=2d/(d-2)$ then $\|\vec \varphi\|_{L^p(B(X_0,2r))} \leq C\|\nabla\vec \varphi\|_{L^2(B(X_0,2r))}$ for all $\vec\varphi\in \dot W^{-1,2}_0(B(X_0,2r))$, and so if $L\vec u=\vec f$ then $\|L\vec u\|_{\dot W^{-1,2}(B(X_0,2r))} \leq C \|\vec f\|_{L^{p'}(B(X_0,2r))}$.

Conversely, if $L\vec u\in \dot W^{-1,2}(B(X_0,2r))$, then by the Hahn-Banach theorem there is some $\arr F\in L^2(B(X_0,2r))$ with $\|\arr F\|_{L^2(B(X_0,2r))}\approx \|L\vec u\|_{\dot W^{-1,2}(B(X_0,2r))}$ such that $L\vec u=\Div \arr F$.
\end{remark}

\begin{remark} In the case of equations ($N=1$) with real-valued coefficients, a Caccioppoli inequality can also be established for subsolutions; that is, instead of a norm $\|Lu\|$ appearing on the right hand side, it is required that $L u\geq 0$ in $B(X_0,2r)$. See, for example, \cite[Section~3]{Mou19p}. This approach is not available in the case of systems or complex coefficients, and has received little study in the case of higher order equations. \end{remark}

The Caccioppoli inequality has been generalized to operators of the form~\eqref{eqn:intro:higher} (higher order equations without lower order terms) in \cite{Cam80} and with some refinements in \cite{AusQ00,Bar16}. It has been extended to operators of the form~\eqref{eqn:intro:lower} (second order operators with lower order terms) in \cite{DavHM18} (see also \cite{BorHLMP20p}).
In the case of higher order operators with lower order terms of the form~\eqref{eqn:intro:higher}, a parabolic Caccioppoli inequality was established in \cite{CaoMY17} under the assumption that all coefficients (including the lower order coefficients) are bounded; this is different from the assumptions of this paper.

In \cite{Mey63}, Meyers established a reverse H\"older estimate. Specifically, he established that for equations ($N=1$) with bounded and elliptic coefficients, for all $p$ and $q$ sufficiently close to~$2$ (and, in particular, for some $p>2$ and $q\leq 2$) we have the estimate
\begin{equation*}
\biggl(\int_{B(X_0,r)} |\nabla u|^p\biggr)^{1/p}
\leq
Cr^{d/p-d/q}\biggl(\int_{B(X_0,r)} |\nabla u|^q\biggr)^{1/q}
+C\|L u\|_{\dot W^{-1,p}(B(X_0,r))}
.\end{equation*}
The exponent $q$ on the right hand side can be lowered if desired; see \cite[Section~9, Lemma~2]{FefS72} in the case of harmonic functions, and \cite[Lemma~33]{Bar16} for more general functions. Meyers's results have been generalized to second order systems (even nonlinear systems) without lower order terms (see \cite[Chapter~V]{Gia83}), and to higher order equations without lower order terms (see \cite{Cam80,AusQ00,Bar16}).

Caccioppoli's inequality is still valid for systems of the form~\eqref{eqn:L:intro}, that is, higher order equations with lower order terms. The argument is largely that of \cite{Cam80,Bar16} and is presented in Section~\ref{CacSec}.

The obvious generalization of Meyers's reverse H\"older inequality is \emph{not} valid in the case of operators (even second order operators) with lower order terms. That is, for any given positive integers $m$, $d$ and nonnegative integers $\mathfrak{a}\in (m-d/2,m]$, $\mathfrak{b}\in (m-d/2,m)$, there exists an operator $L$ of the form
\begin{equation*}
L{u}=\sum_{\substack{\mathfrak{a}\leq|\alpha|\leq m
\\
\mathfrak{b}\leq|\beta|\leq m}} (-1)^{|\alpha|}\partial^\alpha(A_{\alpha,\beta}\partial^\beta u)
\end{equation*}
with coefficients satisfying the conditions~\eqref{gi} and~\eqref{eqn:intro:bound},
and a function $u:Q_0\to\R$ with $L u=0$ in $Q_0$, such that for any $p>2$ and any natural number~$k$, there is a ball $B(X_k,r_k)$ with $B(X_k,2r_k)\subset Q_0$ and with
\begin{equation*}
\biggl(\int_{B(X_k,r_k)} |\nabla^m\vec u|^p\biggr)^{1/p}
\geq
kr_k^{d/p-d/2}\biggl(\int_{B(X_k,2r_k)} |\nabla^m\vec u|^2\biggr)^{1/2}
\end{equation*}
and, indeed, the stronger bound
\begin{equation}\label{eqn:intro:counter}
\biggl(\int_{B(X_k,r_k)} |\nabla^m\vec u|^p\biggr)^{1/p}
\geq
k\sum_{i=\mathfrak{b}+1}^m r_k^{d/p-d/2-(m-i)}\biggl(\int_{B(X_k,2r_k)} |\nabla^i\vec u|^2\biggr)^{1/2}
.\end{equation}
See Section~\ref{sec:meyers:counter}.

Weaker generalizations have been investigated in \cite{BorHLMP20p} and the argument of Section~\ref{Lpc} takes many ideas therefrom. The following theorem is the first main result of this paper. It will be proven in Sections~\ref{CacSec} (the case $p=q=\mu=2$) and Section~\ref{Lpc} (the general case), and represents a simultaneous statement of the Caccioppoli and Meyers inequalities for systems of the form~\eqref{eqn:L:intro}.
\begin{theorem}\label{umpm:intro}
Let $m\geq 1$ and $d\geq 2$ be integers. Let $L$ be an operator of the form~\eqref{eqn:L:intro} for some coefficients $\mat{A}$ that satisfy the ellipticity condition~\eqref{gi} and one of the bounds \eqref{eqn:intro:bound} or~\eqref{eqn:intro:bound:Bochner}.

Then there is a $\delta>0$ depending on $m$, $d$ and the constants $\lambda$ and $\Lambda$ in the bounds~\eqref{gi} and \eqref{eqn:intro:bound} or~\eqref{eqn:intro:bound:Bochner} with the following significance.

Let $p\in [2,2+\delta)$, $\mu\in (2-\delta,2+\delta)$, and let $0<q\leq\infty$. Let $j$ and $\varpi$ be integers with $0\leq j\leq m$ and
$0\leq \varpi\leq \min(j,\mathfrak{b})$. If $p=2$, we impose the additional requirement that either $q\geq 2$ or $\varpi\geq 1$ (and thus $j\geq 1$).

Let $Q\subset\RR^d$ be a cube with sides parallel to the coordinate axes. Let $\vec u\in Y^{m,\mu}(\theta Q)$ be such that $\|L\vec u\|_{\dot W^{-m,p}(\theta Q)}<\infty$.

Then $\nabla^j u\in L^p(Q)$, and there exist positive constants $\kappa$ and $C$ depending on $p$, $q$, $m$, $d$, $\lambda$, and~$\Lambda$ such that
\begin{multline*}
\frac{1}{|Q|^{(m-j)/d}}
\|\nabla^j u\|_{L^{p}(Q)}
\\\leq
\frac{C}{(\theta-1)^\kappa}\|L\vec u\|_{Y^{-m,p}(\theta Q)}
+ \frac{C|Q|^{1/p-1/q-(m-\varpi)/d}}{(\theta-1)^\kappa} \|\nabla^\varpi\vec u\|_{L^q(\theta Q\setminus Q)}
\end{multline*}
for all $1<\theta\leq 2$.
\end{theorem}

Here $\theta Q$ is the cube concentric to~$Q$ with volume $|\theta Q|=\theta^d|Q|$.
Note that the condition $\vec u\in Y^{m,\mu}(\theta Q)$ is stronger than the condition $\nabla^\varpi\vec u\in L^q(\theta Q\setminus Q)$, that is, that the right hand side of the given bound be finite.
The assumption $\nabla^m\vec u\in L^\mu(\theta Q)$ implies that $L\vec u$ is a bounded linear operator on $\dot W^{m,\mu'}_0(\theta Q)=\{\vec\psi:
\nabla^m\vec\psi\in L^{\mu'}(\R^d),
\>\vec\psi=0\text{ in }\R^d\setminus\theta Q\}$; we require $L\vec u$ to be a bounded linear operator on $\dot W^{m,p'}_0(\theta Q)$ (or, more precisely, on the space $\dot W^{m,\mu'}_0(\theta Q)\cap \dot W^{m,p'}_0(\theta Q)$ equipped with the $\dot W^{m,p'}$-norm).

If $L\vec u\in \dot W^{-m,p}(\theta Q)$ for some $p<2$ but sufficiently close to~$2$, a weaker result is still available; see Theorem~\ref{umpm:less} below.

\subsection{The fundamental solution}

The fundamental solution $\vec E^L_{X,j}$ for the operator $L$ is, formally, the solution to $L\vec E^L_{X,j}=\delta_X \vec e_j$, where $\delta_X$ denotes the Dirac mass at~$X$. The fundamental solution has proven to be a very useful tool in the theory of differential equations without lower order terms (that is, of the forms \eqref{eqn:intro:classic} and~\eqref{eqn:intro:higher}). By definition, integrating against the fundamental solution allows one to solve the Poisson problem $L\vec u=\vec f$ in~$\R^d$. The fundamental solution is also used in the theory of layer potentials, an essential tool in the theory of boundary value problems; for example, layer potentials based on the fundamental solution for certain variable coefficient operators of the form~\eqref{eqn:intro:classic} were used in \cite{KenR09,Rul07,Agr09,AlfAAHK11,MitM11,Bar13,BarM16A,Ros13,AusM14,HofKMP15B,HofMayMou15, HofMitMor15,AusS16,AmeA18, AusM19} and of the form~\eqref{eqn:intro:higher} in \cite{BarHM18,BarHM20,Bar19p,Bar20p}.

Formally, the fundamental solution can be written as $\vec E^L_{X,j}=L^{-1} (\delta_X \vec e_j)$, where $\delta_X \vec e_j$ is the element of a dual space given by $\langle \delta_X \vec e_j, \vec \varphi\rangle=\varphi_j(X)$.
In the case of constant coefficient operators, one can directly solve the equation $\vec E^L_{X,j}=L^{-1} (\delta_X \vec e_j)$ using the Fourier transform. For some well behaved variable coefficients, $L$ is an invertible map from  some function space into a space containing $\delta_X\vec e_j$, and so this approach is still valid. In case~\eqref{eqn:intro:classic} of second order operators without lower order terms, see \cite{LitSW63} ($N=1$ and real symmetric coefficients), \cite{KenN85} ($N=1$, real nonsymmetric coefficients, and $d=2$) or \cite{Fuc86,DolM95} ($N\geq 1$ and continuous coefficients).

This is the approach taken in both \cite{Bar16} and the present paper for general higher order operators of the form~\eqref{eqn:intro:higher} or~\eqref{eqn:L:intro}. By the assumptions \eqref{gi} and~\eqref{eqn:elliptic:bound} and the Lax-Milgram lemma, $L$ is invertible $Y^{m,2}(\R^d)\to Y^{-m,2}(\R^d)$.
If $2m>d$, then by Morrey's inequality, all representatives of elements of $Y^{m,2}(\R^d)$ are H\"older continuous. Recall that elements of $Y^{m,2}(\R^d)$ are equivalence classes of functions up to adding polynomials of degree at most $m-d/2$. If a suitable (although somewhat artificial) normalization condition is applied, then $\delta_X \vec e_j$ is a well defined and bounded linear operator on $Y^{m,2}(\R^d)$, that is, an element of $Y^{-m,2}(\R^d)$. We therefore may construct $\vec E^L_{X,j}$ as $\vec E^L_{X,j}=L^{-1} (\delta_X \vec e_j)$ if $2m>d$. If $2m\leq d$, then the above argument yields a fundamental solution for the operator $\widetilde L=(-\Delta)^ML(-\Delta)^M$ of order $4M+2m$ if $M$ is large enough; the fundamental solution for $L$ may be then derived from that for~$\widetilde L$.

This approach, with some attention to the details and use of the Caccioppoli and Meyers inequalities, yields the following theorem. This theorem is the second main result of the present paper.

\begin{theorem}\label{thm:intro:fundamental}
Let $L$ be an operator of order $2m$ of the form~\eqref{eqn:L:intro} that satisfies the ellipticity condition~\eqref{gi} and one of the bounds \eqref{eqn:intro:bound} or~\eqref{eqn:intro:bound:Bochner}.

Then there exists a number $\delta>0$ and an array of functions $E^L_{j,k}$ for pairs of integers $j$, $k$ in $[1,N]$ and defined on $\R^d\times\R^d$ with the following properties. This array of functions is unique up to adding functions $ P_{j,k}$ defined on $\R^d\times\R^d$ that satisfy $\partial^\zeta_X\partial^\xi_Y P_{j,k}(Y,X)=0$ whenever $m-d/2\leq |\zeta|\leq m$, $m-d/2\leq |\xi|\leq m$, and $(|\zeta|,|\xi|)\neq (m-d/2,m-d/2)$.

Suppose that $\alpha$ and $\beta$ are two multiindices with
$m-d/2\leq |\alpha|\leq m$, $m-d/2\leq |\beta|\leq m$, and $(|\alpha|,|\beta|)\neq (m-d/2,m-d/2)$.

Suppose further that $Q$ and $\Gamma$ are two cubes in $\R^d$ with $|Q|=|\Gamma|$ and $\Gamma\subset 8Q\setminus 4Q$. Then the partial derivative $\partial^\alpha_X \partial^\beta_Y E^{L}_{j,k}(Y,X)$ exists as a $L^2(Q\times\Gamma)$ function and satisfies the bounds
\begin{align}
\int_{Q}\int_{\Gamma}|\partial^\alpha_X \partial^\beta_Y E^L_{j,k}(Y,X)|^2\,dX\,dY \leq C|Q|^{(4m-2|\alpha|-2|\beta|)/d}
.
\end{align}
If $2-\delta<{{p}}<2+\delta$, and if ${{p}}<2$ or $|\beta|>m-d/2$, then
\begin{align}
\int_\Gamma \biggl(\int_{Q} |\partial_X^\alpha \partial_Y^\beta  E^L_{j,k}(Y,X)|^{{{p}}_\beta}\,dY\biggr)^{2/{{p}}_\beta}\,dX
\leq C|Q|^{2m/d-1+2/{{p}}-2|\alpha|}
\end{align}
where $\frac{1}{{{p}}_\beta}=\frac{1}{{{p}}}-\frac{m-|\beta|}{d}$.

Furthermore, we have the symmetry property
\begin{equation}\partial^\alpha_X \partial^\beta_Y E^L_{j,k}(Y,X)=\overline{\partial^\alpha_X \partial^\beta_Y E^{L^*}_{k,j}(X,Y)}.
\end{equation}

Finally, suppose that $2-\delta<q<2+\delta$ 
and that $m-d/q<|\xi|\leq m$.
Let $ F\in L^{(q_\xi)'}(\R^d)$ be compactly supported, where $\frac{1}{(q_\xi)'} = 1-\frac{1}{q}+\frac{m-|\xi|}{d}$. Let $1\leq \ell\leq N$.
For each $\beta$ with $m\geq |\beta|>m-d/q'$ and each $1\leq k\leq N$, let
\begin{equation}
(u_\beta)_k(X)
=
\int_{\RR^d} {\partial_X^\beta\partial_Y^\xi  E^{L}_{k,\ell}(X,Y)}\,{ F(Y)}\,dY
.\end{equation}
The integral converges absolutely for almost every $X\in\R^d\setminus \supp F$ for all such $\beta$ and $\xi$; if $|\beta|<m$ or $|\xi|<m$ then the integral converges absolutely for almost every $X\in\R^d$.

Then there is a function $\vec u\in Y^{m,q}(\R^d)$ with $\partial^\beta \vec u=\vec u_\beta$ for all such $\beta$ almost everywhere (if $|\beta|+|\xi|<2m$) or almost everywhere in $\R^d\setminus\supp F$ (otherwise) and such that
\[
\int_{\RR^d} \partial^\xi \varphi_\ell\,{ F}
=
\sum_{k,j=1}^N \sum_{\substack{\mathfrak{a}\leq |\alpha|\leq m\\\mathfrak{b}\leq |\beta|\leq m}}
\int_{\R^d}
\partial^\alpha \varphi_j\,{A_{\alpha,\beta}^{j,k}} \,
{\partial^\beta  u_k}
\]
for all $\vec\varphi\in Y^{m,q'}(\R^d)$.
\end{theorem}

Many assumptions on the coefficients other than \eqref{eqn:intro:bound} and~\eqref{eqn:intro:bound:Bochner} are reasonable. We construct the fundamental solution in Section~\ref{FSS}. In that section, we will not explicitly use the assumptions \eqref{eqn:intro:bound} and~\eqref{eqn:intro:bound:Bochner}; instead we will use their consequences, the Caccioppoli and Meyers inequalities, for the operator $\widetilde L=\Delta^ML\Delta^M$. The results in Section~\ref{FSS}, and in particular Theorem~\ref{Fslowth}, will allow the interested reader to construct the fundamental solution for other classes of coefficients once a suitable higher order Caccioppoli inequality has been established.

\subsubsection{Other approaches.}

The approach of this paper and of \cite{Bar16} uses higher order operators, and in particular the higher order Caccioppoli and Meyers inequalities, to construct the fundamental solution, and as such has only been available since the development of a strong theory of higher order operators. The fundamental solution for second order operators has been of interest for a long time and other approaches to its construction have been used.

If $d\geq 2$, then $\delta_X\vec e_j$ is not an element of $Y^{-1,2}(\R^d)$. Specifically, elements of $Y^{1,2}(\R^d)$ are elements of Lebesgue spaces (or of $BMO$) and so their value at a single point is not well defined. In some special cases (discussed above), $L$ is invertible from $Y^{1,p}_0(B)$ to $Y^{-1,p}(B)$ for open balls $B$ and $p$ large enough to apply Morrey's inequality, and so the fundamental solution can be constructed using the approach discussed above and some attention to the behavior outside of~$B$. However, this approach is not available in other cases.

In some cases, solutions to $L\vec u=0$ may be locally H\"older continuous even if general $Y^{1,2}$ functions are not. In this case, the fundamental solution may be constructed as a limit of $L^{-1}T_\rho$, where $T_\rho\to\delta_X \vec e_j$ as $\rho\to 0^+$ and each $T_\rho$ is in $Y^{-1,2}(\R^d)$. Careful application of the Caccioppoli inequality, the local H\"older continuity, and other arguments yields that $L^{-1}T_\rho$ converges to a fundamental solution.

This was done in for operators of the form~\eqref{eqn:intro:classic} (\cite{GruW82,HofK07}) and \eqref{eqn:intro:lower} (\cite{DavHM18}) in dimension $d\geq 3$ under the assumption that solutions are locally H\"older continuous. %The approximation process requires as an additional assumption that solutions to $L\vec u=0$ are H\"older continuous; this is an explicit assumption made in \cite{Hof07,DavHM18}. (\cite{GruW82} considers operators with $N=m=1$ and with real coefficients; in this case continuity of solutions to $Lu=0$ was established by De Giorgi and Nash in \cite{DeG57,Nas58}.)
Green's functions in domains (rather than in all of $\R^d$) were constructed using this method in \cite{KanK10,KimS19,Sak21,Mou19p}.

A different approach involving kernels for the heat semigroup $e^{-tL}$ was used in \cite{AusMT98} to construct the fundamental solution in dimension~$2$; as observed in \cite{DonK09} their approach is valid for systems of the form~\eqref{eqn:intro:classic} with $N\geq 1$ and with complex nonsymmetric coefficients. The papers \cite{DonK09,ChoDK12} establishes results analogous to those of \cite{AusMT98} for the Green's function of a domain rather than all of~$\R^2$.

Considerably more work must be expended to apply the semigroup approach in dimension $d\geq 3$; heat semigroups were used in \cite{MayP19} to construct the fundamental solution for the magnetic Schr\"odinger operator, and a different form of semigroup was used in \cite{Ros13} to construct the fundamental solution assuming only local boundedness, not local H\"older continuity.

This approach does require the Di Giorgi-Nash property of elliptic operators, or a condition, such as real coefficients, that implies this property. However, this approach often yields stronger estimates than those of the present paper, and indeed stronger estimates than those true of the fundamental solution for the Laplace operator. See, for example, \cite{She99,MayP19,DavI22p}.

% Ifra and Riahi 2005 (only one kind of lower order term, domains)
% Zhuge Zhang 2016 (domains)

%\cite{BaiMR21} uses \cite{DavHM18}.

\subsection{Outline} The outline of this paper is as follows. In Section~\ref{defs} we will define our terminology.
We will give some results concerning function spaces (in particular, Sobolev spaces) in Section~\ref{sec:function}.

We will prove the Caccioppoli inequality in Section~\ref{CacSec}. We will prove our generalization of Meyers's reverse H\"older inequality in Section~\ref{sec:meyers:proof}, and construct the counterexample of the inequality~\eqref{eqn:intro:counter} in Section~\ref{sec:meyers:counter}.

We will construct the fundamental solution in Section~\ref{FSS}.

Some results concerning invertibility of the operator~$L$ between certain function spaces will be used both in Section~\ref{Lpc} and Section~\ref{FSS}; we will give these results in Section~\ref{FS}.

\section{Definitions}\label{defs}

\subsection{Basic notation}

We consider divergence-form elliptic systems of $N$ partial differential equations of order $2m$ in $d$-dimensional Euclidean space $\RR^d$, $d\geq 2$.

When $\Omega\subset\RR^d$ is a set of finite measure, we let $\fint_\Omega f=\frac{1}{|\Omega|}\int f$, where $|\Omega|$ denotes the Lebesgue measure of~$\Omega$.

As mentioned in Theorem~\ref{umpm:intro}, if $Q$ is a cube in $\RR^d$ or $\RR^{d-1}$ and $\theta>0$ is a positive real number, we let $\theta Q$ denote the concentric cube with $|\theta Q|=\theta^d |Q|$ (so the side length of $\theta Q$ is $\theta$ times the side length of~$Q$).

We employ the use of multiindices in $(\NN_0)^d$. We will define
$$|\gamma|=\sum_{i=1}^d\gamma_i\quad \text{and}\quad\gamma!=\gamma_1!\cdot\gamma_2!\cdots\gamma_d!
$$
for any multiindex $\gamma=(\gamma_1,\dots,\gamma_d)$.
When $\delta$ is another multiindex in $\NN^d$ we say that $\delta\leq\gamma$ if $\delta_i\leq\gamma_i$ for each $1\leq i\leq d$.  Futhermore, we say $\delta<\gamma$ if $\delta_i<\gamma_i$ for at least one such $i$.

We will use the Leibniz Rule for multiindices, that is, that for all suitably differentiable functions $u$ and $v$ and a multiindex $\alpha$, we have that
\begin{equation*}\partial^\alpha(uv)=\sum_{\gamma\leq\alpha}\frac{\alpha!}{\gamma!(\alpha-\gamma)!}\partial^\gamma u\,\partial^{\alpha-\gamma}v.\end{equation*}

\subsection{Function spaces}
\label{sec:Y:dfn}
Let $\Omega\subseteq\RR^d$ be a domain.
We denote by $L^p(\Omega)$ and $L^\infty (\Omega)$ the standard Lebesgue spaces with respect to Lebesgue measure, with norms given by
\begin{equation*}
\|u\|_{L^p(\Omega)}=\bigg(\int_\Omega|u|^p\bigg)^{1/p}
\end{equation*} if $1\leq p<\infty $, and
\begin{equation*}
\|u\|_{L^\infty (\Omega)}=\es_\Omega|u|.
\end{equation*}
If $1\leq p\leq\infty $, we let $p'$ be the extended real number that satisfies $1/p+1/p'=1$.

If $t\in\R$, let $[\Omega]^t=\{x\in\RR^{d-1}:(x,t)\in \Omega\}$.
We define the Bochner norm $L_t^q L^p_{\vphantom{t}x}(\Omega)$ by
\begin{equation}\label{eqn:bochner}
\|u\|_{L_t^q L^p_{\vphantom{t}x}(\Omega)}
=
\biggl(\int_{-\infty }^\infty \biggl(
\int_{[\Omega]^t} |u(x,t)|^p\,dx
\biggr)^{q/p} dt\biggr)^{1/q}
\end{equation}
with a suitable modification in the case $p=\infty $ or $q=\infty $.

We define the inhomogeneous Sobolev norm as
\begin{equation*}
\|\vec u\|_{W^{k,p}(\Omega)} = \sum_{j=0}^k\|\nabla^j\vec{u}\|_{L^p(\Omega)}
\end{equation*}
where derivatives are required to exist in the weak sense.
We then define the homogeneous Sobolev norm as
\begin{equation}\label{dfn:W:norm}
\|\vec u\|_{\dot{W}^{k,p}(\Omega)}=\|\nabla^k\vec{u}\|_{L^p(\Omega)}.
\end{equation}
Observe that by the Poincar\'e inequality, if $\vec u\in \dot{W}^{k,p}(\Omega)$ and $\Omega$ is bounded, then $\nabla^j u\in L^p(\Omega)$ for all $0\leq j<k$; however, the Poincar\'e inequality does not yield finiteness of $\|\nabla^j u\|_{L^p(\Omega)}$ in the case where $\Omega$ is unbounded. 

The Sobolev spaces are then the spaces of equivalence classes of functions whose Sobolev norm is finite, with the equivalence relation $\vec u\sim\vec v$ if $\|\vec u-\vec v\|=0$. Observe that elements of inhomogeneous Sobolev spaces, like elements of Lebesgue spaces, are defined up to sets of measure zero, while elements of homogeneous Sobolev spaces (in connected domains) are defined up to sets of measure zero and also up to adding polynomials of degree at most~$k-1$.

Recall that for $1\leq p<d$, the Sobolev conjugate of $p$ is defined to be
\begin{equation*}
p^*=\frac{dp}{d-p}.
\end{equation*}
See, for example, \cite[Section 5.6]{Eva98}.
Notice that
\begin{equation}\label{pstar}
\frac{1}{p^*}=\frac{1}{p}-\frac{1}{d}.
\end{equation}
We will now generalize equation \eqref{pstar}.  Let $k$ be an integer so that $m-\frac{d}{p}<k\leq m$.  We then define $p_{m,d,k}$ so that
\begin{equation}\label{pk}
\frac{1}{p_{m,d,k}}=\frac{1}{p}-\frac{m-k}{d}.
\end{equation}

When considering elliptic operators of order $2m$ in dimension~$d$, and  the numbers $m$ and $d$ are clear from context, we will let $p_k=p_{m,d,k}$.
If $\alpha$ is a multiindex, we will let $p_\alpha=p_{m,d,\alpha}=p_{m,d,|\alpha|}$.  Notice that when $|\alpha|=m$ we have that $2_\alpha=2$, when $|\alpha|=m-1$ then $2_\alpha=2_*$ and so on.  This definition for $2_\alpha$ will help keep the notation throughout this paper relatively clean and help us to avoid backwards summation.

If $\Omega\subseteq\RR^d$ is a domain, $m\geq 1$ is an integer, and $1\leq p\leq \infty $, we define the $Y^{m,p}(\Omega)$ norm as
\begin{equation}\label{dfn:Y:norm}
\|u\|_{Y^{m,p}(\Omega)}:=\sum_{m-d/p<|\alpha|\leq m}\|\partial^\alpha u\|_{L^{p_{m,d,\alpha}}(\Omega)}.\end{equation}
We then define $Y^{m,p}(\Omega)$ analogously to $\dot W^{m,p}(\Omega)$. Observe that elements of $Y^{m,p}(\Omega)$ are defined up to adding polynomials of degree at most $m-d/p$. We let
\begin{equation*}Y^{m,p}_0(\Omega)=\{\vec\varphi\in Y^{m,p}(\RR^d):\vec\varphi=0\text{ outside }\Omega\}.\end{equation*}
Then $Y^{m,p}_0(\Omega)$ is the space of functions in $Y^{m,p}(\Omega)$ which are zero near the boundary in an appropriate sense. Note that $Y^{m,p}_0(\RR^d)=Y^{m,p}(\RR^d)$. Conversely, if $\RR^d\setminus\Omega$ has nonempty interior, then elements of $Y^{m,p}(\Omega)$ have a natural normalization condition (that is, nonzero polynomials are not representatives of elements of $Y^{m,p}_0(\Omega)$).

We will generally write bounded linear operators on $Y^{m,p}_0(\Omega)$ as $\langle T,\,\cdot\,\rangle_\Omega$; if $\Omega=\RR^d$ we will omit the $\Omega$ subscript. We define the antidual space $Y^{-m,p'}(\Omega)=(Y^{m,p}_0(\Omega))'$, for $1/p+1/p'=1$, by
\begin{equation}\label{eqn:antidual}
\langle T,\,\cdot\,\rangle_\Omega\text{ is a bounded linear operator on $Y^{m,p}_0(\Omega)$ if and only if }T\in Y^{-m,p'}(\Omega).\end{equation}
Note that if $\alpha\in\CC$ then $\langle \alpha T,\vec\Phi\rangle_\Omega=\overline\alpha\langle T,\vec\Phi\rangle_\Omega$.

\subsection{Elliptic operators}\label{sec:elliptic}

Let $m$ be a positive integer.
Let $\mat{A}=(A^{j,k}_{\alpha,\beta})$ be an array of measurable real or complex coefficients defined on $\RR^d$ indexed by integers $j$ and $k$ such that $1\leq j\leq N$ and $1\leq k\leq N$ and multiindices $\alpha$ and $\beta$ with $|\alpha|\leq m$ and $|\beta|\leq m$.

We define the differential operator $L$ with coefficients $\mat{A}$ as follows.
If $\vec u$ is a Sobolev function, we let $\langle L\vec u,\,\cdot\,\rangle_\Omega$ be the linear operator that satisfies
\begin{equation}\label{dfn:L}
\sum_{j,k=1}^N\sum_{|\alpha|\leq m}\sum_{|\beta|\leq m}
\int_{\Omega}\partial^\alpha\varphi_j\,  \overline{A^{j,k}_{\alpha,\beta} \,\partial^\beta u_k}
={\langle L\vec u,\vec\varphi\rangle_\Omega}
\end{equation}
for all appropriate test functions~$\vec\varphi$.

\begin{remark}
If $\mat{A}$, $\vec u$, and $\vec\varphi$ are sufficiently smooth and decay sufficiently rapidly at infinity, we may integrate by parts to see that
\begin{equation*}\langle L\vec u,\vec\varphi\rangle
= \sum_{j=1}^N
\int_{\R^\dmn}\varphi_j \overline{\sum_{k=1}^N\sum_{|\alpha|\leq m}\sum_{|\beta|\leq m} (-1)^{|\alpha|} \partial^\alpha(A^{j,k}_{\alpha,\beta} \,\partial^\beta u_k)}
\end{equation*}
Thus, in this case we may write
\begin{equation*}(L\vec u)_j=\sum_{k=1}^N\sum_{|\alpha|\leq m}\sum_{|\beta|\leq m} (-1)^{|\alpha|} \partial^\alpha(A^{j,k}_{\alpha,\beta} \,\partial^\beta u_k)\end{equation*}
as a classically defined linear differential operator; this coincides with formula~\eqref{dfn:L} if $\langle\,\cdot\,,\,\cdot\,\rangle_\Omega$ denotes the usual (complex) inner product in $L^2(\R^\dmn;\CC^N)$.
\end{remark}

We define
\begin{align}\label{eqn:tildealpha}
\mathfrak{a}&=\mathfrak{a}_L=\min\{|\alpha|:A_{\alpha,\beta}^{j,k}(X)\neq 0\text{ for some }j,k,\beta,X\},
\\\label{eqn:tildebeta}
\mathfrak{b}&=\mathfrak{b}_L=\min\{|\beta|:A_{\alpha,\beta}^{j,k}(X)\neq 0\text{ for some }j,k,\alpha,X\}.
\end{align}

\begin{definition}\label{dfn:ppm}
We let $\Pi_L$ be the largest interval with
\begin{equation*}\Pi_L
\subseteq(1,\infty)\cap\biggl\{p:\frac{m-\mathfrak{b}}{d}<\frac{1}{p} < \frac{d-m+\mathfrak{a}}{d}\biggr\}
\end{equation*}
and such that if $p\in\Pi_L$, then there is a $\Lambda(p)\in[0,\infty)$ such that the bound~\eqref{eqn:elliptic:bound} is valid, that is, such that
\begin{equation}\label{eqn:PiL}
\int_{\RR^d}
\biggl|
\sum_{j,k=1}^N
\sum_{\substack{\mathfrak{a}\leq |\alpha|\leq m
\\
\mathfrak{b}\leq|\beta|\leq m}}
{\partial^\alpha\varphi_j} \,\overline{A_{\alpha,\beta}^{j,k}\,\partial^\beta\psi_k}
\biggl|
\leq \Lambda(p)\,
\|\vec{\varphi}\|_{Y^{m,p'}(\R^d)} \,\|\vec{\psi}\|_{Y^{m,p}(\R^d)}
\end{equation}
for all $\vec\varphi\in Y^{m,p'}(\R^d)$, $\vec \psi\in Y^{m,p}(\R^d)$.
\end{definition}

We consider singleton sets to be intervals, so $\{2\}=[2,2]$ is a possible value of~$\Pi_L$.
We will usually assume that $2\in\Pi_L$; in particular, this implies that $\mathfrak{a}$, $\mathfrak{b}>m-d/2$.

\begin{remark}\label{rmk:L:bounded}
If $p\in\Pi_L $, then $|\langle L\vec u,\vec\varphi\rangle|\leq \Lambda(p) \|\vec\varphi\|_{Y^{m,p'}(\R^\dmn)} \|\vec u\|_{Y^{m,p}(\R^\dmn)}$ and the integral in the definition of $\langle L\vec u,\vec\varphi\rangle$ converges absolutely for such $\vec u$ and $\vec\varphi$; thus, if $\vec u\in Y^{m,p}(\R^\dmn)$ then the given integral is a linear operator on $Y^{m,p'}_0(\R^\dmn)$, and so $L\vec u\in Y^{-m,p}(\R^\dmn)$. Our conventions for $Y^{-m,p}$ yield that $L$ is a bounded linear operator (and not a conjugate linear operator) from $Y^{m,p}(\R^\dmn)$ to $Y^{-m,p}(\R^\dmn)$.
\end{remark}

\begin{remark}
The condition $d/(d+\mathfrak{a}-m)<p<d/(m-\mathfrak{b})$ ensures that the derivatives $\partial^\alpha\vec\varphi$, $\partial^\beta\vec\psi$ appearing in the bound~\eqref{eqn:PiL} satisfy $|\alpha|>m-d/p'$ and $|\beta|>m-d/p$. By the definition~\eqref{dfn:Y:norm} of $Y^{m,p}(\R^d)$, this means that $\partial^\alpha\vec\varphi\in L^{p'_\alpha}(\R^d)$, $\partial^\beta\vec\psi\in L^{p_\beta}(\R^d)$. Derivatives of $Y^{m,p}(\RR^d)$ or $Y^{m,p'}(\RR^d)$ functions of lower order are defined only up to adding constants or polynomials, which would preclude validity of the bound~\eqref{eqn:PiL}. It might be possible to consider the case $\mathfrak{a}\leq m-d/2$ or $\mathfrak{b}\leq m-d/2$ by considering more delicate cancellation conditions or Hilbert spaces other than $Y^{m,2}(\RR^d)$, but such constructions are beyond the scope of this paper.

As noted in the introduction, if $m=1$ and $d\geq 3$, then the condition $\mathfrak{a}$, $\mathfrak{b}>m-d/2$  is vacuous, as $m-d/2<0$ and so there are no multiindices $\alpha\in (\NN_0)^d$ with $|\alpha|\leq m-d/2$.
Conversely, if $d=2$,
then $A_{\alpha\beta}\neq 0$ only in the case when $|\alpha|=|\beta|=m$, and so the present paper does not represent a generalization of previous results such as~\cite{Cam80,AusQ00,DavHM18,Bar16}. 
\end{remark}

We will consider coefficients which satisfy the G\aa{}rding inequality~\eqref{gi}.
In \cite{AusQ00}, Auscher and Qafsaoui consider higher order elliptic systems in divergence form in which ellipticity is in the sense of the following weaker G\aa{}rding inequality
\begin{equation}\label{wgi}
\text{Re}\sum_{j,k=1}^N\sum_{|\alpha|\leq m}\sum_{|\beta|\leq m}\int_{\RR^d}\overline{\partial^\alpha\varphi_j} \,A_{\alpha,\beta}^{j,k}\,\partial^\beta\varphi_k
\geq\lambda\|\nabla^m\vec{\varphi}\|^2_{L^2(\RR^d)}-\delta\|\vec{\varphi}\|^2_{L^2(\RR^d)}
\end{equation}
where $\lambda>0$ and $\delta\geq 0$ are real numbers, for all $\vec{\varphi}$ which are smooth and compactly supported in $\RR^d$. The standard Gårding inequality \eqref{gi} is thus the weak inequality~\eqref{wgi} with $\delta=0$. In Section~\ref{CacSec}, we will prove results in the generality of the bound~\eqref{wgi} instead of~\eqref{gi}.

Throughout we will let $C$ denote a positive constant whose value may change from line to line, but that depends only on the dimension~$d$, the order $2m$ of our differential operators, the size~$N$ of our system of equations, the constant $\lambda$ in the bound~\eqref{gi} (or~\eqref{wgi}), and the constant $\Lambda(2)$ in the bound~\eqref{eqn:elliptic:bound}. A constant depending on a number $p\in \Pi_L$ may also depend on $\Lambda(p)$.

A standard argument involving the Lax-Milgram lemma (see Lemma~\ref{lem:L:invertible} below) shows that if $L$ satisfies the condition~\eqref{gi} and $2\in\Pi_L$, then $L$ is not only bounded but invertible $Y^{m,2}(\R^d)\to Y^{-m,2}(\R^d)$.
\begin{definition}\label{dfn:compatible}
If $L:Y^{m,2}(\R^d)\to Y^{-m,2}(\R^d)$ is bounded and invertible, then we define
\begin{equation}\label{eqn:SL}\Upsilon_L=\{p:L\text{ is bounded and compatibly invertible } Y^{m,p}(\RR^d)\rightarrow Y^{-m,p}(\RR^d)\}.\end{equation}
By compatibly invertible, we mean that $L:Y^{m,p}(\RR^d)\to Y^{-m,p}(\RR^d)$ is invertible with bounded inverse and that if $T\in Y^{-m,p}(\RR^d)\cap Y^{-m,2}(\RR^d)$ then $L^{-1}T\in Y^{m,p}(\RR^d)\cap Y^{m,2}(\RR^d)$. (Thus, $L^{-1}T$ has the same value whether we regard $L$ as an operator on $Y^{m,2}(\RR^d)$ or $Y^{m,p}(\RR^d)$.)
\end{definition}
Compatibility is not automatically true; see \cite{Axe10} for an example of operators which are invertible, but not compatibly invertible, in some sense.

We will conclude this section by reminding the reader that our main focus is on coefficients that satisfy the bound~\eqref{eqn:intro:bound}, that is,
\begin{equation*}
\begin{cases}
\|A^{j,k}_{\alpha,\beta}\|_{L^{2_{\alpha,\beta}}(\R^d)}
\leq \Lambda
&\text{if } m\geq|\alpha|>m-\frac{d}{2}\text{ and } m\geq|\beta|>m-\frac{d}{2},
\\
\quad A^{j,k}_{\alpha,\beta}=0 &\text{otherwise,}\end{cases}
\end{equation*}
or the bound~\eqref{eqn:intro:bound:Bochner}, that is,
\begin{equation*}
\begin{cases}
\|A_{\alpha,\beta}^{j,k}\|_{L_t^\infty L^{\widetilde 2_{\alpha,\beta}}_{x}(\R^d)}
\leq \Lambda
&\text{if }m\geq|\alpha|>m-\frac{d-1}{2}\text{ and }m\geq|\beta|>m-\frac{d-1}{2},
\\
\quad A^{j,k}_{\alpha,\beta}=0 &\text{otherwise}.\end{cases}
\end{equation*}
where
\[
2_{\alpha,\beta}=\frac{d}{2m-|\alpha|-|\beta|},
\qquad
\widetilde 2_{\alpha,\beta}=\frac{d-1}{2m-|\alpha|-|\beta|}.\]

Elementary computations involving H\"older's inequality (see Lemma~\ref{lem:p:+}) shows that the conditions \eqref{eqn:intro:bound} and~\eqref{eqn:intro:bound:Bochner} both imply that $\Pi_L$ contains an interval around~$2$ whose radius depends only on the dimension~$d$.

\section{The Gagliardo-Nirenberg-Sobolev and Poincar\'e inequalities and their consequences}
\label{sec:function}

In this section we will collect some results regarding Sobolev functions that will be useful throughout the paper. These results are mainly consequences of the Gagliardo-Nirenberg-Sobolev inequality and induction arguments.

We will begin with Section~\ref{sec:W:Y}, in which we will consider the global function spaces $\dot W^{m,p}(\R^d)$ and $Y^{m,p}(\R^d)$.
In Section~\ref{sec:evlem} we will study $Y^{m,p}(Q)$ for a cube~$Q$.

We will often wish to consider the behavior of functions in thin annuli. Thus, in Section~\ref{sec:annuli} we will establish results in (possibly thin) annuli rather than cubes. We will sometimes need different forms of estimates, and so will also investigate the Poincar\'e inequality in thin annuli.

Finally, in Section~\ref{sec:cutoff}, we will investigate the behavior of Sobolev functions when multiplied by cutoff functions; since our standard cutoff functions have gradients supported in an annulus, this will build on the results of Section~\ref{sec:annuli}.

\subsection{Global Sobolev spaces}\label{sec:W:Y} In this section we will establish some basic properties of the spaces $\dot W^{m,p}(\RR^d)$ and $Y^{m,p}(\RR^d)$.
We begin by citing the Gagliardo-Nirenberg-Sobolev inequality; a proof may be found in (for example) \cite[Section~5.6.1, Theorem~1]{Eva98}.

\begin{theorem}[The Gagliargo-Nirenberg-Sobolev inequality]
\label{GNSi}
Assume $1\leq p<d$.  Then there is a constant $C$ which depends only on $p$ and $d$ so that
\begin{equation*}
\|u\|_{L^{p^*}(\RR^d)}\leq C\|\nabla u\|_{L^p(\RR^d)}
\end{equation*}
for all $u\in C^1_c(\RR^d)$.
\end{theorem}
Here $p^*$ is as in formula~\eqref{pstar}.

\begin{remark}\label{rmk:GNS} If $u$ is a representative of an element of $\dot W^{1,p}(\R^d)$ for $p<d$, (that is, if $\nabla u\in L^p(\RR^d)$,) a standard argument involving the Poincar\'e inequality in an annulus shows that even if $u\notin C^1_c(\RR^d)$, there is a unique constant $c$ such that $u-c\in L^{p^*}(\RR^d)$ and
\begin{equation*}
\|u-c\|_{L^{p^*}(\RR^d)}\leq C\|\nabla u\|_{L^p(\RR^d)}.
\end{equation*}
\end{remark}

\begin{corollary}\label{cor:GNS:global}
Suppose that $m\geq 1$, $d\geq 2$ are integers and that $1\leq p<\infty $. Then there exists a constant $c$ depending only on $d$, $m$ and~$p$ with the following significance.
Suppose $\vec u$ is a representative of an element of $\dot W^{m,p}(\RR^d)$. Then there is a polynomial $\vec P$ of order at most $m-1$, unique up to adding polynomials of order at most $m-d/p$, such that
\begin{equation*}
\| u- P\|_{Y^{m,p}(\RR^d)}
\leq c\| u\|_{\dot W^{m,p}(\RR^d)}.\end{equation*}
In particular $\| u- P\|_{Y^{m,p}(\RR^d)}$ is finite.
\end{corollary}
\begin{proof}
Recall the definition~\eqref{pk} of~$p_{m,d,k}$.
Because $(p_{m,d,k+1})^* = p_{m,d,k}$, if $m-d/p<k< m$, the bound
\begin{equation*}
\|\nabla^k (u-P)\|_{L^{p_{m,d,k}}(\RR^d)}\leq C_k\|\nabla^{k+1} u\|_{L^{p_{m,d,k+1}}(\RR^d)}
\end{equation*}
for some $C_k$ follows from Remark~\ref{rmk:GNS}. By induction, and because $p_{m,d,m}=p$,
\begin{equation*}
\|\nabla^k (u-P)\|_{L^{p_{m,d,k}}(\RR^d)}\leq C_k\|\nabla^{m} u\|_{L^{p}(\RR^d)}
.\end{equation*}
Applying the definitions \eqref{dfn:W:norm} and~\eqref{dfn:Y:norm} of $\dot W^{m,p}(\RR^d)$ and~$Y^{m,p}(\RR^d)$ completes the proof.
\end{proof}

We will now establish a bound on the Bochner norm of elements of~$Y^{m,p}(\R^d)$.

\begin{corollary}\label{cor:GNS:slices}
Let $m\in\NN$, $p\in[1,d-1)$. Let $k\in\NN_0$ satisfy $m-(d-1)/p<k< m $. Let $u$ be a representative of an element of $\dot W^{m,p}(\R^d)$ and let $P$ be the polynomial in Corollary~\ref{cor:GNS:global}.
Then
\begin{equation*}
\|\nabla^k (u-P)\|_{L_t^pL_x^{p_{m,d-1,k}}(\R^d)}
\leq C\|\nabla^m u\|_{L^p(\RR^d)}
.\end{equation*}
In particular, if $u\in Y^{m,p}(\RR^d)$ then this bound is valid with $P=0$.
\end{corollary}
\begin{proof}
By definition, 
\begin{align*}
\|\nabla^m u\|_{L^p(\RR^d)}
&=
\left(\int_{-\infty }^\infty \int_{\R^{d-1}} |\nabla^{m} u(x,t)|^{p}\,dx\,dt\right)^{1/p_k}
\\
&=
\left(\int_{-\infty }^\infty \|\nabla^{m} u(\,\cdot\,,t)\|_{L^{p}(\R^{d-1})}^{p}\,dt\right)^{1/p}
.\end{align*}
Let $0\leq j\leq k$.
Applying Corollary~\ref{cor:GNS:global} in $\R^{d-1}$ with $d$ replaced by $d-1$ yields that, for some polynomial $P_{j,t}$ defined on $\R^{d-1}$,
\begin{align*}\|\nabla_x^{k-j}\partial_t u(\,\cdot\,,t)-\nabla_x^{k-j} P_{j,t}\|_{L^{ p_{m,d-1,k}}(\R^{d-1})}
&\leq C\|\nabla_x^{m-j}\partial_t u(\,\cdot\,,t)\|_{L^{p}(\R^{d-1})}
\\&\leq C\|\nabla^{m} u(\,\cdot\,,t)\|_{L^{p}(\R^{d-1})}
\end{align*}
where $\nabla_x$ denotes the gradient taken strictly in the horizontal variables.

For almost every $t\in\R$, by Corollary~\ref{cor:GNS:global} (in $\R^d$) we have that
\begin{equation*}\|\nabla^{k-j}_x\partial_t^j u(\,\cdot\,,t) - \nabla^{k-j}_x\partial_t^j P(\,\cdot\,,t)\|_{L^{p_{m,d,k}}(\R^{d-1})} <\infty .\end{equation*}
Because $\partial_t^j P(x,t)$ and $P_{j,t}(x)$ are polynomials in~$x$, we must have that for almost every $t\in\R$, $\nabla^{k-j}_x\partial_t^j P(x,t)=\nabla_x^{k-j} P_{j,t}(x)$ for all $x\in\R^{d-1}$. This completes the proof.
\end{proof}

\subsection{Sobolev functions in cubes}

\label{sec:evlem}

In this section we will establish analogues to Corollaries \ref{cor:GNS:global} and~\ref{cor:GNS:slices} in cubes (rather than in all of Euclidean space).

\begin{lemma}\label{evlem} 
Let $m$, $d\in\NN$, $d\geq 2$, $p\in[1,\infty )$, and let $j$, $k\in\NN_0$ satisfy
$0\leq j\leq k$ and $m-d/p< k\leq  m$. Let $p_k=p_{m,d,k}$. Then there is a constant $C$ depending only on~$p$, $d$, and~$m$ such that if $Q\subset\R^d$ is a cube and $u\in W^{m,p}(Q)$, then
\begin{equation*}
\|\nabla^j u\|_{L^{p_{k}}(Q)}\leq C\sum_{i=j}^{ m-k+j} |Q|^{(i-j+k- m)/d} \|\nabla^i u\|_{L^{p}(Q)}
.\end{equation*}

\end{lemma}
\begin{proof} 
Suppose first that $|Q|=1$.
By the Gagliardo-Nirenberg-Sobolev inequality in bounded domains (see, for example, \cite[Section~5.6.1, Theorem 2]{Eva98}) and the definition~\eqref{pk} of~$p_k$, we have that
\begin{equation*}\|w\|_{L^{p_{k}}(Q)}\leq C\|w\|_{W^{1,p_{k+1}}(Q)}
=C\sum_{i=0}^{1}\|\nabla^i w\|_{L^{p_{k+1}}(Q)}\end{equation*}
for any function $w\in W^{1,p_{k+1}}(Q)$. Taking $w=\nabla^j u$, we see that
\begin{equation*}\|\nabla^ju\|_{L^{p_{k}}(Q)}\leq C\sum_{i=j}^{j+1}\|\nabla^{i} u\|_{L^{p_{k+1}}(Q)}\end{equation*}
Iterating this argument with $w=\nabla^i u$ and recalling that $p=p_m$ yields the $|Q|=1$ case of the lemma. A chage of variables establishes the case for general~$Q$.
\end{proof}

We may also control Bochner norms; this is very useful in the case that the coefficients satisfy the condition~\eqref{eqn:intro:bound:Bochner}.
\begin{lemma}\label{evlem:t} %
Let $m$, $d\in\NN$, $d\geq 2$, $p\in[1,\infty )$, and let $j$, $k\in\NN_0$ satisfy $0\leq j\leq k$ and $m-(d-1)/p< k\leq m$. Let $\widetilde p_k=p_{m,d-1,k}$.
There is a constant $C$ depending only on~$p$, $d$, and~$m$ such that if $Q\subset\R^d$ is a cube with sides parallel to the coordinate axes and $u\in W^{m,p}(Q)$, then
\begin{equation*}
\|\nabla^j u\|_{L_t^pL^{\widetilde p_k}_k(Q)}
\leq
C\sum_{i=j}^{ m-k+j} |Q|^{(i-j+k- m)/d}\|\nabla^i u\|_{L^p(Q)}
.\end{equation*}
\end{lemma}
\begin{proof}
Let $Q=\Delta\times [t_0,t_0+R]$, where $\Delta\subset\R^\dmnMinusOne$ is a cube, $t_0\in\R$, and $R=|Q|^{1/d}$.
Recall that
\begin{equation*}
\|\nabla^j u\|_{L_t^pL^{\widetilde p_k}_k(Q)}
=\left(\int_{t_0}^{t_0+R}\left(\int_{\Delta} |\nabla^j u(x,t)|^{\widetilde p_k}\,dx\right)^{p/\widetilde p_k} dt \right)^{1/p}
.\end{equation*}
Applying Lemma~\ref{evlem} in dimension $d-1$, we see that
\begin{equation*}
\left(\int_{\Delta} |\nabla^j u(x,t)|^{\widetilde p_k}\,dx\right)^{1/\widetilde p_k}
\leq
C\sum_{i=j}^{ m-k+j} R^{i-j+k- m}\|\nabla^i u(\,\cdot\,,t)\|_{L^p(\Delta)}.
\end{equation*}
Integrating in~$t$ completes the proof.
\end{proof}

\subsection{Sobolev functions in annuli}
\label{sec:annuli}

We will now establish analogues to Lemmas~\ref{evlem} and~\ref{evlem:t} in annuli.

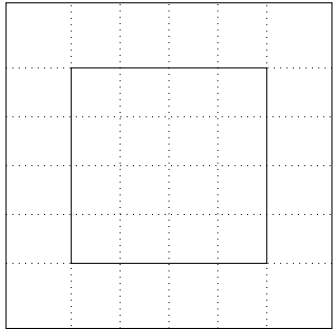
\begin{figure}
\begin{center}
\begin{tikzpicture}[scale=1.3]
\def\r{1cm}
\def\N{4}
\def\R{1.6666cm} % R = \theta = 1+2/(N-1)
\count255=\N\advance\count255 -1 \xdef\NMinusOne{\the\count255}

\draw (-\r,-\r) rectangle (\r,\r);
\draw (-\R,-\R) rectangle (\R,\R);

\foreach \i in {1,...,\NMinusOne} {
\draw [dotted] (-\R,-\r+2*\i*\r/\N) -- (\R,-\r+2*\i*\r/\N);
\draw [dotted] (-\r+2*\i*\r/\N,-\R) -- (-\r+2*\i*\r/\N,\R);
}

\draw [dotted] (-\R,-\r)--(-\r,-\r)--(-\r,-\R);
\draw [dotted] (\R,-\r)--(\r,-\r)--(\r,-\R);
\draw [dotted] (-\R,\r)--(-\r,\r)--(-\r,\R);
\draw [dotted] (\R,\r)--(\r,\r)--(\r,\R);

%\node at (0,0) [below] {$x_0$};\node at (0,0) {$\bullet$};

\end{tikzpicture}
\end{center}
\caption{The rectangles in the proof of Lemma~\ref{evlem:annulus}}
\end{figure}

\begin{lemma}\label{evlem:annulus}
Let $m$, $d\in\NN$, $d\geq 2$, $p\in[1,\infty )$, and let $j$, $k\in\NN_0$ satisfy
$0\leq j\leq k$ and $m-d/p< k\leq m$. Let $p_k=p_{m,d,k}$. Let $1<\theta\leq 2$.

Then there is a constant $C$ depending only on $p$, $d$, and~$m$ such that if $Q\subset\RR^d$ is a cube with sides parallel to the coordinate axes and $u\in W^{m,p}(\theta Q\setminus Q)$, then
\begin{equation*}
\|\nabla^j u\|_{L^{p_{k}}(\theta Q\setminus Q)}\leq C\sum_{i=j}^{m-k+j} \frac{C}{((\theta-1)|Q|^{1/d})^{m-k+j-i}} \|\nabla^i u\|_{L^p(\theta Q\setminus Q)}.\end{equation*}

If in addition $k>m-(d-1)/p$, then
\begin{equation*}
\|\nabla^j u\|_{L_t^pL_x^{\widetilde p_k}(\theta Q\setminus Q)}
\leq
C\sum_{i=j}^{m-k+j} \frac{C}{((\theta-1)|Q|^{1/d})^{m-k+j-i}} \|\nabla^i u\|_{L^p(\theta Q\setminus Q)}.\end{equation*}
\end{lemma}

\begin{proof}
Observe that there exists an integer $n\geq 2$ with $\frac{1}{n}\leq\frac{\theta-1}{2}< \frac{1}{n-1}\leq \frac{2}{n}$. Without loss of generality we assume that $Q$ is open. Let $I_1,
\dots,I_d$ be the $d$ open intervals that satisfy $Q=I_1\times\dots\times I_d$ . If $I_k=(a_k,b_k)$, and $r=b_k-a_k=|Q|^{1/d}$, define the $d(n+2)$ intervals $I_{k,j}$ by
\begin{gather*}I_{k,0}=\biggl(a_k-\frac{\theta-1}{2}r,a_k\biggr),
\quad I_{k,n+1}=\biggl(b_k,b_k+\frac{\theta-1}{2}r\biggr),
\\ I_{k,j}=\biggl(a_k+\frac{j-1}{n}r,a_k+\frac{j}{n}r\biggr)\text{ if } 1\leq j\leq n. \end{gather*}
Let $G=\{I_{1,j_1}\times I_{2,j_2}\times\dots \times I_{d,j_d}:j_k\in \{0,1,\dots,n+1\}\}$, and let $H\subset G$ be given by $H=\{I_{1,j_1}\times I_{2,j_2}\times\dots \times I_{d,j_d}:j_k\in \{1,\dots,n\}\}$. Up to a set of measure zero,
\begin{equation*}\theta Q=\bigcup_{R\in G} R,\qquad Q=\bigcup_{R\in H} R.\end{equation*}
Furthermore, the rectangles in $G$ are pairwise disjoint. If $R\in G$ then the shortest side of $R$ is at least $r/n$ and the longest side is at most $(\theta-1)r/2<2r/n$. A change of variables argument shows that Lemmas~\ref{evlem} and~\ref{evlem:t} are valid in $R$ with uniformly bounded constants.

Suppose $m-(d-1)/p<k<m$.
If $\Omega\subseteq\RR^d$, recall that $[\Omega]^t=\{x\in\RR^{d-1}:(x,t)\in \Omega\}$.
Then
\begin{align*}
\|\nabla^j u\|_{L_t^p L_x^{\widetilde p_k}(\theta Q\setminus Q)}
&=
\biggl(
\int_{-\infty }^\infty
\biggl(\int_{[\theta Q\setminus Q]^t} |\nabla^j u(x,t)|^{\widetilde p_k} \,dx\biggr)^{p/\widetilde p_k}
dt
\biggr)^{1/p}
\\&=
\biggl(
\int_{-\infty }^\infty
\biggl(
\sum_{R\in G\setminus H}
\int_{[R]^t} |\nabla^j u(x,t)|^{\widetilde p_k} \,dx\biggr)^{p/\widetilde p_k}
dt
\biggr)^{1/p}
.\end{align*}
Because $p/\widetilde p_k \leq 1$, we have that
\begin{align*}
\|\nabla^j u\|_{L_t^p L_x^{\widetilde p_k}(\theta Q\setminus Q)}
&\leq
\biggl(
\sum_{R\in G\setminus H}
\int_{-\infty }^\infty
\biggl(
\int_{[R]^t} |\nabla^j u(x,t)|^{\widetilde p_k} \,dx\biggr)^{p/\widetilde p_k}
dt
\biggr)^{1/p}
\\&=
\biggl(
\sum_{R\in G\setminus H}
\|\nabla^j u\|_{L_t^p L_x^{\widetilde p_k}(R)}^p\biggr)^{1/p}
.\end{align*}
By Lemma~\ref{evlem:t} in rectangles,
\begin{align*}
\|\nabla^j u\|_{L_t^p L_x^{\widetilde p_k}(\theta Q\setminus Q)}
&\leq
\biggl(
\sum_{R\in G\setminus H}
\biggl(\sum_{i=j}^{m-k+j} \frac{C}{((\theta-1)r)^{m-k+j-i}} \|\nabla^i u\|_{L^p(R)}\biggr)^p\biggr)^{1/p}
.\end{align*}
By the triangle inequality in the sequence space $\ell^p$,
\begin{align*}
\|\nabla^j u\|_{L_t^p L_x^{\widetilde p_k}(\theta Q\setminus Q)}
&\leq
\frac{C}{((\theta-1)r)^{m-k+j-i}}\sum_{i=j}^{m-k+j}
\biggl(
\sum_{R\in G\setminus H}
\bigl( \|\nabla^i u\|_{L^p(R)}\bigr)^p\biggr)^{1/p}
\\&=
\frac{C}{((\theta-1)r)^{m-k+j-i}}\sum_{i=j}^{m-k+j}
\|\nabla^i u\|_{L^p(\theta Q\setminus Q)}
.\end{align*}
A similar (and simpler) argument establishes the bound on $\|\nabla^j u\|_{L^{p_k}(\theta Q\setminus Q)}$.
\end{proof}

Lemma~\ref{evlem:annulus} generalizes the Gagliardo-Nirenberg-Sobolev inequality to thin annuli. We remark on the presence of the term $\theta-1$ in the denominator of the right-hand side. In a thin annulus, this term is potentially very small and so Lemma~\ref{evlem:annulus} yields a poor bound.

The following lemma allows us to bound a function $u$ in an annulus by its gradient, without powers of $(\theta-1)$. We observe that the following lemma is a special case of the Poincar\'e inequality and not of the Gagliardo-Nirenberg-Sobolev inequality; that is, we do not gain higher integrability (a higher power of~$u$) on the left-hand side. We will use both Lemma~\ref{evlem:annulus} and Lemma~\ref{lem:annulus} in different contexts.

\begin{lemma}\label{lem:annulus}
Let $d\geq 2$ be an integer and let $1\leq p<\infty$. There is a constant $C=C_{d,p}$ depending only on $d$ and~$p$ such that if $Q\subset\R^\dmn$ is a cube, $1<\theta\leq 2$, and $u\in W^{1,p}(\theta Q\setminus Q)$, then
\begin{equation*}\int_{\theta Q\setminus Q} \bigl|u-{\textstyle\fint_{\theta Q\setminus Q}u}\bigr|^p\leq C_{d,p}|Q|^{p/d}\int_{\theta Q\setminus Q} |\nabla u|^p.\end{equation*}
\end{lemma}

\begin{proof}
We restrict to the case $|Q|=1$ and where the midpoint of~$Q$ is the origin (that is, the case $Q=(-1/2,1/2)^\dmn$); rescaling and translating yields the general case.

Let $\rho(X)=2\max\{X_1,\dots,X_d\}$. Thus, if $X\in\RR^d$, then $\rho(X)$ is the unique real number with $X\in\partial(\rho(X)\,Q)$. Observe that $\rho$ is a Lipschitz function with $|\nabla\rho|=2$ almost everywhere and with $\rho(X)\leq 2|X|\leq \sqrt{d}\rho(X)$.
Define
\begin{equation*}r(t)=\biggl(\frac{\theta^d-1}{2^d-1}t^d+ \frac{2^d-\theta^d}{2^d-1}\biggr)^{1/d}.
\end{equation*}
Observe that $r(1)=1$, $r(2)=\theta$, $r$ is increasing, $r(t)/t$ is decreasing, and $r(t)^{d-1}\,r'(t) = \frac{\theta^d-1}{2^d-1} t^{d-1}$. In particular, if $1\leq t\leq 2$ then $0<r'(t)\leq \theta^d-1$.

Let $\psi(X)= X\,r(\rho(X))/\rho(X)$. Then $\psi$ is a bilipschitz change of variables $\psi:2Q\setminus Q\to \theta Q\setminus Q$.

If $f\in L^1(2 Q\setminus Q)$, then
\begin{align*}\int_{2Q\setminus Q} f
&= \frac{1}{2}
\int_1^2 \int_{\partial(t Q)} f(X)\,d\sigma(X)\,dt
\end{align*}
where $\sigma$ denotes $d-1$-dimensional Hausdorff measure (that is, surface measure on the boundary of the cube~$tQ$). In particular, letting $f=g\circ\psi$ and making the change of variables $X=tY$ in the inner integral, we have that
\begin{align*}\int_{2Q\setminus Q} g\circ\psi
&= \frac{1}{2}
\int_1^2 t^{d-1}\int_{\partial Q} g\circ\psi(tY)\,d\sigma(Y)\,dt
.\end{align*}
If $Y\in \partial Q$, then $\rho(tY)=t$ and so $\psi(tY)=r(t)\,Y$. Thus
\begin{align*}\int_{2Q\setminus Q} g\circ\psi
&= \frac{1}{2}
\int_1^2 \int_{\partial Q} g(r(t)\,Y)\,d\sigma(Y)\,t^{d-1}\,dt
.\end{align*}
Applying our above formula for~$r'(t)$,
\begin{align*}\int_{2Q\setminus Q} g\circ\psi
&= \frac{2^d-1}{2(\theta^d-1)}\int_1^2 \int_{\partial Q} g(r(t)\,X)\,d\sigma(X)\,r(t)^{d-1}\,r'(t)\,dt
.\end{align*}
Using the chain rule of single variable calculus and reversing our above arguments,
\begin{align*}\int_{2Q\setminus Q} g\circ\psi
&= \frac{2^d-1}{2(\theta^d-1)}\int_1^\theta \int_{\partial Q} g(r\,X)\,d\sigma(X)\,r^{d-1}\,dr
\\&= \frac{2^d-1}{2(\theta^d-1)}\int_1^\theta \int_{\partial (rQ)} g(X)\,d\sigma(X)\,dr
\\&= \frac{2^d-1}{\theta^d-1}\int_{\theta Q\setminus Q} g
.\end{align*}
We will apply this argument to $g=u$ and to $g=|u|^p$. In particular,
\begin{align*}
\fint_{\theta Q\setminus Q} u
&= \frac{1}{|\theta Q\setminus Q|}\int_{\theta Q\setminus Q} u
=\frac{\theta^d-1}{(2^d-1)|\theta Q\setminus Q|}
\int_{2Q\setminus Q} u\circ\psi
=
\fint_{2Q\setminus Q} u\circ\psi
.\end{align*}

We also need to integrate the gradient.
Let $J_\psi$ be the Jacobian matrix for the change of variables $\psi$, so that $\nabla (u\circ\psi)=(J_\psi\nabla u)\circ\psi$. If $X\in 2Q\setminus Q$, then
\begin{equation*}\left|\frac{\partial \psi_j}{\partial X_k}\right|
=\left|\frac{r(\rho(X))}{\rho(X)} \delta_{jk}
+X_j \frac{r'(\rho(X)) \,\rho(X)-r(\rho(X))}{\rho(X)^2}\partial_k\rho(X)\right|
\leq
1+2(\theta^d-1)
\end{equation*}
and so $J_\psi$ is a bounded matrix. Thus,
\begin{align*}
\biggl(\int_{2Q\setminus Q} |\nabla (u\circ\psi)|^p\biggr)^{1/p}
&=\biggl(\int_{2Q\setminus Q} |(J_\psi\nabla u)\circ\psi|^p\biggr)^{1/p}
\\&\leq C_{d}\biggl(\int_{2Q\setminus Q} |(\nabla u)\circ\psi|^p\biggr)^{1/p}
.\end{align*}

Now, 
\begin{align*}
\int_{\theta Q\setminus Q} |u-{\textstyle\fint_{\theta Q\setminus Q}u}|^p
&=\frac{\theta^d-1}{2^d-1} \int_{2 Q\setminus Q} |u\circ\psi-{\textstyle\fint_{2 Q\setminus Q}u\circ\psi}|^p
\\&\leq C_{d,p}\frac{\theta^d-1}{2^d-1} \int_{2Q\setminus Q} |\nabla (u\circ\psi)|^p
\\&\leq C_{d,p}\frac{\theta^d-1}{2^d-1}\int_{2Q\setminus Q} |(\nabla u)\circ\psi|^p
= C_{d,p}\int_{\theta Q\setminus Q} |\nabla u|^p
.\end{align*}
Thus the Poincar\'e inequality holds in an annulus with constant independent of~$\theta$.
\end{proof}

\subsection{Sobolev norms and cutoff functions}

\label{sec:cutoff}

A particular application of Lemmas~\ref{evlem:annulus} and~\ref{lem:annulus} is the following result concerning smooth cutoff functions.

\begin{lemma}\label{uchi}
Let $m$, $d\in\NN$, $d\geq 2$, and let $1\leq p<\infty$. There is a constant $C$ depending on $m$, $d$ and $p$ with the following significance.

Let $Q\subset\RR^d$ be a cube and let $1<\theta\leq 2$. Let $\chi\in C^\infty _c(\RR^d)$ be a test function supported in~$\theta Q$ and identically equal to~$1$ in~$Q$, with $0\leq \chi\leq 1$.
Define $X=\max_{1\leq i\leq d} (\theta-1)^i|Q|^{i/d}\|\nabla^i \chi\|_{L^\infty (Q)}$.

If $u\in W^{m,p}(\theta Q)$ (equivalently, if $u\in Y^{m,p}(\theta Q)$), and if we extend $u\chi$ by zero outside of~$\theta Q$, then ${u}\chi\in Y^{m,p}(\RR^\dmn)$ and
\begin{align*}
\|{u}\chi\|_{Y^{m,p}(\RR^d)}
&\leq \|u\|_{Y^{m,p}(\theta Q)}+
\sum_{i=0}^{m-1} \frac{CX}{((\theta-1)|Q|^{1/d})^{m-i}} \|\nabla^i u\|_{L^p(\theta Q\setminus Q)}
.\end{align*}

\end{lemma}
\begin{proof}

We begin by using the definition of the $Y^{m,p}$-norm  and the Leibniz rule.
\begin{align*}\|{u}\chi\|_{Y^{m,p}(\RR^d)}
&=\sum_{m-d/p<k\leq m}\|\nabla^k({u}\chi)\|_{L^{p_k}(\RR^d)}
\\&\leq \sum_{m-d/p<k\leq m}\biggl(\int_{\RR^d}\biggl(\sum_{j=0}^k C_{j,k}|\nabla^{k-j}\chi|\,|\nabla^j{u}|\biggr)^{p_k}\biggr)^{1/p_k}.
\end{align*}
Observe that $C_{k,k}=1$. By definition of~$X$ and isolating the $j=k$ terms,
\begin{align*}
\|{u}\chi\|_{Y^{m,p}(\RR^d)}
&\leq
\sum_{m-d/p<k\leq m}\biggl(\int_{\theta Q}
|\nabla^{k}{u}|^{p_k}\biggr)^{1/p_k}
\\&\qquad+
C\sum_{m-d/p<k\leq m}\biggl(\int_{\theta Q\setminus Q}\biggl(\sum_{j=0}^{k-1}
X(\theta-1)^{j-k}|Q|^{(j-k)/d}
|\nabla^{j}{u}|\biggr)^{p_k}\biggr)^{1/p_k}
\\
&\leq \|u\|_{Y^{m,p}(\theta Q)}+
\sum_{m-d/p<k\leq m}\sum_{j=0}^{k-1}
\frac{CX}{
((\theta-1)|Q|^{1/d})^{k-j}}
\bigg(\int_{\theta Q\setminus Q}| \nabla^{j}{u}|^{p_k}\bigg)^{1/p_k}
.\end{align*}
By Lemma~\ref{evlem:annulus},
\begin{align*}
\|{u}\chi\|_{Y^{m,p}(\RR^d)}
&\leq \|u\|_{Y^{m,p}(\theta Q)}+
\sum_{i=0}^{m-1} \frac{CX}{((\theta-1)|Q|^{1/d})^{m-i}} \|\nabla^i u\|_{L^p(\theta Q\setminus Q)}
.\end{align*}
This completes the proof.
\end{proof}

\section{The Caccioppoli inequality}\label{CacSec}

The Caccioppoli inequality was established first by Caccioppoli in the early twentieth century and is a foundational result used throughout the theory of second order divergence form equations. It has been generalized to the case of second order operators with lower order terms in \cite{DavHM18}, and of higher order equations (without lower order terms) first in \cite{Cam80}, and later with some refinements in \cite{AusQ00,Bar16}.

We now generalize these results to the case of higher order equations with lower order terms.
We will follow \cite{AusQ00} and derive a Caccioppoli inequality for equations that satisfy the weak Gårding inequality~\eqref{wgi} (and not necessarily the stronger Gårding inequality~\eqref{gi}).
We will follow \cite{Cam80} and establish the Caccioppoli inequality for solutions $\vec u$ to inhomogeneous equations $L\vec{u}=T$ for a (possibly nonzero) element~$T$ of $Y^{-m,p}$.

We begin with the following lemma. This lemma was proven first in \cite{Cam80} for operators of order~$2m$ without lower order terms.

\begin{lemma}\label{CacLem}
Let $L$ be an operator of order $2m$ of the form~\eqref{dfn:L} associated to coefficients $A$ that satisfy the weak G\aa{}rding inequality \eqref{wgi} and either the bound \eqref{eqn:intro:bound} or the bound~\eqref{eqn:intro:bound:Bochner}.

Let  $Q\subset\RR^d$ be an open cube with sides parallel to the coordinate axes, and let $1<\theta\leq 2$.
Let $\vec u\in W^{m,2}(\theta Q)$. Let $T\in Y^{-m,2}(\theta Q)$.
Suppose that $L\vec u=T$ in $\theta Q$ in the sense of formula~\eqref{dfn:L}.

Then we have that
\begin{equation*}
\int_{Q}|\nabla^m\vec{u}|^2\leq \sum_{k=0}^{m-1}\frac{C}{((\theta-1)|Q|^{1/d})^{2m-2k}}
\int_{\theta Q\setminus Q}|\nabla^k\vec{u}|^2
+C\delta\int_{\theta Q}|\vec{u}|^2
+C\|T\|^2
\end{equation*}
where  $C$ is a constant depending on the dimension $d$, the order $2m$ of $L$, the number $\lambda$ in the bound~\eqref{wgi}, and the number $\Lambda$ in the bound~\eqref{eqn:intro:bound} or~\eqref{eqn:intro:bound:Bochner}. Here $\|T\|=\|T\|_{Y^{-m,2}(\theta Q)}$ is the operator norm, that is, the smallest number such that $|\langle\vec\psi,T\rangle|\leq \|\vec\psi\|_{Y^{m,2}(\theta Q)} \|T\|$ for all $\vec\psi\in Y^{m,2}_0(\theta Q)$.
\end{lemma}

\begin{proof}
Let $\rho=((\theta-1)/2)|Q|^{1/d}$ be the distance from $Q$ to $\RR^d\setminus\theta Q$.
Let $\varphi$ be a smooth, real valued test function with $0\leq\varphi\leq 1$, supported in $\theta Q$ and identically equal to 1 on~$Q$.  We require also that $|\nabla^k\varphi|\leq C_k \rho^{-k}$ for any integer $k\geq0$.

Define $\vec{\psi}=\varphi^{4m}\vec{u}$.
Notice that by Lem\-ma~\ref{uchi}, $\vec{\psi}\in Y^{m,2}_0(\theta Q)$. Furthermore, by formula~\eqref{dfn:L},
\begin{equation}\label{first}
\sum_{j,k=1}^N\sum_{|\alpha|\leq m}\sum_{|\beta|\leq m} \int_{\theta Q} \partial^\alpha(\varphi^{4m}\overline{u}_j)\, A^{j,k}_{\alpha,\beta} \,\partial^\beta u_k
=\overline{\langle T,\varphi^{4m}\vec{u}\rangle_{\theta Q}}.
\end{equation}

We first consider the left hand side of formula~\eqref{first}.
By the Leibniz rule, and separating out the $\gamma=\alpha$ terms, we see the following.
\begin{align*}
\int_{\theta Q} \partial^\alpha(\varphi^{4m}\overline{u}_j)\,A^{j,k}_{\alpha,\beta} \, \partial^\beta u_k
&=
\int_{\theta Q} \partial^\alpha(\varphi^{2m}\overline{u}_j)\, A^{j,k}_{\alpha,\beta} \,\varphi^{2m}\,\partial^\beta u_k
\\
&\qquad+\int_{\theta Q} \sum_{\gamma<\alpha} \frac{\alpha!}{\gamma!(\alpha-\gamma)!} \, \partial^{\alpha-\gamma}(\varphi^{2m}) \, \partial^\gamma(\varphi^{2m}\overline{u}_j) \,A^{j,k}_{\alpha,\beta} \,\partial^\beta u_k
.\end{align*}
Now as in \cite{Bar16}, we write
\begin{equation}\label{7}
\sum_{\gamma<\alpha} \frac{\alpha!}{\gamma!(\alpha-\gamma)!} \partial^{\alpha-\gamma}(\varphi^{2m})\partial^\gamma(\varphi^{2m}\overline{u}_j)=\sum_{\zeta<\alpha}\varphi^{2m}\Phi_{\alpha,\zeta}\partial^\zeta\overline{u}_j
\end{equation}
for some functions $\Phi_{\alpha,\zeta}$ which are supported in $\theta Q \setminus Q$ and satisfy $|\Phi_{\alpha,\zeta}|\leq C\rho^{|\zeta|-|\alpha|}$.  Thus we have
\begin{align*}
\int_{\theta Q} \partial^\alpha(\varphi^{4m}\overline{u}_j)
\, A^{j,k}_{\alpha,\beta}\, \partial^\beta u_k
&=
\int_{\theta Q} \partial^\alpha(\varphi^{2m}\overline{u}_j)\, A^{j,k}_{\alpha,\beta}\, \varphi^{2m}\, \partial^\beta u_k
\\&\qquad
+\int_{\theta Q}\sum_{\zeta<\alpha}\Phi_{\alpha,\zeta} \, \partial^\zeta\overline{u}_j\, A^{j,k}_{\alpha,\beta}\, \varphi^{2m}\, \partial^\beta u_k.
\end{align*}
It is desirable to have our final term in terms of $\partial^\beta(\varphi^{2m}u_k)$ rather than $\varphi^{2m}\partial^\beta u_k$, so after one more application of the Leibniz rule, and writing as in formula~\eqref{7}, we have for some functions $\Psi_{\beta,\xi}$ which are supported in $\theta Q\setminus Q$ and satisfy $|\Psi_{\beta,\xi}|\leq C\rho^{|\xi|-|\beta|}$
\begin{align*}
\int_{\theta Q} \partial^\alpha(\varphi^{4m}\overline{u}_j)\, A^{j,k}_{\alpha,\beta}\, \partial^\beta u_k
&=
\int_{\theta Q}\partial^\alpha(\varphi^{2m}\overline{u}_j)\, A^{j,k}_{\alpha,\beta}\, \varphi^{2m}\, \partial^\beta u_k
\\
&
+ \int_{\theta Q} \sum_{\zeta<\alpha}\Phi_{\alpha,\zeta}\, \partial^\zeta\overline{u}_j \, A^{j,k}_{\alpha,\beta}\, \partial^\beta(\varphi^{2m}u_k)
\\
&
-\int_{\theta Q} \sum_{\zeta<\alpha}\Phi_{\alpha,\zeta}\, \partial^\zeta\overline{u}_j \, A^{j,k}_{\alpha,\beta}\sum_{\xi<\beta}\varphi^m\, \Psi_{\beta,\xi}\, \partial^\xi u_k.
\end{align*}
Similar measures as taken above also give us
\begin{align*}
\int_{\theta Q}\partial^\alpha(\varphi^{2m}\overline{u}_j)\, A^{j,k}_{\alpha,\beta}\, \partial^\beta(\varphi^{2m}u_k)
&=
\int_{\theta Q}\partial^\alpha(\varphi^{2m}\overline{u}_j)\, A^{j,k}_{\alpha,\beta}\, \sum_{\xi<\beta}\varphi^{m}\, \Psi_{\beta,\xi}\, \partial^\xi u_k
\\
&\qquad+\int_{\theta Q}\partial^\alpha(\varphi^{2m}\overline{u}_j) \, A^{j,k}_{\alpha,\beta} \, \varphi^{2m}\, \partial^\beta u_k.
\end{align*}
Thus combining the previous two equations and reintroducing summation, we see that
\begin{multline*}
\sum_{j,k=1}^N
\sum_{|\alpha|\leq m} \sum_{|\beta|\leq m}
\int_{\theta Q}\partial^\alpha(\varphi^{2m}\overline{u}_j)\, A^{j,k}_{\alpha,\beta}\, \partial^\beta(\varphi^{2m}u_k)
\\\begin{aligned}&=
\sum_{j,k=1}^N\sum_{|\alpha|\leq m} \sum_{|\beta|\leq m}
\int_{\theta Q}\partial^\alpha(\varphi^{4m}\overline{u}_j)\, A^{j,k}_{\alpha,\beta}\, \partial^\beta u_k
\\&\qquad
-\sum_{j,k=1}^N\sum_{|\alpha|\leq m} \sum_{|\beta|\leq m}
\int_{\theta Q}\sum_{\zeta<\alpha}\Phi_{\alpha,\zeta}\, \partial^\zeta\overline{u}_j\, A^{j,k}_{\alpha,\beta}\, \partial^\beta(\varphi^{2m}u_k)
\\&\qquad
+\sum_{j,k=1}^N \sum_{|\alpha|\leq m} \sum_{|\beta|\leq m}
\int_{\theta Q}\sum_{\zeta<\alpha}\Phi_{\alpha,\zeta}\, \partial^\zeta\overline{u}_j\, A^{j,k}_{\alpha,\beta}\sum_{\xi<\beta}\varphi^{m}\, \Psi_{\beta,\xi}\, \partial^\xi u_k
\\&\qquad
+\sum_{j,k=1}^N\sum_{|\alpha|\leq m} \sum_{|\beta|\leq m}
\int_{\theta Q}\partial^\alpha(\varphi^{2m}\overline{u}_j)\, A^{j,k}_{\alpha,\beta}\sum_{\xi<\beta}\varphi^{m}\, \Psi_{\beta,\xi}\, \partial^\xi u_k.
\end{aligned}
\end{multline*}
We write this as \rn{1}=\rn{2}+\rn{3}+\rn{4}+\rn{5}. Observe that by formula~\eqref{first},
\begin{equation}\label{eqn:II}
\text{\rn{2}} = \overline{\langle T,\varphi^{4m}\vec u\rangle _{\theta Q}}.\end{equation}
By the condition~\eqref{wgi}, we have that
\begin{equation*}
\lambda\|\nabla^m(\varphi^{2m}\vec{u})\|^2_{L^2({\theta Q})}
\leq \mathop{\mathrm{Re}} \text{\rn{1}}+\delta\|\varphi^{2m}\vec{u}\|^2_{L^2({\theta Q})}.
\end{equation*}

Suppose that the condition~\eqref{eqn:intro:bound:Bochner} is true.
By H\"{o}lder's inequality and properties of~$\Phi_{\alpha,\zeta}$,
\begin{align*}
|\text{\rn{3}}|
&\leq
\sum_{\substack{m-(d-1)/2<|\alpha|\leq m\\m-(d-1)/2<|\beta|\leq m}}\,
\sum_{\zeta<\alpha}
\frac{C\Lambda}{\rho^{|\alpha|-|\zeta|}}
\|\partial^\zeta \vec u\|_{L_t^2 L_x^{\widetilde 2_\alpha}({\theta Q}\setminus Q)}
\|\partial^\beta(\varphi^{2m}\vec u)\|_{L_t^2 L_x^{\widetilde 2_\beta}({\theta Q})}.
\end{align*}
Recall that $\varphi^{2m}\vec u\in Y^{m,2}_0({\theta Q})$ and so may be extended by zero to a $Y^{m,2}(\RR^d)$-function. By Corollary~\ref{cor:GNS:slices} we have that
\begin{align*}\|\partial^\beta(\varphi^{2m}\vec u)\|_{L_t^2 L_x^{\widetilde 2_\beta}({\theta Q})}
&=\|\partial^\beta(\varphi^{2m}\vec u)\|_{L_t^2 L_x^{\widetilde 2_\beta}(\RR^d)}
\\&\leq C\|\nabla^m(\varphi^{2m}\vec{u})\|_{L^2(\RR^d)}
= C\|\nabla^m(\varphi^{2m}\vec{u})\|_{L^2({\theta Q})}.\end{align*}
Summing, we see that
\begin{align*}
|\text{\rn{3}}|
&\leq
\sum_{m-d/2<|\alpha|\leq m}\,
\sum_{\zeta<\alpha}
\frac{C\Lambda}{\rho^{|\alpha|-|\zeta|}}
\|\partial^\zeta \vec u\|_{L_t^2 L_x^{\widetilde 2_\alpha}({\theta Q}\setminus Q)}
\|\varphi^{2m}\vec u\|_{\dot W^{m,2}({\theta Q})}.
\end{align*}
By Lemma~\ref{evlem:annulus},
\begin{equation*}\|\partial^\zeta \vec u\|_{L_t^2 L_x^{\widetilde 2_\alpha}({\theta Q}\setminus  Q)}
\leq
\sum_{i=|\zeta|}^{m-(|\alpha|-|\zeta|)}
\frac{C}{\rho^{m-|\alpha|+|\zeta|-i}}\|\nabla^i \vec u\|_{L^2(\theta Q\setminus Q)}.
\end{equation*}
So
\begin{align*}
|\text{\rn{3}}|
&\leq
\sum_{i=0}^{m-1}
\frac{C}{\rho^{m-i}}\|\nabla^i \vec u\|_{L^2(\theta Q\setminus Q)}
\|\varphi^{2m}\vec u\|_{\dot W^{m,2}({\theta Q})}.
\end{align*}
Applying Young's inequality, we see that
\begin{align*}
|\text{\rn{3}}|
&\leq
\sum_{i=0}^{m-1}
\frac{C}{\rho^{2m-2i}}
\|\nabla^i \vec u\|_{L^2(\theta Q\setminus Q)}^2
+\frac{\lambda}{4}
\|\varphi^{2m}u_k\|_{\dot W^{m,2}({\theta Q})}.
\end{align*}
A similar argument with the roles of $\alpha$, $\zeta$ and $\beta$, $\xi$ reversed yields the same bound on \rn{5}, while an even simpler argument yields the bound
\begin{align*}
|\text{\rn{4}}|
&\leq
\sum_{i=0}^{m-1}
\frac{C}{\rho^{2m-2i}}
\|\nabla^i \vec u\|_{L^2(\theta Q\setminus Q)}^2
.\end{align*}
The argument in the case that the condition~\eqref{eqn:intro:bound} is true is similar.

We thus have that
\begin{align*}
\lambda\|\varphi^{2m}\vec u\|_{\dot W^{m,2}({\theta Q})}^2
&\leq \mathop{\mathrm{Re}} \text{\rn{1}}+\delta\|\varphi^{2m}\vec{u}\|^2_{L^2({\theta Q})}
\\&\leq
|\text{\rn{2}}|+|\text{\rn{3}}|+|\text{\rn{4}}|+|\text{\rn{5}}| +\delta\|\varphi^{2m}\vec{u}\|^2_{L^2({\theta Q})}
\\&\leq
|\text{\rn{2}}|
+C\sum_{i=0}^{m-1} \frac{\|\nabla^i\vec{u}\|_{L^{2}(\theta Q\setminus Q)}^2}{\rho^{2m-2i}}
+\delta\|\varphi^{2m}\vec{u}\|^2_{L^2({\theta Q})}
\\&\qquad+\frac{\lambda}{2}\|\varphi^{2m}\vec{u}\|^2_{\dot W^{m,2}({\theta Q})}
.\end{align*}
Subtracting the final term and applying formula~\eqref{eqn:II} yields that
\begin{align}
\label{eqn:step}
\frac{\lambda}{2}\|\varphi^{2m}\vec u\|_{\dot W^{m,2}({\theta Q})}^2
&\leq
|\langle T, \varphi^{4m}\vec u\rangle_{\theta Q}|
+C\sum_{i=0}^{m-1} \frac{\|\nabla^i\vec{u}\|_{L^{2}(\theta Q\setminus Q)}^2}{\rho^{2m-2i}}
+\delta\|\varphi^{2m}\vec{u}\|^2_{L^2({\theta Q})}
.\end{align}
By definition of~$\|T\|$,
\begin{equation*}|{\langle T,\varphi^{4m}\vec{u}\rangle_{\theta Q}}|
\leq \|T\| \,\|\varphi^{4m}\vec{u}\|_{Y^{m,2}({\theta Q})}
.\end{equation*}
By Lemma~\ref{uchi} with $\chi=\varphi^{2m}$,
\begin{equation*}\|\varphi^{4m}\vec{u}\|_{Y^{m,2}({\theta Q})}
\leq
\|\varphi^{2m}\vec{u}\|_{Y^{m,2}({\theta Q})}
+
\sum_{i=0}^{m-1} \frac{C}{\rho^{m-i}} \|\nabla^i (\varphi^{2m}u)\|_{L^2(\theta Q\setminus Q)}
.\end{equation*}
Using the Leibniz rule and arguing as before,
\begin{align*}\|\varphi^{4m}\vec{u}\|_{Y^{m,2}(\RR^d)}
&\leq
\|\varphi^{2m}\vec u\|_{Y^{m,2}({\theta Q})}
+
C
\sum_{i=0}^{m-1} \frac{C}{\rho ^{m-i} } \|\nabla^i \vec u\|_{L^2(\theta Q\setminus Q)}
.\end{align*}
By Corollary~\ref{cor:GNS:global}, $\|\varphi^{2m}\vec u\|_{Y^{m,2}({\theta Q})}\leq C\|\varphi^{2m}\vec u\|_{\dot W^{m,2}({\theta Q})}$.
By Young's inequality and formula~\eqref{eqn:step} we have
\begin{align*}
\frac{\lambda}{2}\|\varphi^{2m}\vec u\|_{\dot W^{m,2}({\theta Q})}^2
&\leq
C\|T\|^2 +\frac{\lambda}{4}\|\varphi^{2m}\vec u\|_{\dot W^{m,2}({\theta Q})}^2
\\&\qquad+
C\sum_{i=0}^{m-1} \frac{\|\nabla^i\vec{u}\|_{L^{2}(\theta Q\setminus Q)}^2}{\rho^{2m-2i}}
+\delta\|\varphi^{2m}\vec{u}\|^2_{L^2({\theta Q})}
.\end{align*}
Subtracting the second term on the right hand side and observing that $\|\nabla^m\vec u\|_{L^2(Q)}\leq \|\varphi^{2m}\vec u\|_{\dot W^{m,2}({\theta Q})}$ completes the proof.
\end{proof}

We wish to improve the Caccioppoli inequality by removing the intermediate derivatives (that is, $\nabla^k\vec u$ for $1\leq k\leq m-1$). The following theorem was proven in \cite[Theorem~18]{Bar16} in the case of balls rather than cubes; the proof in~\cite{Bar16} carries through with the obvious modifications.
\begin{theorem}\label{CacTh}
Let $Q\subset\RR^d$ be a cube with sides parallel to the coordinate axes. Let $1<\theta\leq 2$. Suppose that $\vec u\in W^{m,2}(\theta Q)$ is a function that satisfies the inequality
\begin{equation}\label{thhyp}
\int_{\vartheta Q} |\nabla^m\vec{u}|^2\leq \sum_{k=0}^{m-1} \frac{C_0}{((\mu-\vartheta)|Q|^{1/d})^{2m-2k}} \int_{\mu Q\setminus \vartheta Q}|\nabla^k\vec{u}|^2+F
\end{equation}
whenever $0<\vartheta<\mu<\theta$, for some $F>0$.

 Then $\vec{u}$ satisfies the stronger inequality
\begin{equation*}
\int_{Q}|\nabla^m\vec{u}|^2\leq \frac{C}{((\theta-1)|Q|^{1/d})^{2m}}\int_{\theta Q\setminus Q}|\vec{u}|^2+CF
\end{equation*}
for some constant $C$ depending only on $m$, the dimension $d$, and the constant $C_0$.

Furthermore, if $0\leq j\leq m$, then $\vec u$ satisfies
\begin{equation*}
\int_{Q}|\nabla^j\vec{u}|^2\leq \frac{C}{((\theta-1)|Q|^{1/d})^{2j}}\int_{\theta Q}|\vec{u}|^2+C|Q|^{(2m-2j)/d} F.
\end{equation*}
\end{theorem}

Now if we combine Lemma~\ref{CacLem} and Theorem~\ref{CacTh}, we obtain the desired Caccioppoli inequality in which we bound $|\nabla^m\vec{u}|^2$ without the intermediate gradient terms, as stated in the following corollary.

\begin{corollary}\label{HoC}
Let $L$ be an operator of order $2m$ of the form~\eqref{dfn:L} associated to coefficients $A$ that satisfy the weak G\aa{}rding inequality \eqref{wgi} and either the bound \eqref{eqn:intro:bound} or the bound~\eqref{eqn:intro:bound:Bochner}.

Let  $Q\subset\RR^d$ be an open cube with sides parallel to the coordinate axes, and let $1<\theta\leq 2$.
Let $\vec u\in Y^{m,2}(\theta Q)$. Let $T\in Y^{-m,2}(\theta Q)$.
Suppose that $L\vec u=T$ in $\theta Q$ in the sense of formula~\eqref{dfn:L}.

Then we have that
\begin{equation*}
\int_{Q}|\nabla^m\vec{u}|^2\leq \frac{C}{((\theta-1)|Q|^{1/d})^{2m}}
\int_{\theta Q\setminus Q}|\vec{u}|^2
+C\delta\int_{\theta Q}|\vec{u}|^2
+C\|T\|^2
\end{equation*}
and for all $j$ with $1\leq j\leq m-1$ we have that
\begin{equation}\label{clr}
\frac{1}{|Q|^{(2m-2j)/d}}\int_{Q}|\nabla^j\vec{u}|^2\leq \frac{C}{(\theta-1)^{2j}|Q|^{2m}}
\int_{\theta Q}|\vec{u}|^2
+C\delta\int_{\theta Q}|\vec{u}|^2
+C\|T\|^2
\end{equation}
where $C$ is a constant depending on the dimension $d$, the order $2m$ of $L$, the number $\lambda$ in the bound~\eqref{wgi}, and the number $\Lambda$ in the bound~\eqref{eqn:intro:bound} or~\eqref{eqn:intro:bound:Bochner}. Here $\|T\|=\|T\|_{Y^{-m,2}(\theta Q)}$ is the operator norm, that is, the smallest number such that $|\langle\vec\psi,T\rangle|\leq \|\vec\psi\|_{Y^{m,2}(\theta Q)} \|T\|$ for all $\vec\psi\in Y^{m,2}_0(\theta Q)$.
\end{corollary}

\begin{remark} If $m-d/2<j<m$ and $\delta=0$, then we can replace the term $\int_{\theta Q}|\vec{u}|^2$ in the bound~\eqref{clr} by $\int_{\theta Q\setminus Q}|\vec{u}|^2$ at a cost of some additional negative powers of $(\theta-1)$. See Section~\ref{Lpc}. \end{remark}

\section{Invertibility of~$L$}\label{FS}

In this section we will investigate boundedness and invertibility of the operator $L:Y^{m,p}(\RR^\dmn)\to Y^{-m,p}(\RR^\dmn)$. The argument for invertibility parallels that used in \cite[Lemma 3.4]{BorHLMP20p} in the second order case.

We remark that invertibility requires the Gårding inequality~\eqref{gi}, and not only the weaker Gårding inequality~\eqref{wgi} of Section~\ref{CacSec} and \cite{AusQ00}; thus, for the remainder of this paper, we will always assume the strong Gårding inequality~\eqref{gi}.

We will begin with boundedness of $L$ for a range of~$p$.

\begin{lemma}\label{lem:p:+}
Let $L$ be an operator of the form~\eqref{dfn:L} associated to coefficients $\mat A$ that satisfy either the bound \eqref{eqn:intro:bound} or the bound~\eqref{eqn:intro:bound:Bochner}.

If $\mat A$ satisfies the bound \eqref{eqn:intro:bound} then $(\frac{2d}{d+1},\frac{2d}{d-1})\subseteq\Pi_L=(\frac{d}{d+\mathfrak{a}-m},\frac{d}{m-\mathfrak{b}})$. If $\mat A$ satisfies the bound \eqref{eqn:intro:bound:Bochner} then $(\frac{2d}{d+1},\frac{2d}{d-1})\subseteq(\frac{d-1}{d-1+\mathfrak{a}-m},\frac{d-1}{m-\mathfrak{b}})\subseteq\Pi_L$, where $\Pi_L$ is as in Definition~\ref{dfn:ppm}.

If $p\in \Pi_L$ then the constants $\Lambda(p)$ in the bound~\eqref{eqn:elliptic:bound} depend only on $p$, $d$, $m$ and the constant $\Lambda$ in the bound \eqref{eqn:intro:bound} or~\eqref{eqn:intro:bound:Bochner}.
\end{lemma}
\begin{proof}
If $L$ satisfies the condition~\eqref{eqn:intro:bound} then $m\geq\mathfrak{a}>m-d/2$ and $m\geq\mathfrak{b}>m-d/2$. Observe that $m$, $d$ and $\mathfrak{a}$, $\mathfrak{b}$ are integers, and so $m\geq\mathfrak{a}\geq m-d/2+1/2$, $m\geq\mathfrak{b} \geq m-d/2+1/2$.
A straightforward computation yields that \begin{equation*}\biggl(\frac{2d}{d+1},\frac{2d}{d-1}\biggr)
\subseteq \biggl(\frac{d}{d+\mathfrak{a}-m}, \frac{d}{m-\mathfrak{b}}\biggr).\end{equation*}
Similarly, if $L$ satisfies the condition~\eqref{eqn:intro:bound:Bochner} then $m\geq\mathfrak{a}\geq m-(d-1)/2+1/2$ and $m\geq\mathfrak{b}\geq m-(d-1)/2+1/2$. Thus
\begin{align*}\biggl(\frac{2d}{d+1},\frac{2d}{d-1}\biggr)
&\subseteq
\biggl(\frac{2(d-1)}{d},\frac{2(d-1)}{d-2}\biggr)
\subseteq \biggl(\frac{d-1}{d-1+\mathfrak{a}-m}, \frac{d-1}{m-\mathfrak{b}}\biggr)
\\&\subseteq \biggl(\frac{d}{d+\mathfrak{a}-m}, \frac{d}{m-\mathfrak{b}}\biggr).\end{align*}

Suppose that $L$ satisfies the condition~\eqref{eqn:intro:bound}.
If $p\in (\frac{d}{d+\mathfrak{a}-m}, \frac{d}{m-\mathfrak{b}})$ then $\mathfrak{a}>m-d/p'$, $\mathfrak{b} >m-d/p$, and so if $\mathfrak{a}\leq |\alpha|\leq m$ and $\mathfrak{b}\leq |\beta|\leq m$ then $p'_\alpha$ and $p_\beta$ exist and are finite.
By formulas~\eqref{pk} and~\eqref{eqn:intro:bound},
\begin{equation*}\frac{1}{p_\beta}+\frac{1}{(p')_\alpha}+\frac{1}{2_{\alpha,\beta}}=1 .\end{equation*}
Thus by H\"older's inequality, for such $p$, $\alpha$, and~$\beta$,
\begin{equation*}\int_{\RR^d}\bigl|\partial^\alpha\varphi_j \overline{A^{j,k}_{\alpha,\beta}\partial^\beta u_k}
\bigr|
\leq
\|\partial^\alpha\varphi_j\|_{L^{(p')_\alpha}(\RR^d)} \|\partial^\beta{u}_k\|_{L^{p_\beta}(\RR^d)}
\|A^{j,k}_{\alpha,\beta}\|_{L^{2_{\alpha,\beta}}(\RR^d)}
\end{equation*}
which by the condition~\eqref{eqn:intro:bound} and the definition~\eqref{dfn:Y:norm} of $Y^{m,p}(\R^d)$ satisfies
\begin{equation*}\int_{\RR^d}\bigl|\partial^\alpha\varphi_j \overline{A^{j,k}_{\alpha,\beta}\partial^\beta u_k}
\bigr|
\leq
\Lambda
\|\vec\varphi\|_{Y^{m,p'}(\RR^d)} \|\vec{\psi}\|_{Y^{m,p}(\RR^d)}.\end{equation*}
Summing over $\alpha$, $\beta$, $j$ and~$k$ and using Definition~\ref{dfn:ppm} completes the proof.

Now suppose that $L$ satisfies the condition~\eqref{eqn:intro:bound:Bochner}.
If $p\in (\frac{d-1}{d-1+\mathfrak{a}-m}, \frac{d-1}{m-\mathfrak{b}})$ then $\mathfrak{a}>m-(d-1)/p'$, $\mathfrak{b} >m-(d-1)/p$, and so if $\mathfrak{a}\leq |\alpha|\leq m$ and $\mathfrak{b}\leq |\beta|\leq m$ then $\widetilde{p'}_\alpha$ and $\widetilde p_\beta$ exist and are finite.
Again
\begin{equation*}\frac{1}{\widetilde p_\beta}+\frac{1}{\widetilde {(p')}_\alpha}+\frac{1}{\widetilde 2_{\alpha,\beta}}=1 .\end{equation*}
Observe that
\begin{equation*}\int_{\RR^d}|\partial^\alpha\varphi_j \, A^{j,k}_{\alpha,\beta} \, \partial^\beta \psi_k|
\leq
\int_{-\infty }^\infty
\int_{\RR^{d-1}}|\partial^\alpha\varphi_j \, A^{j,k}_{\alpha,\beta} \, \partial^\beta \psi_k|
\,dx
\,dt
.\end{equation*}
Applying H\"older's inequality first in $\RR^{d-1}$ and then in $\RR$ yields that
\begin{equation*}\int_{\RR^d}|\partial^\alpha\varphi_j \, A^{j,k}_{\alpha,\beta} \, \partial^\beta \psi_k|
\leq
\Lambda
\|\partial^\alpha\varphi_j\|_{L_t^{p'}L_x^{\widetilde{(p')}_\alpha}(\RR^d)} \|\partial^\beta{u}_k\|_{L_t^{p}L_x^{\widetilde p_\beta}(\RR^d)}
.\end{equation*}
Applying Corollary~\ref{cor:GNS:slices} and summing completes the proof.
\end{proof}

We now establish invertibility of~$L$ for $p=2$. The main tool in the proof is the complex valued Lax-Milgram lemma, which we now state.
\begin{theorem}\cite[Theorem 2.1]{Bab70}\label{LM}
Let $H_1$ and $H_2$ be two Hilbert spaces, and let $B$ be a bounded sesquilinear form on $H_1\times H_2$ that is coercive in the sense that
$$\sup_{w\in H_1\setminus\{0\}}\frac{|B(w,v)|}{\|w\|_{H_1}}\geq\lambda\|v\|_{H_2}, \hspace{20pt}
\sup_{w\in H_2\setminus\{0\}}\frac{|B(u,w)|}{\|w\|_{H_2}}\geq\lambda\|u\|_{H_1}$$
for every $u\in H_1$, and $v\in H_2$, for some fixed $\lambda>0$.  Then for every linear functional $T$ defined on $H_2$ there is a unique $u_T\in H_1$ such that $B(v,u_T)=\overline{T(v)}$.  Furthermore $\|u_T\|_{H_1}\leq\frac{1}{\lambda}\|T\|_{H'_2}$.
\end{theorem}

\begin{lemma}\label{lem:L:invertible} Let $L$ be an operator of the form~\eqref{dfn:L} order $2m$ which satisfies the ellipticity condition~\eqref{gi} and such that $2\in \Pi_L$, where $\Pi_L$ is as in Definition~\ref{dfn:ppm}. Then $L$ is invertible with bounded inverse $Y^{m,2}(\RR^d)\to Y^{-m,2}(\RR^d)$.
\end{lemma}
\begin{proof}
Let $B(\vec{u},\vec{v})$ be the form given by
\begin{equation}\label{Bdef}
B(\vec{u},\vec{v})=\sum_{j,k=1}^N\sum_{|\alpha|\leq m}\sum_{|\beta|\leq m}\int_{\RR^d}\overline{\partial^\alpha u_j}A^{j,k}_{\alpha,\beta}\partial^\beta v_k
.\end{equation}
Notice that by formula~\eqref{gi} $B$ is a coercive sesquilinear operator on $Y^{m,2}(\RR^d)\times Y^{m,2}(\RR^d)$ in the sense of Theorem~\ref{LM}, while by Definition~\ref{dfn:ppm}, $B$ is bounded on $Y^{m,2}(\RR^d)\times Y^{m,2}(\RR^d)$ with the bound
\begin{equation}\label{bbound}
|B(\vec{u},\vec{v})|\leq \Lambda(2)\|\vec{u}\|_{Y^{m,2}(\RR^d)}\|\vec{v}\|_{Y^{m,2}(\RR^d)}.
\end{equation}

Let $T$ be an element of $Y^{-m,2}(\RR^d)$. Recall that we write bounded linear operators on $Y^{m,2}(\RR^d)$ as $\langle T,\,\cdot\,\rangle$.
Let $\vec u_T\in Y^{m,2}(\RR^d)$ be the unique element of $Y^{m,2}(\RR^d)$ given by the Lax-Milgram lemma, so
\begin{equation}\label{dfn:Linverse}
\sum_{j,k=1}^N\sum_{|\alpha|\leq m}\sum_{|\beta|\leq m}\int_{\RR^d} \overline{\partial^\alpha\varphi_j}\, A^{j,k}_{\alpha,\beta} \,\partial^\beta(u_T)_k
=\overline{\langle T,\varphi\rangle}
\end{equation}
for all $\varphi\in Y^{m,2}(\RR^d)$. Observe that by formula~\eqref{dfn:L}, $L\vec u_T=T$. By the boundedness property of the Lax-Milgram lemma, $\|\vec u_T\|_{Y^{m,2}(\RR^d)}\leq \frac{1}{\lambda}\|T\|_{Y^{-m,2}(\RR^d)}$, and by the uniqueness property in the Lax-Milgram lemma, $\vec u=\vec u_T$ is the only element of $Y^{m,2}(\RR^d)$ with $L\vec u=T$. Thus the operator $T\mapsto \vec u_T$ is well defined, bounded, linear, and an inverse to~$L$.
\end{proof}

We conclude this section by establishing invertibility of $L$ for a range of~$p$. In this case the main tool is \v{S}ne\v{\i}berg's lemma.
\begin{lemma}\label{snei}(\v{S}ne\v{\i}berg's lemma \cite[Theorem A.1]{AusBES19})
Let $\overline{X}=(X_0,X_1)$ and $\overline{Z}=(Z_0,Z_1)$ be interpolation couples, and $T\in\mathcal{B}(\overline{X},\overline{Z})$.  Suppose that for some $\theta^*\in(0,1)$ and some $\kappa>0$, the lower bound $\|Tx\|_{Z_{\theta^*}}\geq\kappa\|x\|_{X_{\theta^*}}$ holds for all $x\in X_{\theta^*}$.  Then the following are true.
\begin{itemize}
\item [(i)] Given $0<\epsilon<1/4$, the lower bound $\|Tx\|_{Z_\theta}\geq\epsilon\kappa\|x\|_{X_\theta}$ holds for all $x\in X_\theta$, provided that $|\theta-\theta^*|\leq\frac{\kappa(1-4\epsilon)\min\{\theta^*,1-\theta^*\}}{3\kappa+6M}$, where $M=\max_{j=0,1}\|T\|_{X_j\rightarrow Z_j}$.
\item [(ii)] If $T:X_{\theta^*}\longrightarrow Z_{\theta^*}$ is invertible, then the same is true for $T:X_\theta\longrightarrow Z_\theta$ if $\theta$ is as in~$(i)$.  The inverse mappings agree on $Z_\theta\cap Z_{\theta^*}$ and their norms are bounded by $\frac{1}{\epsilon\kappa}$.
\end{itemize}
\end{lemma}

\begin{lemma}\label{NPLp}
Let $L:Y^{m,2}(\R^d)\to Y^{-m,2}(\R^d)$ be bounded and invertible, and suppose that $L$ extends by density to a bounded operator $L:Y^{m,p}(\R^d)\to Y^{-m,p}(\R^d)$ for all $p$ in an open neighborhood of~$2$.

Let $\Upsilon_L$ be as in Definition~\ref{dfn:compatible}, that is, the set of all $p$ such that $L:Y^{m,p}(\RR^d)\to Y^{-m,p}(\RR^d)$ is bounded and compatibly invertible.

Then $\Upsilon_L$ is an interval, and there is a $\delta>0$ such that if $2-\delta<p<2+\delta$ then $p\in \Upsilon_L$.

In particular, these conditions are satisfied if $L$ is an operator of the form~\eqref{dfn:L} that satisfies the ellipticity condition~\eqref{gi} and such that $\Pi_L$ as given by Definition~\ref{dfn:ppm} contains an open neighborhood of~$2$. In this case $\delta$ depends only on $\Pi_L$ and the standard parameters.

\end{lemma}
\begin{proof} 
By assumption or by Lemma~\ref{lem:L:invertible},  $L:Y^{m,2}(\RR^d)\rightarrow Y^{-m,2}(\RR^d)$ is invertible. Thus $2\in \Upsilon_L$.

By \cite{Tri95}, $\dot{W}^{m,p}(\RR^d)$ forms a complex interpolation scale. The map which sends an element of $\dot{W}^{m,p}(\RR^d)$ to its unique representative in $Y^{m,p}(\RR^d)$ is a retract \cite[Lemma 7.11]{KalMM07}, and so we have that $Y^{m,p}(\RR^d)$ forms a complex interpolation scale.  Next, we have from \cite[Theorem 4.5.1]{BerL76} that the antidual space $Y^{-m,p}(\RR^d)$ also forms a complex interpolation scale.

A straightforward interpolation argument shows that if $L$ is bounded and compatibly invertible $Y^{m,p}(\RR^d)\rightarrow Y^{-m,p}(\RR^d)$ then $L$ is bounded and compatibly invertible $Y^{m,q}(\RR^d)\rightarrow Y^{-m,q}(\RR^d)$ whenever $q$ is between $p$ and~$2$, and so $\Upsilon_L$ is an interval.

Finally, by \v{S}ne\v{\i}berg's lemma, $L$ is invertible $Y^{m,q}(\RR^d)\to Y^{-m,q}(\RR^d)$ whenever $2-\delta<q<2+\delta$, where $\delta$ is as dictated by (i) from \v{S}ne\v{\i}berg's lemma. This completes the proof.
\end{proof}

\section{$L^p$ bounds on solutions and their gradients}\label{Lpc}

In \cite{Mey63}, Meyers established a reverse H\"older estimate; in the notation of the present paper, he established that if $L=-\nabla\cdot A\nabla$ is a second order divergence form operator without lower order terms, and if $Q$ is a cube, then for all $p$ and $q$ sufficiently close to~$2$ (and, in particular, for some $p>2$ and $q<2$) we have the estimate
\begin{equation*}
\|\nabla u\|_{L^{p}(Q)}
\leq
C|Q|^{1/p-1/q}
\|\nabla u\|_{L^q(2Q)}
+C\|Lu\|_{Y^{-1,p}(2Q)}
\end{equation*}
for all suitable functions~$u$. The exponent $q$ on the right hand side can be lowered if desired; see \cite[Section~9, Lemma~2]{FefS72} in the case of harmonic functions, and \cite[Lemma~33]{Bar16} for more general functions. Meyers's results can be generalized to second order systems (even nonlinear systems) without lower order terms (see \cite[Chapter~V]{Gia83}), or to higher order equations without lower order terms (see \cite{Cam80,AusQ00,Bar16}).

Theorem~\ref{umpm:intro} represents a generalization to the case of operators with lower order terms. It follows immediately from the next theorem and Lemma~\ref{NPLp}. We remark that the $m=1$ case of this theorem was essentially established in \cite[Section~3.1]{BorHLMP20p} and that the higher order case uses many of the same arguments.

\begin{theorem}\label{umpm}
Let $m\geq 1$ and $d\geq 2$ be integers. Let $L$ be an operator of order $2m$ of the form~\eqref{dfn:L} associated to coefficients $A$ that satisfy the G\aa{}rding inequality \eqref{gi} and either the bound \eqref{eqn:intro:bound} or the bound~\eqref{eqn:intro:bound:Bochner}.

Let $\Pi_L$ and $\Upsilon_L$ be as in Definitions \ref{dfn:ppm} and~\ref{dfn:compatible}. Let $p$, $\mu\in \Upsilon_L\cap\Pi_L$ with $p\geq 2$ and let $0<q\leq\infty$. Let $j$ and $\varpi$ be integers with $0\leq j\leq m$ and
$0\leq \varpi\leq \min(j,\mathfrak{b})$. If $p=2$, we impose the additional requirement that either $q\geq 2$ or $\varpi\geq 1$.

Let
$Q\subset\RR^d$ be a cube with sides parallel to the coordinate axes. Let $1<\theta\leq 2$.
Suppose that $\vec{u}\in Y^{m,\mu}(\theta Q)$ and that $L\vec{u}\in Y^{-m,p}(\theta Q)$ (in the sense that if $\vec\psi\in Y^{m,p'}_0(\theta Q)\cap Y^{m,\mu'}_0(\theta Q)$ then $|\langle L\vec u,\vec\psi\rangle_{\theta Q}| \leq C\|\vec\psi\|_{Y^{m,p'}(\theta Q)}$).

Then $\nabla^j u\in L^p(Q)$, and there exist positive constants $\kappa$ and $C$ depending on $p$, $q$, and the standard parameters such that
\begin{multline*}\frac{1}{|Q|^{(m-j)/d}}
\|\nabla^j u\|_{L^{p}(Q)}
\\\leq
\frac{C}{(\theta-1)^\kappa}\|L\vec u\|_{Y^{-m,p}(\theta Q)}
+ \frac{C|Q|^{1/p-1/q-(m-\varpi)/d}}{(\theta-1)^\kappa} \|\nabla^\varpi\vec u\|_{L^q(\theta Q\setminus Q)}
.\end{multline*}
\end{theorem}

Here $\mathfrak{b}$ is as in Definition~\ref{dfn:ppm}, that is, $\mathfrak{b} =\min\{|\beta|:A_{\alpha,\beta}^{j,k}(X)\neq 0$ for some $\alpha$, $j$, $k$, and~$X\}$.

\begin{remark}If $j>m-d/p$, we may of course immediately apply the Gagliardo-Nirenberg-Sobolev inequality (Lemma~\ref{evlem}) to bound $\|\nabla^j \vec u\|_{L^{p_j}(Q)}$; if $j<m-d/p$, then improved estimates on $\nabla^j\vec u$, such as local H\"older continuity, may be derived from further Sobolev space results such as Morrey's inequality.
\end{remark}

In the case of operators without lower order terms (in which case $\mathfrak{b}=m$), we may take $j=\varpi=m$; Theorem~\ref{umpm} then yields the same bounds as the classical inequality of Meyers (and the generalizations of \cite{Cam80,AusQ00,Bar16}).

We will also establish an estimate for functions $\vec u$ with $L\vec u\in Y^{-m,p}(\theta Q)$ for $p<2$ sufficiently close to~$2$.
\begin{theorem}\label{umpm:less}
Let $m\geq 1$ and $d\geq 2$ be integers. Let $L$ be an operator of order $2m$ of the form~\eqref{dfn:L} associated to coefficients $A$ that satisfy the G\aa{}rding inequality \eqref{gi} and either the bound \eqref{eqn:intro:bound} or the bound~\eqref{eqn:intro:bound:Bochner}.

Let $\Pi_L$ and $\Upsilon_L$ be as in Definitions \ref{dfn:ppm} and~\ref{dfn:compatible}. Let $p$, $\mu\in \Upsilon_L\cap\Pi_L$ and let $0<q\leq\infty$. Let $j$ be an integer with $0\leq j\leq m$.

Let
$Q\subset\RR^d$ be a cube with sides parallel to the coordinate axes. Let $1<\theta\leq 2$.
Suppose that $\vec{u}\in Y^{m,\mu}(\theta Q)$ and that $L\vec{u}\in Y^{-m,p}(\theta Q)$.

Then $\nabla^j u\in L^p(Q)$, and there exist positive constants $\kappa$ and $C$ depending on $p$, $q$, and the standard parameters such that
\begin{multline*}
\frac{1}{|Q|^{(m-j)/d}}
\|\nabla^j u\|_{L^{p}(Q)}
\\\leq
\frac{C}{(\theta-1)^\kappa}\|L\vec u\|_{Y^{-m,p}(\theta Q)}
+ \frac{C|Q|^{1/p-1/q}}{(\theta-1)^\kappa}
\sum_{i=\min(j,\mathfrak{b})}^m \frac{1}{|Q|^{(j-i)/d}}
\|\nabla^i\vec u\|_{L^q(\theta Q\setminus Q)}
.\end{multline*}
\end{theorem}

Given operators with lower order terms, Theorem~\ref{umpm} cannot be strengthened, as shown in the following example.

\begin{theorem}\label{thm:meyers:counterexample}
Let $\dmn\geq 3$, $m\geq 1$, $\mathfrak{a}\in (m-\pdmn/2,m]$, and $\mathfrak{b}\in (m-\pdmn/2,m)$ be nonnegative integers, and let $\varepsilon>0$.

Let $Q_0\subset\RR^\dmn$ be the cube of volume~$1$ centered at the origin.
Let $\widetilde A_{\alpha,\beta}$ be real nonnegative constant coefficients such that
\begin{equation*}(-\Delta)^m = (-1)^{m} \sum_{\abs\alpha=\abs\beta= m} \widetilde A_{\alpha,\beta}\partial^{\alpha+\beta},
\qquad
\widetilde A_{\alpha,\beta}=0\text{ if }|\alpha|<m\text{ or } |\beta|<m.\end{equation*}

Then there exists a linear operator $L$ of the form~\eqref{dfn:L} with $N=1$ associated to coefficients $A_{\alpha,\beta}=A_{\alpha,\beta}^{1,1}$ and a function $u$ such that
\begin{itemize}
\item $\|A_{\alpha,\beta}-\widetilde A_{\alpha,\beta}\|_{L^\infty(Q)}\leq \varepsilon$ for all $\abs\alpha\leq m$ and~$\abs\beta\leq m$,
\item The numbers $\mathfrak{a}$ and $\mathfrak{b}$ chosen above also satisfy the conditions~\eqref{eqn:tildealpha}--\eqref{eqn:tildebeta} given in Definition~\ref{dfn:ppm},
\item $Lu=0$ in~$Q$,
\item If $\widetilde C>0$ and $2<p<\infty$ then there is a cube $Q\subseteq Q_0$ with
\begin{equation*}
\|\nabla^m \vec u\|_{L^p(Q)}
\geq \widetilde C
\sum_{i=\mathfrak{b}+1}^{m}
|Q|^{1/p-1/2-(m-i)/\pdmn}
\|\nabla^i \vec u\|_{L^{2}(2Q)}
.\end{equation*}
\end{itemize}
\end{theorem}

Constant coefficient operators without lower order terms such as $(-\Delta)^m$ clearly satisfy the bounds \eqref{eqn:intro:bound} and~\eqref{eqn:intro:bound:Bochner} for some $\Lambda>0$. Extending $A_{\alpha,\beta}$ by zero, we see that by taking $\varepsilon$ small enough, we may ensure that $L$ satisfies the bound \eqref{eqn:intro:bound} with constant $\Lambda$ arbitrarily close to that of~$(-\Delta)^m$.

By an elementary (and very well known) argument using the Fourier transform, the operator $(-\Delta)^m$ satisfies the bound~\eqref{gi} for some $\lambda>0$. By Corollary~\ref{cor:GNS:global}, and again by taking $\varepsilon$ small enough, the operator $L$ satisfies the bound~\eqref{gi} with constant $\lambda$ arbitrarily close to that of~$(-\Delta)^m$.

We will prove Theorems~\ref{umpm} and~\ref{umpm:less} in Section~\ref{sec:meyers:proof}, and prove Theorem~\ref{thm:meyers:counterexample} in Section~\ref{sec:meyers:counter}.

\subsection{Proof of Theorems~\ref{umpm} and~\ref{umpm:less}}
\label{sec:meyers:proof}

We begin with the following variant of Lemmas~\ref{evlem}, \ref{evlem:t}, and~\ref{evlem:annulus} in the case where the exponents on each side are different.

\begin{lemma}\label{ux-m:step2}
Let $m$, $d\in\NN$, $d\geq 2$, $p\in[1,\infty )$, and let $j$, $k\in\NN_0$ satisfy
$0\leq j\leq k-1$ and $m-d/p< k\leq m$. Let $p_k=p_{m,d,k}$. Let $1<\theta\leq 2$. Let $\mu$ satisfy $0<1/\mu\leq\min(1,1/p+1/d)$.

Then there is a constant $C$ depending only on $p$, $d$, and~$m$ such that if $Q\subset\RR^d$ is a cube with sides parallel to the coordinate axes and $u\in W^{m,p}(\theta Q)$, then
\begin{align*}
\|\nabla^j u\|_{L^{p_{k}}(\theta Q)}
&\leq \sum_{i=j}^{m} C|Q|^{1/p-1/\mu-(m-k+j-i)/d}
\|\nabla^i u\|_{L^\mu(\theta Q)}
,\\
\|\nabla^j u\|_{L^{p_{k}}(\theta Q\setminus Q)}
&\leq \sum_{i=j}^{m} \frac{C|Q|^{1/p-1/\mu-(m-k+j-i)/d}}{(\theta-1)^{m-k+j-i+1}}
\|\nabla^i u\|_{L^\mu(\theta Q\setminus Q)}
.\end{align*}
If in addition $k>m-(d-1)/p$, then
\begin{align*}
\|\nabla^j u\|_{L_t^pL_x^{\widetilde p_k}(\theta Q)}
&\leq
\sum_{i=j}^{m} C|Q|^{1/p-1/\mu-(m-k+j-i)/d}
\|\nabla^i u\|_{L^\mu(\theta Q)}
,\\
\|\nabla^j u\|_{L_t^pL_x^{\widetilde p_k}(\theta Q\setminus Q)}
&\leq
\sum_{i=j}^{m} \frac{C|Q|^{1/p-1/\mu-(m-k+j-i)/d}}{(\theta-1)^{m-k+j-i+1}}
\|\nabla^i u\|_{L^\mu(\theta Q\setminus Q)}
.\end{align*}

\end{lemma}
\begin{proof}
By H\"older's inequality, it suffices to establish the listed bounds for the endpoint value $1/\mu=\min(1,1/p+1/d)$.
We will establish the last of the listed bounds; the arguments for the three preceding bounds are similar (in the first two cases with Lemmas~\ref{evlem} or \ref{evlem:t} in place of Lemma~\ref{evlem:annulus}).

By Lemma~\ref{evlem:annulus}, and because $k-j\geq 1$, we have that
\begin{align*}
\|\nabla^j u_k\|_{L^{ p_k}(\theta Q\setminus Q)}
&\leq
\sum_{i=j}^{m-1} \frac{C}{((\theta-1)|Q|^{1/d})^{m-k+j-i} } \|\nabla^i \vec u\|_{L^p(\theta Q\setminus Q)}
.\end{align*}
Recall that we have taken $\mu$ to satisfy $1/\mu=\min(1,1/p+1/d)$. Because $d\geq 2$, we have that
\begin{equation*}
0<\frac{1}{\mu_{m-1}}=\frac{1}{\mu}-\frac{1}{d}\leq \frac{1}{p}
\end{equation*}
(in particular, $\mu_{m-1}$ exists) and so by H\"older's inequality,
\begin{align*}
\|\nabla^j u_k\|_{L^{ p_k}(\theta Q\setminus Q)}
&\leq
\sum_{i=j}^{m-1} \frac{C}{((\theta-1)|Q|^{1/d})^{m-k+j-i} }
|Q|^{1/p-1/\mu+1/\pdmn}
\|\nabla^i \vec u\|_{L^{\mu_{m-1}}(\theta Q\setminus Q)}
.\end{align*}
Another application of Lemma~\ref{evlem:annulus} yields
\begin{align*}
\|\nabla^j u_k\|_{L^{ p_k}(\theta Q\setminus Q)}
&\leq
\sum_{i=j}^{m-1} \frac{C}{((\theta-1)|Q|^{1/d})^{m-k+j-i+1} }
|Q|^{1/p-1/\mu+1/\pdmn}
\|\nabla^{i} \vec u\|_{L^{\mu}(\theta Q\setminus Q)}
\end{align*}
as desired.
\end{proof}

Now, recall from Lemma~\ref{uchi} that if $u\in Y^{m,\mu}(\theta Q)$ then $u\chi\in Y^{m,\mu}(\theta Q)$ for all $\chi\in C^\infty_0(\theta Q)$.
By Definition~\ref{dfn:ppm}, if $\mu\in\Pi_L$ then $L(u\chi)\in Y^{-m,\mu}(\RR^\dmn)$.
We now show that under some circumstances, $L(u\chi)$ is also in $Y^{-m,p}(\RR^\dmn)$.

\begin{lemma}\label{ux-m}
Let $m\geq 1$ and $d\geq 2$ be integers. Let $L$ be an operator of the form~\eqref{dfn:L} for some coefficients $\mat{A}$ that satisfy either the bound \eqref{eqn:intro:bound} or the bound~\eqref{eqn:intro:bound:Bochner}.

If $\mat{A}$ satisfies the bound~\eqref{eqn:intro:bound}, let $p$, $\mu\in (\frac{d}{d+\mathfrak{a}-m}, \frac{d}{m-\mathfrak{b}})$. If $\mat{A}$ satisfies the bound~\eqref{eqn:intro:bound:Bochner}, let $p$, $\mu\in (\frac{d-1}{d-1+\mathfrak{a}-m}, \frac{d-1}{m-\mathfrak{b}})$. By Lemma~\ref{lem:p:+}, these ranges include $(\frac{2d}{d+1},\frac{2d}{d-1})$. In either case we additionally require that $1/\mu\leq  1/p+1/d$. 

Let
$Q\subset\RR^d$ be a cube with sides parallel to the coordinate axes. Let $1<\theta\leq 2$.
Let $\vec{u}\in Y^{m,\mu}(\theta Q)$ be such that $L\vec{u}\in Y^{-m,p}(\theta Q)$ (in the sense that if $\vec\psi\in Y^{m,p'}_0(\theta Q)\cap Y^{m,\mu'}_0(\theta Q)$ then $|\langle L\vec u,\vec\psi\rangle_{\theta Q}| \leq C\|\vec\psi\|_{Y^{m,p'}(\theta Q)}$).

Let $\chi\in C^\infty _c(\RR^d)$ be a test function with $0\leq \chi\leq 1$ such that $\chi=1$ in~$Q$ and $\chi=0$ outside~$\theta Q$. We extend $\vec{u}\chi$ by $0$ outside of $\theta Q$.

Then $L(\vec{u}\chi)$ extends to a bounded operator on~$Y^{m,p'}(\RR^d)$.

Furthermore, if $0\leq \varpi\leq\mathfrak{b}$, then there is a polynomial $\vec P$ of degree less than $\varpi$ and positive constants $C$ and $\kappa$ depending on the standard parameters such that
\begin{align*}
\|L((\vec{u}-\vec P)\chi)\|_{Y^{-m,p}(\RR^d)}
&\leq
\frac{C}{(\theta-1)^m}
\|L\vec u\|_{Y^{-m,p}(\theta Q)}
\\&\qquad+
\frac{CX\Lambda|Q|^{1/p-1/\mu}}{(\theta-1)^{\kappa}}
\sum_{i=\varpi}^m \frac{1}{|Q|^{(m-i)/d}}
\| \nabla^i \vec u\|_{L^{\mu}(\theta Q\setminus Q)}
\end{align*}
where $X=\max_{1\leq i\leq d} (\theta-1)^i|Q|^{i/d}\|\nabla^i \chi\|_{L^\infty (Q)}$.
\end{lemma}
We follow the convention that the zero function is a polynomial of negative degree; thus, if $\varpi=0$ then $P\equiv 0$. For any $p\in (\frac{d}{d+\mathfrak{a}-m}, \frac{d}{m-\mathfrak{b}})$ or $(\frac{d-1}{d-1+\mathfrak{a}-m}, \frac{d-1}{m-\mathfrak{b}})$, there is a $\mu$ in the same range with $\mu<p$ and with $1/\mu\leq 1/p+1/d$.

\begin{proof}[Proof of Lemma~\ref{ux-m}]
Let $\vec P$ be the polynomial of degree less than $\varpi$
with $\int_{\theta Q\setminus Q} \partial^\gamma(\vec u-\vec P)=0$ for all $|\gamma|<\varpi$. Because $\varpi\leq \mathfrak{b}$ and by defintion of~$\mathfrak{b}$,  $L\vec P=0$. The function $\chi \vec P$ is smooth and compactly supported and so $L(\chi\vec P)\in Y^{-m,p}(\RR^d)$. Thus, we need only show that $L((\vec u-\vec P)\chi)\in Y^{-m,p}(\RR^d)$ and establish an appropriate bound on its norm.
For notational convenience we will take $\vec P=0$.

Recall that $Y^{-m,p}(\RR^d)$ is the antidual space to $Y^{m,p'}(\RR^d)$. So to show that $L(\chi\vec u)\in Y^{-m,p}(\RR^d)$, we need only bound $\langle L(\chi\vec u),\vec\varphi\rangle$ for all $\vec\varphi$ in $Y^{m,p'}(\RR^d)$. By density we may assume that $\vec\varphi\in Y^{m,\mu'}(\RR^d)$, so by Lemma~\ref{uchi}, $\langle L(\vec u\chi),\varphi\rangle$ represents an absolutely convergent integral.

Let $\vec\varphi$ be (a representative of) an element of $Y^{m,p'}(\RR^d)\cap Y^{m,\mu'}(\RR^d)$.  By the weak definition~\eqref{dfn:L} of~$L$,
\begin{align*}\overline{\langle L(\vec{u}\chi),\vec\varphi\rangle}
&=
\int_{\theta Q}
\sum_{j,k=1}^N
\sum_{|\alpha|\leq m} \sum_{|\beta|\leq m}
\overline{\partial^\alpha\varphi_j}\, A_{\alpha,\beta}^{j,k}\,\partial^\beta (\chi u_k)
.\end{align*}
Let $\vec\psi=\vec\varphi-\vec R$, where $\vec R$ is the polynomial of degree less than $\mathfrak{a}$ with $\int_{\theta Q} \partial^\gamma(\vec \varphi-\vec R)=0$ for all $|\gamma|<\mathfrak{a}$.
Then $L^*\vec R=0$. Therefore,
\begin{align*}\overline{\langle L(\vec{u}\chi),\vec\varphi\rangle}
&=
\overline{\langle L(\vec{u}\chi),\vec\varphi-\vec R\rangle}
=
\overline{\langle L(\vec{u}\chi),\vec\psi\rangle}
=
\sum_{j,k=1}^N
\sum_{|\alpha|\leq m} \sum_{|\beta|\leq m} \int_{\theta Q}
\overline{\partial^\alpha\psi_j}\, A_{\alpha,\beta}^{j,k}\,\partial^\beta (\chi u_k)
.\end{align*}

We remark on the symmetry of our situation: $\vec \psi\in Y^{m,p'}(\theta Q)$, $\vec u\in Y^{m,\mu}(\theta Q)$,  $\int_{\theta Q} \partial^\gamma\vec\psi=0$ if $|\gamma|<\mathfrak{a}$, and $\int_{\theta Q\setminus Q} \partial^\delta u=0$ if $|\delta|<\varpi$.

By the Leibniz rule,
\begin{align*}
\overline{\langle L(\vec{u}\chi),\vec\varphi\rangle}
&=
\sum_{j,k=1}^N
\sum_{|\alpha|\leq m} \sum_{|\beta|\leq m} \int_{\theta Q}
\overline{\partial^\alpha(\psi_j\bar \chi)}\, A_{\alpha,\beta}^{j,k}\,\partial^\beta u_k
\\&\qquad+
\sum_{j,k=1}^N \sum_{|\alpha|\leq m} \sum_{|\beta|\leq m} \sum_{\gamma<\beta}\int_{\theta Q\setminus Q}
\frac{\beta!}{\gamma!(\beta-\gamma)!} \overline{\partial^\alpha\psi_j} A_{\alpha,\beta}^{j,k} \partial^\gamma u_k \,\partial^{\beta-\gamma}\chi
\\
&\qquad
-\sum_{j,k=1}^N \sum_{|\alpha|\leq m}\sum_{|\beta|\leq m}\sum_{\delta<\alpha}\int_{\theta Q\setminus Q} \frac{\alpha!}{\delta!(\alpha-\delta)!}
\overline{\partial^\delta\psi_j} \, \partial^{\alpha-\delta}\chi \, A_{\alpha,\beta}^{j,k}\, \partial^\beta u_k
.\end{align*}
Recall from Lemma~\ref{uchi} that $\bar\chi\vec\psi\in Y^{m,p'}_0(\theta Q)\cap Y^{m,\mu'}_0(\theta Q)$.
By the weak definition~\eqref{dfn:L} of~$L$, we have that
\begin{equation*}
\sum_{j,k=1}^N
\sum_{|\alpha|\leq m} \sum_{|\beta|\leq m} \int_{\theta Q}
\overline{\partial^\alpha(\psi_j\bar \chi)}\, A_{\alpha,\beta}^{j,k}\,\partial^\beta u_k
=
\overline{\langle L\vec{u},\bar\chi\vec\psi\rangle}_{\theta Q}
.\end{equation*}
By definition of~$Y^{-m,p}$,
\begin{equation*}|{\langle L\vec{u},\bar\chi\vec\psi\rangle}_{\theta Q}
|
\leq
\|L\vec u\|_{Y^{-m,p}(\theta Q)}\|\bar\chi\vec\psi\|_{Y^{m,p'}_0(\theta Q)}
.\end{equation*}
By Lemmas~\ref{uchi} and~\ref{evlem},
\begin{equation*}\|\bar\chi\vec\psi\|_{Y^{m,p'}_0(\theta Q)}
\leq
\sum_{i=0}^{m} \frac{CX}{(\theta-1)^{m-i}}\frac{1}{|Q|^{(m-i)/d}} \|\nabla^i \vec\psi\|_{L^{p'}(\theta Q)}.
\end{equation*}
By the Poincar\'e inequality, and because $\nabla^i\vec\psi=\nabla^i\vec\varphi$ for all $i\geq \mathfrak{a}$,
\begin{equation*}\|\bar\chi\vec\psi\|_{Y^{m,p'}_0(\theta Q)}
\leq
\frac{CX}{(\theta-1)^m}
\sum_{i=\mathfrak{a}}^{m} \frac{1}{|Q|^{(m-i)/d}} \|\nabla^i \vec\varphi\|_{L^{p'}(\theta Q)}.
\end{equation*}
Recall that $p>\frac{d}{d-m+\mathfrak{a}}$.
If $i\geq \mathfrak{a}$, then $1/p'-(m-i)/d>0$ and so by formula~\eqref{pk} $(p')_i$ is well defined and finite. Thus by H\"older's inequality
\begin{equation*}\|\bar\chi\vec\psi\|_{Y^{m,p'}_0(\theta Q)}
\leq
\frac{CX}{(\theta-1)^m}
\|\vec\varphi\|_{Y^{m,p'}(\theta Q)}.
\end{equation*}
Thus
\begin{align*}|\langle L(\vec{u}\chi),\vec\varphi\rangle|
&\leq
\frac{C}{(\theta-1)^m}
\|L\vec u\|_{Y^{-m,p}(\theta Q)}
\|\vec\varphi\|_{Y^{m,p'}(\theta Q)}
\\&\qquad+
\biggl|
\sum_{j,k=1}^N \sum_{|\alpha|\leq m} \sum_{|\beta|\leq m} \sum_{\gamma<\beta}\int_{\theta Q\setminus Q}
\frac{\beta!}{\gamma!(\beta-\gamma)!}
\overline{\partial^\alpha\psi_j} A_{\alpha,\beta}^{j,k} \partial^\gamma u_k \,\partial^{\beta-\gamma}\chi
\biggr|
\\
&\qquad
+\biggl|\sum_{j,k=1}^N \sum_{|\alpha|\leq m}\sum_{|\beta|\leq m}\sum_{\delta<\alpha}\int_{\theta Q\setminus Q} \frac{\alpha!}{\delta!(\alpha-\delta)!}
\overline{\partial^\delta\psi_j} \, \partial^{\alpha-\delta}\chi \, A_{\alpha,\beta}^{j,k}\, \partial^\beta u_k\biggr|
.\end{align*}

We will now bound the integrals over~$\theta Q\setminus Q$.

Suppose that the coefficients $\mat{A}$ satisfy the condition~\eqref{eqn:intro:bound:Bochner}. Let $\alpha$ and $\beta$ be such that $A_{\alpha,\beta}^{j,k}$ is not identically equal to zero. By assumption on $\mu$ and~$p$, this means that $\widetilde \mu_\beta$, $\widetilde p_\beta$, $\widetilde{(p')}_\alpha$, and~$\widetilde {(\mu')}_\alpha$ exist and are finite.
By H\"older's inequality in $\RR^{d-1}$ and then in~$\RR$,
\begin{multline*}
\biggl|\int_{\theta Q\setminus Q}
\overline{\partial^\alpha\psi_j} A_{\alpha,\beta}^{j,k} \partial^\gamma u_k \,\partial^{\beta-\gamma}\chi
\biggr|
\\\leq
\|{\partial^\alpha\psi_j}\|_{L_t^{p'}L_x^{\widetilde{(p')}_\alpha} (\theta Q\setminus Q)}
\| \partial^{\beta-\gamma}\chi\|_{L^\infty (\theta Q\setminus Q)}
\| A_{\alpha,\beta}^{j,k}\|_{L^{\widetilde 2_{\alpha,\beta}}(\theta Q\setminus Q)}
\| \partial^\gamma u_k\|_{L_t^pL_x^{\widetilde p_\beta}(\theta Q\setminus Q)}
\end{multline*}
and
\begin{multline*}
\biggl|\int_{\theta Q\setminus Q}
\overline{\partial^\delta\psi_j} \, \partial^{\alpha-\delta}\chi \, A_{\alpha,\beta}^{j,k}\, \partial^\beta u_k
\biggr|
\\\leq
\|{\partial^\delta\psi_j}\|_{L_t^{\mu'}L_x^{\widetilde{(\mu')}_\alpha} (\theta Q\setminus Q)}
\| \partial^{\alpha-\delta}\chi\|_{L^\infty (\theta Q\setminus Q)}
\| A_{\alpha,\beta}^{j,k}\|_{L^{\widetilde 2_{\alpha,\beta}}(\theta Q\setminus Q)}
\| \partial^\beta u_k\|_{L_t^\mu L_x^{\widetilde \mu_\beta}(\theta Q\setminus Q)}
.\end{multline*}

Because $|\alpha|\geq \mathfrak{a}$ we have that $\partial^\alpha\vec\psi=\partial^\alpha\vec\varphi$.
By Lemma~\ref{evlem:t} with $j=k=|\alpha|$, the definitions~\eqref{dfn:Y:norm} and~\eqref{pk} of $Y^{m,p'}$ and $p_\alpha$, and H\"older's inequality,
\begin{equation*}\|{\partial^\alpha\psi_j}\|_{L_t^{p'}L_x^{\widetilde{(p')}_\alpha} (\theta Q\setminus Q)}
\leq C\| \vec \varphi\|_{Y^{m,p'}(\theta Q)}.
\end{equation*}
By Lemma~\ref{evlem:annulus} with $j=k=|\beta|$,
\begin{equation*}\| \partial^\beta u_k\|_{L_t^\mu L_x^{{\widetilde \mu}_\beta}(\theta Q\setminus Q)}
\leq
\sum_{i=|\beta|}^m \frac{C}{(\theta-1)^{m-i}|Q|^{(m-i)/d}}
\| \nabla^i u_k\|_{L^{\mu}(\theta Q\setminus Q)}
.\end{equation*}
By Lemma~\ref{ux-m:step2} with $j=|\gamma|<k=|\beta|$,
\begin{equation*}\|\partial^\gamma u_k\|_{L_t^pL_x^{\widetilde p_\beta}(\theta Q\setminus Q)}
\leq
\sum_{i=|\gamma|}^{m} \frac{C|Q|^{1/p-1/\mu-(m-|\beta|+|\gamma|-i)/d}}
{(\theta-1)^{m-|\beta|+|\gamma|-i+1}}
\|\nabla^i u_k\|_{L^\mu(\theta Q\setminus Q)}
\end{equation*}
and by Lemma~\ref{lem:annulus},
\begin{equation*}\|\partial^\gamma u_k\|_{L_t^pL_x^{\widetilde p_\beta}(\theta Q\setminus Q)}
\leq
\sum_{i=\varpi}^{m} \frac{C|Q|^{1/p-1/\mu-(m-|\beta|+|\gamma|-i)/d}}
{(\theta-1)^{m-|\beta|+1}}
\|\nabla^i u_k\|_{L^\mu(\theta Q\setminus Q)}
.\end{equation*}
Observe that $1/p'\leq 1/\mu'+1/d$; thus, by Lemma~\ref{ux-m:step2} and the Poincar\'e inequality with $j=|\delta|<k=|\alpha|$, and with $p$, $\mu$, $u$ replaced by $\mu'$, $p'$,~$\psi$, we have that
\begin{equation*}\|\partial^\gamma \psi\|_{L_t^{\mu'}L_x^{\widetilde {(\mu')}_\alpha}(\theta Q)}
\leq
\sum_{i=\mathfrak{a}}^{m} {C|Q|^{1/p-1/\mu-(m-|\alpha|+|\delta|-i)/d}}
\|\nabla^i \psi\|_{L^{p'}(\theta Q)}.
\end{equation*}
Because $\nabla^i\vec\psi=\nabla^i\vec\varphi$ for all $i\geq \mathfrak{a}$, and by H\"older's inequality, we have that
\begin{align*}\|\partial^\delta \psi\|_{L_t^{\mu'}L_x^{\widetilde {(\mu')}_\alpha}(\theta Q)}
&\leq
\sum_{i=\mathfrak{a}}^{m} {C|Q|^{1/p-1/\mu+(|\alpha|-|\delta|)/d}}
\|\nabla^i \vec\varphi\|_{L^{(p')_i}(\theta Q)}
\\&\leq
{C|Q|^{1/p-1/\mu+(|\alpha|-|\delta|)/d}}
\|\vec\varphi\|_{Y^{m,p'}(\theta Q)}
.
\end{align*}
Combining all of the above estimates and the definitions of $X$ and~$\Lambda$, we see that
\begin{align*}|\langle L(\vec{u}\chi),\vec\varphi\rangle|
&\leq
\frac{C}{(\theta-1)^m}
\|L\vec u\|_{Y^{-m,p}(\theta Q)}
\|\vec\varphi\|_{Y^{m,p'}(\theta Q)}
\\&\qquad+
\| \vec \varphi\|_{Y^{m,p'}(\theta Q)}
\sum_{i=\varpi}^{m} \frac{CX\Lambda|Q|^{1/p-1/\mu-(m-i)/d}}
{(\theta-1)^{m+1}}
\|\nabla^i \vec u\|_{L^\mu(\theta Q\setminus Q)}
.\end{align*}
This completes the proof in the case where $\mat{A}$ satisfies the condition~\eqref{eqn:intro:bound:Bochner}.

If instead $\mat{A}$ satisfies the condition~\eqref{eqn:intro:bound}, a similar argument with Lemma~\ref{evlem} in place of Lemma~\ref{evlem:t} establishes the same bound.
\end{proof}

From Lem\-ma~\ref{ux-m} we have a bound on $L(\vec{u}\chi)$.  We may now prove the following result; this is Theorem~\ref{umpm:less} in the case $q=\mu$.

\begin{lemma}\label{umpm2}
Let $m$, $d$, $L$, $p$, $\mu$, $Q$, $\theta$, $u$, and~$\varpi$ be as in Lemma~\ref{ux-m}.

Suppose in addition that $p$, $\mu\in \Upsilon_L\cap\Pi_L$, where $\Pi_L$ and $\Upsilon_L$ are as in Definitions \ref{dfn:ppm} and~\ref{dfn:compatible}. 

Then there is a constant $C$ depending only on $p$ and $L$ such that, for all $j$ with $\varpi\leq j\leq m$, we have that
\begin{align*}\frac{1}{|Q|^{(m-j)/d}}
\|\nabla^j\vec{u}\|_{L^{p}(Q)}
&\leq
\frac{C}{(\theta-1)^m}
\|L\vec u\|_{Y^{-m,p}(\theta Q)}
\\&\qquad
+
\frac{C\Lambda|Q|^{1/p-1/\mu}}{(\theta-1)^{\kappa}}
\sum_{i=\varpi}^m \frac{1}{|Q|^{(m-i)/d}}
\| \nabla^i \vec u\|_{L^{\mu}(\theta Q\setminus Q)}
.\end{align*}
If $2-\delta<p<2+\delta$, where $\delta$ is the number in Lemma~\ref{NPLp}, then $C$ may be taken depending only on~$p$ and the standard parameters.
\end{lemma}

\begin{proof}
Let $\chi\in C^\infty _c(\theta Q)$ be as in Lemma~\ref{ux-m}; we may require that the parameter $X$ be bounded depending only on $m$ and~$\dmn$. We extend $(\vec u-\vec P)\chi$ by zero, where $P$ is the polynomial in Lemma~\ref{ux-m}.

By the definition of~$\Upsilon_L$, $L$ is invertible $Y^{m,p}(\R^d)\to Y^{-m,p}(\R^d)$, $Y^{m,\mu}(\R^d)\to Y^{-m,\mu}(\R^d)$, and~$Y^{m,2}(\R^d)\to Y^{-m,2}(\R^d)$.

Furthermore, if $T\in Y^{-m,p}(\R^d)\cap Y^{-m,2}(\R^d)$, then $L^{-1}T\in Y^{m,p}(\R^d)\cap Y^{m,2}(\R^d)$. Observe that we may approximate elements of $Y^{-m,p}(\R^d)\cap Y^{-m,\mu}(\R^d)$ by elements of $Y^{-m,p}(\R^d)\cap Y^{-m,\mu}(\R^d)\cap Y^{-m,2}(\R^d)$; thus, by density, if $T\in Y^{-m,p}(\R^d)\cap Y^{-m,\mu}(\R^d)$, then $L^{-1}T\in Y^{m,p}(\R^d)\cap Y^{m,\mu}(\R^d)$ (even if $T\notin Y^{-m,2}(\R^d)$).

Thus, because $(\chi(\vec u-\vec P))\in Y^{m,\mu}(\R^d)$, we have that
\begin{equation*}\chi(\vec u-\vec P) = L^{-1}(L(\chi(\vec u-\vec P))).\end{equation*}
Since $L(\chi(\vec u-\vec P))\in Y^{-m,p}(\R^d)\cap Y^{-m,\mu}(\R^d)$, we have that $\chi(\vec u-\vec P)\in  Y^{m,p}(\R^d)$.
By boundedness of $L^{-1}:Y^{m,p}(\R^d)\to Y^{-m,p}(\R^d)$, we have that
\begin{equation*}\|\chi(\vec u-\vec P)\|_{Y^{m,p}(\R^d)} \leq C(p,L)
\|L(\chi(\vec u-\vec P))\|_{Y^{-m,p}(\R^d)}.\end{equation*}
By Lemma~\ref{ux-m},
\begin{align*}
\|\chi(\vec u-\vec P)\|_{Y^{m,p}(\R^d)}
&\leq
\frac{C}{(\theta-1)^m}
\|L\vec u\|_{Y^{-m,p}(\theta Q)}
\\&\qquad
+
\sum_{i=\varpi}^m
\frac{C\Lambda|Q|^{1/p-1/\mu}}{(\theta-1)^{\kappa}|Q|^{(m-i)/d}}
\| \nabla^i u_k\|_{L^{\mu}(\theta Q\setminus Q)}
.\end{align*}
If $j> m-d/p$, then $p_j$ exists and by H\"older's inequality \begin{equation*}|Q|^{(j-m)/d}\|\nabla^j (\chi(\vec u-\vec P))\|_{L^{p}(Q)}\leq \|\nabla^j (\chi(\vec u-\vec P))\|_{L^{p_j}(Q)}\leq \|\chi(\vec u-\vec P)\|_{Y^{m,p}(Q)}.\end{equation*}
If $\varpi\leq j\leq m-d/p$, recall that $\chi(\vec u-\vec P)$ is compactly supported; by the Poincar\'e inequality, we again have that \begin{equation*}|Q|^{(j-m)/d}\|\nabla^j (\chi(\vec u-\vec P))\|_{L^{p}(Q)}\leq \|\nabla^m (\chi(\vec u-\vec P))\|_{L^{p}(Q)}\leq \|\chi(\vec u-\vec P)\|_{Y^{m,p}(Q)}.\end{equation*}
Therefore,
\begin{equation*}
|Q|^{(j-m)/d}
\|\nabla^j(\chi(\vec u-\vec P))\|_{L^{p}(Q)} \leq
\|u\|_{Y^{m,p}(Q)} \leq
\|\chi(\vec u-\vec P)\|_{Y^{m,p}(\R^d)}.\end{equation*}
Because $j\geq \varpi$, $\nabla^j \vec u=\nabla^j(\chi(\vec u-\vec P))$ in~$Q$
and the proof is complete.
\end{proof}

We may combine Lemma~\ref{umpm2} with the Caccioppoli inequality (Lemma~\ref{HoC}) to prove Theorem~\ref{umpm} in the case $q=2$.

\begin{lemma}\label{umpm3}
Let $m$, $d$, $L$, $p$, $\mu$, $Q$, $\theta$, $u$, $\varpi$, and~$j$ be as in Lemma~\ref{umpm2}, that is, that they are as in Lemma~\ref{ux-m} with $p$, $\mu\in \Upsilon_L\cap\Pi_L$ and $\varpi\leq j\leq m$.

Suppose in addition that $p\geq 2$.

Then there is a positive constant $\kappa$ depending only on the standard parameters and a positive constant $C$ depending on $p$ and~$L$ such that, if $0\leq j\leq m$, then
\begin{align*}
\frac{1}{|Q|^{(m-j)/d}}
\|\nabla^j\vec{u}\|_{L^{p}(Q)}
&\leq
\frac{C}{(\theta-1)^\kappa}
\|L\vec u\|_{Y^{-m,p}(\theta Q)}
\\&\qquad
+
\frac{C|Q|^{1/p-1/2-(m-\varpi)/d}}{(\theta-1)^{\kappa}}
\|\nabla^\varpi\vec u\|_{L^{2}(\theta Q\setminus Q)}
.\end{align*}
If $2\leq p<2+\delta$, where $\delta$ is the number in Lemma~\ref{NPLp}, then $C$ may be taken depending only on~$p$ and the standard parameters.
\end{lemma}

\begin{proof}
Let $\theta_0=1$, $\theta_3=\theta$, and $\theta_3-\theta_2=\theta_2-\theta_1=\theta_1-\theta_0=(\theta-1)/3$.
Choose $\mu=2$. Lemma~\ref{umpm2} yields that
\begin{align*}\frac{1}{|Q|^{(m-j)/d}}
\|\nabla^j\vec{u}\|_{L^{p}(\theta_1 Q)}
&\leq
\frac{C}{(\theta-1)^m}
\|L\vec u\|_{Y^{-m,p}(\theta_2 Q)}
\\&\qquad
+
\frac{C\Lambda|Q|^{1/p-1/\mu}}{(\theta-1)^{\kappa}}
\sum_{i=\varpi}^m \frac{1}{|Q|^{(m-i)/d}}
\| \nabla^i \vec u\|_{L^{2}(\theta_2 Q\setminus \theta_1 Q)}
.\end{align*}
Let $\vec P$ be a polynomial of degree less than $\varpi\leq \min(j,\mathfrak{b})$ such that $\fint_{\theta Q\setminus Q} \partial^\gamma (\vec u-\vec P)=0$ for all $|\gamma|<\varpi$. Observe that $L\vec u=L(\vec u-\vec P)$ and $\nabla^j\vec u=\nabla^j(\vec u-\vec P)$. Applying Corollary~\ref{HoC} to $\vec u-\vec P$ and a covering argument yields that
\begin{align*}
\frac{1}{|Q|^{(m-j)/d}}
\|\nabla^j\vec{u}\|_{L^{p}(Q)}
&\leq
\frac{C}{(\theta-1)^m}
\|L\vec u\|_{Y^{-m,p}(\theta Q)}
+
\frac{C|Q|^{1/p-1/2}}{(\theta-1)^{\kappa}}
\|L\vec u\|_{Y^{-m,2}(\theta Q\setminus Q)}
\\&\qquad
+
\frac{C|Q|^{1/p-1/2-m/d}}{(\theta-1)^{\kappa+m}}
\|\vec u-\vec P\|_{L^{2}(\theta Q\setminus Q)}
.\end{align*}
Because $p\geq2$, by H\"older's inequality
$|Q|^{1/p-1/2}\|L\vec u\|_{Y^{-m,2}(\theta Q)}\leq C\|L\vec u\|_{Y^{-m,p}(\theta Q)}$. By Lemma~\ref{lem:annulus}, we may replace $\|\vec u-\vec P\|_{L^{2}(\theta Q\setminus Q)}$ by $|Q|^{\varpi/d}\|\nabla^\varpi \vec u\|_{L^{2}(\theta Q\setminus Q)}$.
Redefining $\kappa$ completes the proof.
\end{proof}

\begin{remark}\label{rmk:HoC} If $p=2$, Lemma~\ref{umpm3} still represents an improvement over the Caccioppoli inequality (Corollary~\ref{HoC}) in that, if $m-d/2<j<m$, then we can bound $\|\nabla^j u\|_{L^2(Q)}$ by
$\|\vec u\|_{L^{2}(\theta Q\setminus Q)}$ and not $
\|\vec u\|_{L^{2}(\theta Q)}$.
\end{remark}

\begin{remark}If $p=2$ and $\varpi\geq 1$, then by Lemmas~\ref{umpm3} and~\ref{evlem:annulus},
\begin{align*}
\frac{1}{|Q|^{(m-j)/d}}
\|\nabla^j\vec{u}\|_{L^{p}(Q)}
&\leq
\frac{C\|L\vec u\|_{Y^{-m,p}(\theta Q)}
}{(\theta-1)^\kappa}
%\\&\qquad
+
\frac{C|Q|^{1/p-1/2-(m-\varpi+1)/d}}{(\theta-1)^{\kappa}}
\|\nabla^{\varpi-1}\vec u\|_{L^{2}(\theta Q\setminus Q)}
\\&\leq
\frac{C\|L\vec u\|_{Y^{-m,p}(\theta Q)}
}{(\theta-1)^\kappa}
%\\&\qquad
+
\frac{C|Q|^{1/p-1/2-(m-\varpi)/d}}{(\theta-1)^{\kappa+1}}
\|\nabla^{\varpi}\vec u\|_{L^{q}(\theta Q\setminus Q)}
\end{align*}
for $q$ satisfying $1/\mu=1/2+1/\pdmn$; notice that this $q$ satisfies $q<2$.
\end{remark}

We have now established that Theorem~\ref{umpm:less} is valid if $q=\mu$, and that Theorem~\ref{umpm} is valid if $q=2$ or if $\varpi\geq 1$ and $q$ takes a specific value less than~$2$. In particular, these theorems are valid for at least one $q<p$. By H\"older's inequality, these theorems are valid for all $q\geq p$. The following lemma will complete the proof by establishing validity for all positive but smaller~$q$.

\begin{lemma}\label{lem:subaverage}
Let $\dmn\geq 2$ and $0\leq \varpi\leq n\leq m$ be integers.
Let $Q\subset\R^\dmn$ be a cube and let $1<\theta\leq 2$.

For each $i$ with $\varpi\leq i\leq m$, let $p_i$, $u_i$ satisfy $0<p_i\leq \infty$ and $u_i\in L^{p_i}(\theta Q)$; if in addition $\varpi\leq i\leq n$, let $\widehat q_i$ satisfy $0<\widehat q_i< p_i$.

Suppose that, whenever $1\leq\vartheta<\zeta\leq\theta$, we have the bound
\begin{equation}
\label{epn:local-bound:steq1}
\sum_{j=\varpi}^m \|u_j\|_{L^{p_j}(\vartheta Q)}
\leq
\frac{F}{(\zeta-\vartheta)^{\kappa}}
+\frac{c_0}{(\zeta-\vartheta)^{\kappa}}
\sum_{i=\varpi}^n
\|u_i\|_{L^{\widehat q_i}(\tau Q\setminus \vartheta Q)}
\end{equation}
for some nonnegative constants~$c_0$, $\kappa$ and~$F$ independent of $\zeta$ and~$\vartheta$.

Then for every set of numbers $q_i$ with $0<q_i\leq\widehat q_i$, there are some constants $C$ and $\widetilde\kappa$, deqending only on the $q_i$s, $\widehat q_i$s, $p_i$s, $c_0$, and~$\kappa$, such that
\begin{align*}
\sum_{j=\varpi}^m \|u_j\|_{L^{p_j}(\vartheta Q)}
&\leq
\frac{C}{(\theta-1)^{\widetilde\kappa}}
\biggl(F+
\sum_{i=\varpi}^n
\doublebar{u_i}_{L^{q_i}(\theta Q\setminus Q)}\biggr)
.\end{align*}
\end{lemma}

\begin{proof} If $c_0=0$ we are done, so throughout we may assume $c_0>0$. We are also done if $q_i=\widehat q_i$ for all~$i$; we will consider the case where $q_i<\widehat q_i$ for at least one~$i$. In the present paper we will only need the case where $q_i=q$, $\widehat q_i=\widehat q$ for some $q$, $\widehat q$ independent of~$i$, but for completeness we present the general case.

Let $1=\vartheta_0<\vartheta_1<\vartheta_2<\dots$ for some $\vartheta_\ell \in [1,\theta)$ to be chosen momentarily, and let $Q_\ell =\vartheta_\ell  Q$. Let $A_\ell =Q_{\ell+1} \setminus Q_{\ell}$. If $\varpi\leq i\leq n$, let
\begin{equation*}\tau_i = \frac{1/\widehat q_i-1/p_i}{1/q_i-1/p_i}
=\frac{q_i(p_i-\widehat q_i)}{\widehat q_i(p_i-q_i)}
.\end{equation*}
If $0<q_i< \widehat q_i < p_i$, we have that $0<\tau_i<1$. Thus
\begin{align*}
\sum_{i=\varpi}^n \|u_i\|_{L^{\widehat q_i}(A_\ell )}
&=
\sum_{i=\varpi}^n
\biggl(\int_{A_\ell } \abs{u_i}^{\tau_i \widehat q_i} \abs{u_i}^{(1-\tau_i) \widehat q_i}
\biggr)^{1/\widehat q_i}
.\end{align*}
We compute that
\begin{equation*}\frac{q_i}{\tau_i \widehat q_i}
=
\frac{p_i-q_i}{p_i-\widehat q_i}
\in (1,\infty ),
\qquad
\biggl(\frac{q_i}{\tau_i \widehat q_i} \biggr)'(1-\tau_i)\widehat q_i=p_i
.\end{equation*}
So we may apply H\"older's inequality to see that
\begin{align*}
\sum_{i=\varpi}^n \|u_i\|_{L^{\widehat q_i}(A_\ell )}
&\leq
\sum_{i=\varpi}^n
\doublebar{u_i}_{L^{q_i}({A_\ell })}^{\tau_i}
\doublebar{u_i}_{L^{p_i}({A_\ell })}^{1-\tau_i}
.\end{align*}
By Young's inequality,
\begin{align*}
\sum_{i=\varpi}^n \|u_i\|_{L^{\widehat q_i}(A_\ell )}
&\leq
\sum_{i=\varpi}^n
\tau_i
\biggl(\frac{c_0}{(\vartheta_{\ell+1}-\vartheta_\ell)^{\kappa}}
\biggr)^{(1-\tau_i)/\tau_i}
\doublebar{u_i}_{L^{q_i}({A_\ell })}
\\&\qquad+
\sum_{i=\varpi}^n
(1-\tau_i)
\frac{(\vartheta_{\ell+1}-\vartheta_\ell)^{\kappa}}{c_0}
\doublebar{u_i}_{L^{p_i}({A_\ell })}
\end{align*}
If $q_i=\widehat q_i$ and so $\tau_i=1$, this bound is still true.
By the bound~\eqref{epn:local-bound:steq1},
\begin{align*}
\sum_{j=\varpi}^m \|u_j\|_{L^{p_j}(Q_{\ell})}
&\leq
\frac{F}{(\vartheta_{\ell+1}-\vartheta_\ell)^{\kappa}}
+\frac{c_0}{(\vartheta_{\ell+1}-\vartheta_\ell)^{\kappa}}
\sum_{i=\varpi}^n
\|u_i\|_{L^{\widehat q_i}(A_{\ell})}
\\&\leq
\frac{F}{(\vartheta_{\ell+1}-\vartheta_\ell)^{\kappa}}
%\\&\qquad
+
\sum_{i=\varpi}^n
\tau_i
\biggl(\frac{c_0}{(\vartheta_{\ell+1}-\vartheta_\ell)^{\kappa}}
\biggr)^{1/\tau_i}
\doublebar{u_i}_{L^{q_i}({A_\ell })}
\\&\qquad
+\sum_{i=\varpi}^n
(1-\tau_i)
\doublebar{u_i}_{L^{p_i}({A_\ell })}
.\end{align*}
Recall that $\vartheta_0=1$. We now let $\vartheta_{\ell+1}=\vartheta_\ell + (\theta-1)(1-\sigma) \sigma^\ell$ for some constant $\sigma\in (0,1)$ to be chosen momentarily. Notice that $\lim_{\ell \to\infty } \vartheta_\ell=\theta$. Recall that $A_\ell\subseteq Q_{\ell+1}$. Then
\begin{align*}
\sum_{j=\varpi}^m \|u_j\|_{L^{p_j}(Q_{\ell})}
&\leq
\frac{F}{(\theta-1)^\kappa(1-\sigma)^\kappa \sigma^{\kappa\ell}}
%\\&\qquad
+
\sum_{i=\varpi}^n
\frac{\tau_i c_0^{1/\tau_i}}{(\theta-1)^{\kappa/\tau_i} (1-\sigma)^{\kappa/\tau_i} \sigma^{\kappa\ell/\tau_i}}
\doublebar{u_i}_{L^{q_i}({A_\ell })}
\\&\qquad
+\sum_{i=\varpi}^n
(1-\tau_i)
\doublebar{u_i}_{L^{p_i}({Q_{\ell+1} })}
.\end{align*}
Let $\tau=\min_i\tau_i$. If $\tau=1$ then $q_i=\widehat q_i$ for all~$i$ and there is nothing to prove; otherwise, $\tau\in (0,1)$.
Recall that $\varpi\leq n\leq m$.
Iterating, we see that if $K\geq 0$ is an integer, then
\begin{align*}
\sum_{j=\varpi}^m \|u_j\|_{L^{p_j}(Q_{0})}
&\leq
\sum_{\ell=0}^K (1-\tau)^\ell
\frac{F}{(\theta-1)^\kappa(1-\sigma)^\kappa \sigma^{\kappa\ell}}
\\&\qquad
+
\sum_{\ell=0}^K (1-\tau)^\ell
\sum_{i=\varpi}^n
\frac{\tau_i c_0^{1/\tau_i}}{(\theta-1)^{\kappa/\tau} (1-\sigma)^{\kappa/\tau} \sigma^{\kappa\ell/\tau}}
\doublebar{u_i}_{L^{q_i}({A_\ell })}
\\&\qquad
+\sum_{j=\varpi}^m
(1-\tau)^{K+1}
\doublebar{u_j}_{L^{p_j}({Q_{\ell+1} })}
.\end{align*}
Recall that $Q_0=Q$ and $Q_\ell\subset \theta Q$, $A_\ell\subset \theta Q\setminus Q$ for all $\ell\geq 0$. Changing the order of summation, we see that
\begin{align*}
\sum_{j=\varpi}^m \|u_j\|_{L^{p_j}(Q)}
&\leq
\frac{F}{(\theta-1)^\kappa(1-\sigma)^\kappa }
\sum_{\ell=0}^K
\biggl(\frac{1-\tau}{\sigma^{\kappa}}\biggr)^\ell
\\&\qquad
+
\sum_{i=\varpi}^n
\frac{\tau_i c_0^{1/\tau_i}}{(\theta-1)^{\kappa/\tau} (1-\sigma)^{\kappa/\tau} }
\doublebar{u_i}_{L^{q_i}(\theta Q\setminus Q)}
\sum_{\ell=0}^K \biggl(\frac{1-\tau}{\sigma^{\kappa/\tau}}
\biggr)^\ell
\\&\qquad
+
(1-\tau)^{K+1}
\sum_{j=\varpi}^m
\doublebar{u_j}_{L^{p_j}(\theta Q)}
.\end{align*}
Choose $\sigma\in (0,1)$ such that $1-\tau<\sigma^{\kappa/\tau}$; since $\tau\in (0,1)$, this implies $1-\tau<\sigma^\kappa$. Taking the limit as $K\to\infty $, we have that the geometric series converge and the final term approaches zero, and so
\begin{align*}
\sum_{j=\varpi}^m \|u_j\|_{L^{p_j}(Q)}
&\leq
C\frac{F}{(\theta-1)^\kappa}
%\\&\qquad
+
C\sum_{i=\varpi}^n
\frac{1}{(\theta-1)^{\kappa/\tau} }
\doublebar{u_i}_{L^{q_i}(\theta Q\setminus Q)}
\end{align*}
as desired.
\end{proof}

\subsection{A counterexample}
\label{sec:meyers:counter}

In this section we will prove Theorem~\ref{thm:meyers:counterexample}.

Let $\mathfrak{a}$, $\mathfrak{b}$ and $\varepsilon$ be as in the theorem statement. Without loss of generality we may require $0<\varepsilon\leq 1$.
Fix a multiindex $\zeta$ with $|\zeta|=\mathfrak{b}$.

Define $w(X) = (1+|X|^2)^{-d}$. We may easily compute that $\nabla^m w\in L^p(\R^d)$ for any $p>d/(2d+m)$ (in particular, for all $p\geq2$).

Let $\{Q_k\}_{k=1}^\infty$ be a sequence of pairwise-disjoint cubes contained in $Q$ (whose volumes necessarily tend to zero).
Let $\varphi$ be a smooth cutoff function with $\varphi$ supported in $Q$ and with $\varphi=1$ in $\frac{1}{2}Q$, and let $\varphi_k(X)=\varphi((X-X_k)/\ell_k)$, where $X_k$ is the midpoint of $Q_k$ and $\ell_k=|Q_k|^{1/\dmn}$ is the side length of~$Q_k$. Then $\varphi_k$ is a smooth cutoff function supported in $Q_k$ and identically $1$ in $\frac{1}{2}Q_k$.

Let $\{n_k\}_{k=1}^\infty$ be a sequence of positive numbers such that $n_k\ell_k\to \infty$ and $n_k\ell_k\geq 1$ for all~$k$. Notice that $\ell_k<1$ so $n_k>1$ for all~$k$.
Define
\begin{equation*}u(X)=X^\zeta+\frac{\varepsilon}{C_0}\sum_{k=1}^\infty \varphi_k(X)\frac{1}{n_k^{2m}}w(n_k(X-X_k))
\end{equation*}
for a positive constant $C_0$ to be chosen momentarily.
We may easily compute that
if $X\in \frac{1}{2}Q_k$ and $\gamma$ is a multiindex, then
\begin{equation}\partial^\gamma u(X)=
\partial^\gamma X^\zeta
+\frac{\varepsilon}{C_0}
\frac{1}{n_k^{2m-|\gamma|}} (\partial^\gamma w)(n_k(X-X_k)). \end{equation}
Furthermore, if $X\in Q_k$ and $0\leq |\gamma|\leq 2m$, then
\begin{equation*}
|\partial^\gamma u(X)-
\partial^\gamma X^\zeta|
\leq
\frac{\varepsilon}{C_0}
C(\gamma,\varphi,d) \, n_k^{|\gamma|-2m}
\leq
\frac{\varepsilon}{C_0}
C(\gamma,\varphi,d)
.\end{equation*}
We choose $C_0\geq 2 C(\zeta,\varphi,d)$; this ensures that
\begin{equation*}|\partial^\zeta u-\zeta!|=|\partial^\zeta u-\partial^\zeta X^\zeta|\leq \frac{1}{2}\leq \frac{1}{2}\zeta!\end{equation*}
and so $|\partial^\zeta u(X)|\geq \frac{1}{2}$ for all~$X$.

Recall that $\widetilde A_{\alpha,\beta}$ is a set of real nonnegative constants that satisfies
\begin{equation*}(-\Delta)^m = (-1)^m\sum_{|\alpha|=m} \sum_{|\beta|=m} \widetilde A_{\alpha,\beta} \partial^{\alpha+\beta}.\end{equation*}
(Many possible families of such constants exist.)
Similarly, for any $\mathfrak{a}\leq m$, there exist families of constants $\widetilde B_{\alpha,\gamma}$ such that
\begin{equation*}(-\Delta)^m = (-1)^{m} \sum_{|\alpha|=\mathfrak{a}} \sum_{|\gamma|=2m-\mathfrak{a}} \widetilde B_{\alpha,\gamma} \partial^{\alpha+\gamma}.\end{equation*}
Choose some such family. 

Define the coefficients $A_{\alpha,\beta}=A^{1,1}_{\alpha,\beta}$ as follows.
\begin{itemize}
\item If $|\alpha|=|\beta|=m$, let $A_{\alpha,\beta}=\widetilde A_{\alpha,\beta}$.
\item If $|\alpha|=\mathfrak{a}$ and $\beta=\zeta$, let \begin{equation*}
A_{\alpha,\zeta}=-\sum_{|\gamma|=2m-\mathfrak{a}} \widetilde B_{\alpha,\gamma} \frac{\partial^\gamma u}{\partial^\zeta u}.\end{equation*}
\item Otherwise, let $A_{\alpha,\beta}=0$.
\end{itemize}
Because $|\zeta|=\mathfrak{b}<m$, $A_{\alpha,\beta}$ is well defined.

If $L$ is as given by formula~\eqref{dfn:L}, a straightforward computation yields that $Lu=0$, formulas~\eqref{eqn:tildealpha} and~\eqref{eqn:tildebeta} are valid, and if $C_0$ is large enough then $|A_{\alpha,\beta}(X)-\widetilde A_{\alpha,\beta}(X)|<\varepsilon$ for all $X$, $\alpha$, and~$\beta$.
It remains only to establish a lower bound on $\int_{\frac{1}{2}Q_k} |\nabla^m u|$.

If $p\geq 2$ and $\mathfrak{b}+1\leq j\leq m$, then by definition of~$u$ and a change of variables,
\begin{align*}
\biggl(
\fint_{\frac{1}{2}Q_k}|\nabla^j u|^p\biggr)^{1/p}
&=
\frac{\varepsilon}{C_0 n_k^{2m-j}}
\biggl(\fint_{\frac{1}{2}Q_k}|(\nabla^j w) (n_k(X-X_k))|^p\,dX\biggr)^{1/p}
\\&=
\frac{\varepsilon}{C_0 n_k^{2m-j}}
\biggl(\fint_{\frac{1}{2}\ell_k n_k Q}|\nabla^j w(X)|^p\,dX\biggr)^{1/p}
\\&=
\frac{2^{d/p}\varepsilon}{C_0 n_k^{2m-j} (\ell_k n_k)^{d/p}}
\biggl(\int_{\frac{1}{2}\ell_k n_k Q}|\nabla^j w(X)|^p\,dX\biggr)^{1/p}
.\end{align*}
Thus, recalling that $n_k\ell_k\geq 1$, we have that
\begin{equation*}
\biggl(
\fint_{\frac{1}{2}Q_k}|\nabla^m u|^p\biggr)^{1/p}
\geq
\frac{2^{d/p}\varepsilon}{C_0 n_k^{m} (\ell_k n_k)^{d/p}}
\biggl(\int_{\frac{1}{2} Q}|\nabla^m w(X)|^p\,dX\biggr)^{1/p}
\geq
\frac{c_1}{n_k^{m} (\ell_k n_k)^{d/p}}
\end{equation*}
where $c_1>0$ is independent of~$k$.

Furthermore, if $X\in Q_k\setminus \frac{1}{2}Q_k$ then
\begin{equation*}|\nabla^j u(X)|\leq \frac{C\varepsilon}{C_0n_k^{2m-j}(n_k\ell_k)^{2d+j}}. \end{equation*}
Thus,
\begin{align*}
\frac{1}{\ell_k^{m-j}}
\biggl(
\fint_{Q_k}|\nabla^j u|^2\biggr)^{1/2}
&\leq
\frac{1}{\ell_k^{m-j}}
\biggl(
2^{-d}
\fint_{\frac{1}{2}Q_k}|\nabla^j u|^{2}
+
\biggl(\frac{C\varepsilon}{C_0n_k^{2m-j}(n_k\ell_k)^{2d+j}}\biggr)^{2}
\biggr)^{1/2}
\\&\leq
\frac{\varepsilon}{C_0 n_k^{2m-j} \ell_k^{m-j}}
\biggl(\frac{1}{(\ell_k n_k)^{d}}
\int_{\R^d}|\nabla^j w|^2
+
\frac{C}{(n_k\ell_k)^{4d+2j}}
\biggr)^{1/2}
.\end{align*}
Again using the fact that $n_k\ell_k\geq 1$ and the fact that $\nabla^j w\in L^2(\R^\dmn)$ for any $j\geq 0$, we have that
\begin{align*}
\sum_{j=\mathfrak{b}+1}^m
\frac{1}{\ell_k^{m-j}}
\biggl(
\fint_{Q_k}|\nabla^j u|^2\biggr)^{1/2}
&\leq
\frac{C_2}{n_k^{m}(\ell_k n_k)^{d/2}}
.\end{align*}
If $p>2$, then because $n_k\ell_k\to \infty$, there is some $k$ large enough that
\begin{equation*}
\widetilde C
\frac{C_2}{n_k^{m}(\ell_k n_k)^{d/2}}
\leq
\frac{c_1}{n_k^{m} (\ell_k n_k)^{d/p}}
\end{equation*}
as desired. This completes the proof of Theorem~\ref{thm:meyers:counterexample}.

\section{The fundamental solution}\label{FSS}
In this section we will construct the fundamental solution. We will begin in Section~\ref{sec:FS:high:preliminary} with local estimates on functions in $Y^{m,p}(\RR^d)$ for $m$ large enough. Using these estimates, in Section~\ref{sec:FS:high} we will construct a preliminary version of the fundamental solution in the case $2m>d$. We will investigate the properties of this fundamental solution in Sections~\ref{sec:FS:high}--\ref{sec:FS:Fubini}. We will slightly modify our definition in Section~\ref{sec:FS:even}.
In Section~\ref{sec:FS:low} we will construct the fundamental solution in the case $2m\leq d$, and will address uniqueness in Section~\ref{sec:FS:unique}.

\subsection{Preliminaries for operators of high order}\label{sec:FS:high:preliminary}

Recall from the definition~\eqref{dfn:Y:norm} of $Y^{m,q}(\RR^d)$ that if $u\in Y^{m,q}(\R^d)$, then the derivatives $\partial^\gamma u$ of~$u$ are defined as locally integrable functions if $|\gamma|>m-d/q$, and are defined only up to adding polynomials if $|\gamma|\leq m-d/q$. We will now wish to fix a family of normalizations of functions in $Y^{m,q}(\R^d)$ and investigate their properties.

If $d/m<q<\infty $, let $s_{m,d,q}$ be the number of multiindicies $\gamma\in (\NN_0)^d$ so that $|\gamma|\leq m-d/q$. Observe that $s_{m,d,q}$ is nonnegative, nondecreasing in $q$ and that if $q<\infty $ then $s_{m,d,q}\leq s_{m,d,d}$.
Choose distinct points $H_1$, $H_2,\dots, H_{s_{m,d,d}}$ in $B(0,1)\setminus \overline{B(0,1/2)}$ (so $1/2<|H_i|<1$ for all $1\leq i\leq s_{m,d,d}$).  If the points $H_i$ are chosen appropriately (see \cite{GasS00} for a survey on polynomial interpolation in several variables) then for any $q$ with $d/m<q<\infty $ and any numbers $a_i$ there is a unique polynomial
\begin{equation*}P(X)=\sum_{|\gamma|\leq m-d/q} p_\gamma X^\gamma\text{ such that }P(H_i)=a_i
\text{ for all }1\leq i\leq s_{m,d,q}.\end{equation*}
(We emphasize that if $q<d$ then we cannot specify the values of $P(H_i)$ for $s_{m,d,q}<i\leq s_{m,d,d}$.) Also there is some constant $h<\infty $ depending only on $H_i$ such that \begin{equation*}\sup_{|\gamma|\leq m-d/q}|p_\gamma|\leq h\sup_{1\leq i\leq s_{m,d,q}}|a_i|.\end{equation*}

We now show that this gives a normalization in~$Y^{m,q}(\RR^d)$. We will need some additional properties of this normalization.

\begin{lemma}\label{SFbound} Let $m$, $d\in\NN$ with $d\geq 2$, let $r>0$, and let $Z_0\in\RR^d$. Let $\max(1,d/m)<\mu\leq q<\infty $. Let $U$ satisfy $\|U\|_{Y^{m,\mu}(\RR^d)}<\infty $.

Then there is a unique function $U_{Z_0,r,q}$ that is continuous and satisfies
\begin{equation*}U_{Z_0,r,q}(Z_0+rH_i)=0, \qquad \partial^\zeta U=\partial^\zeta U_{Z_0,r,q} \text{ almost everywhere}\end{equation*}
for all $1\leq i\leq s_{m,d,q}$ and all multiindices $\zeta$ with $m-d/q<|\zeta|\leq m$. In particular, if $q=\mu$ then $U$ and $U_{Z_0,r,q}$ are representatives of the same element of~$Y^{m,\mu}(\RR^d)$.

Furthermore, if $X$, $Y\in\RR^d$, $R=r+|X-Z_0|$, $|X-Y|\leq \frac{1}{2}R$, and $|\gamma|<m-d/\mu$, then we have the bounds
\begin{align*}
| \partial^\gamma U_{Z_0,r,q}(X)|
&\leq C_{\mu}
R^{m-d/\mu-|\gamma|}
\bigg(\frac{R}{r}\bigg)^{\omega_q-1}
\|U\|_{Y^{m,\mu}(\RR^d)}
,\\
| \partial^\gamma U_{Z_0,r,q}(X)-\partial^\gamma U_{Z_0,r,q}(Y)|
&\leq
C_{\mu}R^{m-d/\mu-|\gamma|}
\bigg(\frac{R}{r}\bigg)^{\omega_q-1}
\|U\|_{Y^{m,\mu}(\RR^d)}\biggl(\frac{|X-Y|}{R}\biggr)^{\varepsilon}
,\end{align*}
where $C_{\mu}$ and $\varepsilon>0$ depend on $d$, $m$, and~$\mu$, and $\omega_q$ is the smallest (necessarily positive) integer with $m-d/q<\omega_q$.
\end{lemma}

\begin{proof}
Fix $X\in\RR^d$. Let $Q$ be a cube centered at $Z_0$ of side length~$4R$.
Observe that $\|U\|_{Y^{m,\mu}(Q)}\leq\|U\|_{Y^{m,\mu}(\RR^d)}<\infty $. Then $\nabla^m U\in L^\mu(Q)$, so by the Poincar\'e inequality, we have that $\nabla^i U\in L^\mu(Q)$ (and thus is integrable) for any $0\leq i\leq m$. Let $V=U+P$, where $P$ is a polynomial of degree at most $m-d/\mu$ so that $\int_Q \partial^\gamma V=0$ for all $\gamma$ with $|\gamma|\leq m-d/\mu$ (that is, all $\gamma$ with $|\gamma|<\omega_\mu$). Observe that $\|U-V\|_{Y^{m,\mu}(\RR^d)}=0$, so $\|V\|_{Y^{m,\mu}(Q)}=\|U\|_{Y^{m,\mu}(Q)}<\infty $.

If $d/\mu$ is not an integer, let $\theta=\mu$. Otherwise, let $\theta$ satisfy $d/\theta=d/\mu+1/2$. In either case, $d/\theta$ is not an integer and $\theta\leq \mu$. Since $m>d/\mu+|\gamma|$, if $d/\mu$ is an integer then $m\geq d/\mu+|\gamma|+1$ and so $m>d/\theta+|\gamma|$.
Because $\mu>1$ we have that $d>d/\mu$, so similarly $d>d/\theta$ and so $\theta>1$.

Let $k$ be the unique integer such that $m-d/\theta<k<m-d/\theta+1$. Thus $|\gamma|<m-d/\theta<k<m+1$ and so $|\gamma|+1\leq k\leq m$.
By Lemma~\ref{evlem},
\begin{equation*}\|\nabla \partial^\gamma V\|_{L^{\theta_k}(Q)}
\leq
C_\mu\sum_{i=1}^{m-k+1} R^{i-1+k-m} \|\nabla^i \partial^\gamma V\|_{L^\theta(Q)}\end{equation*}
and by H\"older's inequality and because $k\geq 1+|\gamma|$,
\begin{equation*}\|\nabla \partial^\gamma V\|_{L^{\theta_k}(Q)}
\leq
C_\mu\sum_{i=1}^{m} R^{i-1-|\gamma|+k-m+d/\theta-d/\mu} \|\nabla^i V\|_{L^\mu(Q)}.\end{equation*}
By formula~\eqref{pk},
\begin{equation*}\frac{1}{\theta_k}=\frac{1}{\theta}-\frac{m-k}{d} \in \left(0,\frac{1}{d}\right).\end{equation*}
By Morrey's inequality (see \cite[Section~5.6.2]{Eva98}), we may redefine the weak derivative $\partial^\gamma V$ of $V$ on a set of measure zero in a unique way so that it is continuous (thus defined pointwise everywhere) and, if $\widetilde X\in \frac{1}{2}Q$ and $|\widetilde X-Y|<R/2$, then
\begin{equation*}
|\partial^\gamma V(\widetilde X)-\partial^\gamma V(Y)| \leq C_\mu|\widetilde X-Y|^{1-d/\theta_k} \|\nabla \partial^\gamma V\|_{L^{\theta_k}(Q)}
.\end{equation*}
Let $\varepsilon=1-d/\theta_k=1-d/\theta+m-k$. Observe that $0<\varepsilon<1$. Then
\begin{equation}\label{eqn:V:holder}
|\partial^\gamma V(\widetilde X)-\partial^\gamma V(Y)| \leq C_\mu\frac{|\widetilde X-Y|^{\varepsilon} } {R^{\varepsilon+|\gamma|+d/\mu}} \sum_{i=1}^{m} R^{i} \|\nabla^i V\|_{L^\mu(Q)}
.\end{equation}
Averaging $|\partial^\gamma V(\widetilde X)|\leq |\partial^\gamma V(\widetilde X)-\partial^\gamma V(Y)|+\partial^\gamma |V(Y)|$ over $Y\in B(\widetilde X,R/2)$ we have that
\begin{equation}\label{eqn:V:pointwise}
|\partial^\gamma V(\widetilde X)|
\leq
\frac{C_\mu}{R^{|\gamma|+d/\mu}}
\sum_{i=0}^{m} R^{i} \|\nabla^i V\|_{L^\mu(Q)}
.\end{equation}
We will consider the cases $\widetilde X=X$ and $\widetilde X=Z_0+rH_j$.

We may write
\begin{equation*}
\sum_{i=0}^{m} R^{i} \|\nabla^i V\|_{L^\mu(Q)}
=
\sum_{i=0}^{\omega_\mu-1}R^{i} \biggl(\int_{Q}|\nabla^i V|^\mu\biggr)^{1/\mu}
+\sum_{i=\omega_\mu}^m R^{i} \biggl(\int_{Q}|\nabla^i V|^{\mu}\biggr)^{1/\mu}.
\end{equation*}
Recall that $V$ satisfies $\int_Q \nabla^i V=0$ for all $0\leq i\leq \omega_\mu-1$. We may apply the Poincar\'{e} inequality in the first sum so that
\begin{equation*}R^i\biggl(\int_{Q}|\nabla^i V|^\mu\biggr)^{1/\mu}\leq C_\mu R^{\omega_\mu}\biggl(\int_{Q}|\nabla^{\omega_\mu} V|^\mu\biggr)^{1/\mu}.\end{equation*}
Thus
\begin{equation*}\sum_{i=0}^{m} R^{i} \|\nabla^i V\|_{L^\mu(Q)}
\leq C_\mu\sum_{i=\omega_\mu}^m R^{i} \biggl(\int_{Q}|\nabla^i V|^{\mu}\biggr)^{1/\mu}.
\end{equation*}
By H\"{o}lder's inequality, we have that
\begin{align*}
\sum_{i=0}^{m} R^{i} \|\nabla^i V\|_{L^\mu(Q)}
&\leq
C_\mu\sum_{i=\omega_\mu}^mR^{i+d/\mu-d/\mu_i} \bigg(\int_{Q}|\nabla^i V|^{\mu_i}\bigg)^{1/\mu_i}
.\end{align*}
By formula~\eqref{pk} we have that $i+d/\mu-d/\mu_i=m$.  Thus, by the definition~\eqref{dfn:Y:norm} of the norm on $Y^{m,\mu}(Q)$, we have that
\begin{equation}\label{upwbound2}
\sum_{i=0}^{m} R^{i} \|\nabla^i V\|_{L^\mu(Q)}
\leq C_\mu R^{m}\| U\|_{Y^{m,\mu}(Q)}.
\end{equation}

Let $P_1$ be the (unique) polynomial of degree at most $m-d/q$ with $P_1(Z_0+rH_j)=V(Z_0+rH_j)$ for each $1\leq j\leq s_{m,d,q}$, and let $U_{Z_0,r,q}=V-P_1$. Then $U_{Z_0,r,q}$ is the unique continuous function with $U_{Z_0,r,q}(Z_0+rH_j)=0$ for all $1\leq j\leq \mu$ and with $\partial^\zeta U_{Z_0,r,q}=\partial^\zeta V=\partial^\zeta U$ almost everywhere for all $|\zeta|>m-d/q$. Thus the specified function $U_{Z_0,r,q}$ is constructed; we need only establish the desired bounds on~$U_{Z_0,r,q}$.

We now take $\widetilde X=Z_0+rH_j$ for some~$j$. By formulas~\eqref{eqn:V:pointwise} and~\eqref{upwbound2},
\begin{align*}
|P_1(Z_0+rH_j)|=|V(Z_0+rH_j)|&
\leq
C_\mu
\sum_{i=0}^{m} R^{i-d/\mu} \|\nabla^i V\|_{L^\mu(Q)}
%\\&
\leq C_\mu R^{m-d/\mu}\| U\|_{Y^{m,\mu}(Q)}.
\end{align*}
Let $P_1(Z)=P_2((Z-Z_0)/r)$ so that $P_2(H_i)=P_1(Z_0+rH_i)$ and
$P_2(Z)=\sum_{|\gamma|\leq\omega_q-1}p_\gamma Z^\gamma$ for some $p_\gamma$ where $|p_\gamma|\leq h\sup_j|P_2(Z_0+rH_j)|\leq C_\mu R^{m-d/\mu}\| U\|_{Y^{m,\mu}(Q)}$.
We then have that $P_1(Z)=\sum_{|\gamma|\leq\omega_q-1}p_\gamma r^{-|\gamma|}(Z-Z_0)^\gamma$.
We may then compute that if $Z\in Q$ and $0\leq i\leq \omega_q-1$, then
\begin{equation*}
|\nabla^i P_1(Z)|
\leq C_\mu R^{m-d/\mu-i}\| U\|_{Y^{m,\mu}(Q)}(R/r)^{\omega_q-1}
.\end{equation*}
Combining these pointwise bounds on~$P_1$ with the bound~\eqref{upwbound2}
yields that
\begin{equation}\label{eqn:U:gamma}
\sum_{i=0}^{m} R^{i} \|\nabla^i U_{Z_0,r,q}\|_{L^\mu(Q)}
\leq C_\mu R^{m}\| U\|_{Y^{m,\mu}(Q)}(R/r)^{\omega_q-1}.
\end{equation}
Combining this bound with the bounds~\eqref{eqn:V:holder}, \eqref{eqn:V:pointwise} with $\widetilde X=X$ completes the proof.
\end{proof}

\begin{remark}\label{rmk:BMO}
We observe that if $U\in Y^{m,\mu}(\RR^d)$, then $\partial^\gamma U\in L^{\mu_\gamma}(\RR^d)$ is defined up to sets of measure zero whenever $|\gamma|>m-d/\mu$, while $\partial^\gamma U_{Z_0,r,q}$ is continuous and satisfies the bounds given by Lemma~\ref{SFbound} whenever $q\geq \mu$ and $|\gamma|<m-d/\mu$.

Suppose $|\gamma|=m-d/\mu$. If $k=|\gamma|+1$, then by formula~\eqref{pk} $\mu_k=d$ and so $\nabla\partial^\gamma U\in L^d(\RR^d)$.
By \cite[Section~5.8.1]{Eva98}, we have that $\partial^\gamma U$ lies in the space $BMO$ of bounded mean oscillation with $\doublebar{\partial^\gamma U}_{BMO}\leq C_\mu\|U\|_{Y^{m,\mu}(\RR^d)}$. By the John-Nirenberg inequality (see, for example, \cite{Ste93}) we have that if $1\leq p<\infty $ and $Q$ is any cube then
\begin{equation*}\biggl(\fint_{Q}\abs{\partial^\gamma U-\textstyle\fint_{Q}\partial^\gamma U}^p\biggr)^{1/p}
\leq C_{p,\mu}\|U\|_{Y^{m,\mu}(\RR^d)}
.\end{equation*}
Let $Z_0$, $r$, and $U_{Z_0,r,q}$ be as in Lemma~\ref{SFbound}.
Observe that $\partial^\zeta U=\partial^\zeta U_{Z_0,r,q}$ for all $|\zeta|>|\gamma|$, and so $\partial^\zeta U$ differs from $\partial^\zeta U_{Z_0,r,q}$ by a constant. Thus
\begin{equation*}\biggl(\fint_{Q}\abs{\partial^\gamma U_{Z_0,r,q}-\textstyle\fint_{Q}\partial^\gamma U_{Z_0,r,q}}^p\biggr)^{1/p}
\leq C_{p,\mu}\|U\|_{Y^{m,\mu}(\RR^d)}
.\end{equation*}
By the bound~\eqref{eqn:U:gamma} and H\"older's inequality, if $Q$ is a cube centered at $Z_0$ of side length $4R>4r$, then
\begin{equation*}
|\textstyle\fint_{Q}\partial^\gamma U_{Z_0,r,q}|
\leq |Q|^{-1/\mu}
\|\partial^\gamma U_{Z_0,r,q}\|_{L^\mu(Q)}
\leq C_\mu R^{m-|\gamma|}(R/r)^{\omega_q-1}\| U\|_{Y^{m,\mu}(Q)}
\end{equation*}
and so
\begin{align*}\biggl(\int_{Q}\abs{\partial^\gamma U_{Z_0,r,q}}^p\biggr)^{1/p}
&\leq C_{p,\mu} R^{d/p} (R/r)^{\omega_q-1} \|U\|_{Y^{m,\mu}(\RR^d)}
.\end{align*}
\end{remark}

\subsection{The fundamental solution for operators of high order}
\label{sec:FS:high}

We now define a preliminary version of our fundamental solution for operators of high order. If $d$ is odd, we will use this definition throughout; if $d$ is even then we will modify the definition somewhat in Section~\ref{sec:FS:even}. We will consider operators of lower order in Section~\ref{sec:FS:low}.

\begin{definition}\label{dfn:E:}

Let $m$ and $d$ be integers with $2m>d\geq 2$.
Let $L$ be a bounded and invertible linear operator $L:Y^{m,q}(\R^d)\to Y^{-m,q}(\R^d)$ for some $q$ with $1<q<\infty$ and $1-m/d<1/q<m/d$.
Let $Z_0\in\RR^d$, let $r>0$, and let $1\leq j\leq N$.

Let $T_{X,j,Z_0,r,q}$ be given by
\begin{equation*}\langle T_{X,j,Z_0,r,q},\vec\Phi\rangle = (\Phi_j)_{Z_0,r,q'}(X)\end{equation*}
where $1/q+1/q'=1$.
By Lemma~\ref{SFbound}, this is a well defined bounded linear operator on $Y^{m,q'}(\RR^d)$; that is, $T_{X,j,Z_0,r,q}\in Y^{-m,q}(\RR^d)$.

We define the fundamental solution $\vec E^L_{X,j,Z_0,r,q}$ by
\begin{equation*}
%\label{dfn:E:}
\vec E^L_{X,j,Z_0,r,q} =(L^{-1}T_{X,j,Z_0,r,q})_{Z_0,r,q}.\end{equation*}
\end{definition}

\begin{remark}
If $L$ is bounded and invertible $L:Y^{m,2}(\R^d)\to Y^{-m,2}(\R^d)$, and if $L$ is defined and bounded $Y^{m,q}(\R^d)\to Y^{-m,q}(\R^d)$ for all $q$ in an open neighborhood of~$2$, then by Lemma~\ref{NPLp}, $q$ satisfies the conditions of Definition~\ref{dfn:E:} for all $q$ in a (possibly smaller) neighborhood of~$2$.
\end{remark}

\begin{remark}\label{rmk:L:sym}
Since $Y^{-m,q}(\R^d)$ is by definition the dual space to $Y^{m,q'}(\R^d)$, by standard function theoretic arguments $L:Y^{m,q}(\R^d)\to Y^{-m,q}(\R^d)$ is bounded and invertible if and only if its adjoint operator $L^*:Y^{m,q'}(\R^d)\to Y^{-m,q'}(\R^d)$ is bounded and invertible. Furthermore, $(L^{-1})^*=(L^*)^{-1}$. Also observe that $\max(0,1-m/d)< 1/q<\min(1, m/d)$ if and only if $\max(0,1-m/d)< 1/q'<\min(1, m/d)$. Thus, $L$ and~$q$ satisfy the conditions of Definition~\ref{dfn:E:} if and only if $L^*$ and~$q'$ satisfy those conditions.

That is, $\vec E^L_{X,j,Z_0,r,q}$ exists (for all $X$, $j$, $Z_0$, $r$) if and only if $\vec E^{L^*}_{Y,k,Z_0,r,q'}$ exists (for all $Y$, $k$, $Z_0$, and~$r$).
\end{remark}

In the remainder of this subsection we will establish some basic properties of the fundamental solution; we will establish further properties in Sections~\ref{sec:FS:mixed}--\ref{sec:FS:Fubini}.  We will begin with a symmetry property for the operators $L$ and $L^*$; we will use this property to establish certain symmetries of the fundamental solution.

\begin{theorem}\label{thm:fundamental:1}
Let $L$ and $q$ satisfy the conditions of Definition~\ref{dfn:E:}.
Let $Z_0\in\R^d$, let $r>0$, and let $j$, $k$ be integers in $[1,N]$.

For all $X$, $Y\in\RR^d$ we have that
\begin{equation}\label{eqn:E:sym}
(\vec E^L_{X,j,Z_0,r,q})_k(Y)
=\overline{(\vec E^{L^*}_{Y,k,Z_0,r,q'})_j(X)}
.\end{equation}

For every $S\in Y^{-m,q'}(\RR^d)$ and every $X\in\RR^d$ we have that
\begin{equation}\label{eqn:E:Linv}
\overline{\langle S,\vec E^L_{X,j,Z_0,r,q}\rangle }
=(((L^*)^{-1}S)_j)_{Z_0,r,q'}(X)
.\end{equation}

Finally, if we let
\begin{equation*}E^L_{j,k,Z_0,r,q}(X,Y)
=(\vec E^L_{Y,k,Z_0,r,q})_j(X)
=\overline{(\vec E^{L^*}_{X,j,Z_0,r,q'})_k(Y)},
\end{equation*}
then $E^L_{j,k,Z_0,r,q}$ is continuous on $\RR^d\times\RR^d$.
\end{theorem}

\begin{proof} That $\vec E^{L^*}_{Y,k,Z_0,r,q'} $ exists is Remark~\ref{rmk:L:sym}.

If $X$, $Y\in \RR^d$ and $1\leq j\leq N$, $1\leq k\leq N$, then by Definition~\ref{dfn:E:} and Remark~\ref{rmk:L:sym},
\begin{align*}
(\vec E^L_{X,j,Z_0,r,q})_k(Y)
&=\langle T_{Y,k,Z_0,r,q'},\vec E^L_{X,j,Z_0,r,q}\rangle
=\langle T_{Y,k,Z_0,r,q'},L^{-1}T_{X,j,Z_0,r,q}\rangle
\\&=\overline{\langle T_{X,j,Z_0,r,q}, (L^*)^{-1}T_{Y,k,Z_0,r,q'}\rangle}
=\overline{\langle T_{X,j,Z_0,r,q}, \vec E^{L^*}_{Y,k,Z_0,r,q'}\rangle}
\\&=\overline{(\vec E^{L^*}_{Y,k,Z_0,r,q'})_j(X)}
.\end{align*}
In particular, observe that by Lemma~\ref{SFbound}, $\vec E^L_{Y,k,Z_0,r,q}(X)$ is locally uniformly continuous in both $X$ and~$Y$, and so $E^L_{j,k,Z_0,r,q}$ is continuous on $\RR^d\times\RR^d$.

Similarly, we have that if $S\in Y^{-m,q'}(\RR^d)$, then
\begin{equation*}\overline{\langle S,\vec E^L_{X,j,Z_0,r,q}\rangle }
=\overline{\langle S,L^{-1}T_{X,j,Z_0,r,q}\rangle }
=\langle T_{X,j,Z_0,r,q},(L^*)^{-1}S\rangle
=(((L^*)^{-1}S)_j)_{Z_0,r,q'}(X)
.\end{equation*}
This establishes formula~\eqref{eqn:E:Linv}.
\end{proof}

We will conclude this section with a preliminary bound on the derivatives of the function $\vec E^L_{X,j,Z_0,r,q}$.

\begin{theorem}\label{thm:derivs}
Let $L$ and $q$ satisfy the conditions of Definition~\ref{dfn:E:}.
Let $1<{{p}}\leq 2q$.
Suppose that $L$ also satisfies the conditions of Definition~\ref{dfn:E:} with $q$ replaced by~${{p}}$, and that the inverses are compatible in the sense of Definition~\ref{dfn:compatible}, that is, if $T\in Y^{-m,{{p}}}(\R^d)\cap Y^{-m,q}(\R^d)$ then $L^{-1}T\in Y^{m,{{p}}}(\R^d)\cap Y^{m,q}(\R^d)$.

Suppose that $\beta$ is a multiindex with $0\leq |\beta|\leq m$. Let $Q\subset\RR^d$ be a cube.
Then we have the bound
\begin{equation}\label{eqn:E:2bound}
\biggl(\int_{Q} |\partial^\beta\vec E^L_{X,j,Z_0,r,q}|^{{p}}\biggr)^{1/{{p}}}
\leq
CR^{2m-d+d/{{p}}-|\beta|}\bigg(\frac{R}{r}\bigg)^{\kappa}
\end{equation}
where $R=\max(r,|X-Z_0|,\dist(Z_0,Q)+\diam Q)$, and where $C$ and $\kappa$ are positive constants depending on $q$, ${{p}}$, the norms of $L^{-1}:Y^{-m,q}(\R^d)\to Y^{m,q}(\R^d)$ and $L^{-1}:Y^{-m,{{p}}}(\R^d)\to Y^{m,{{p}}}(\R^d)$, and the standard parameters.
\end{theorem}

Recall from Definition~\ref{dfn:compatible} that $\Upsilon_L$ is the set of all $q$ such that
$L^{-1}$ is compatible between $Y^{m,2}(\R^d)$ and $Y^{m,q}(\R^d)$. By density, if $p$, $q\in\Upsilon_L$, then $L^{-1}$ is compatible between $Y^{m,p}(\R^d)$ and $Y^{m,q}(\R^d)$, as required by the lemma.

\begin{proof}[Proof of Theorem~\ref{thm:derivs}]
By Lemma~\ref{SFbound}, if $T_{X,j,Z_0,r,q}$ is as in Definition~\ref{dfn:E:}, then
\begin{align*}
\|T_{X,j,Z_0,r,q}\|_{Y^{-m,q}(\RR^d)} \leq C_q R^{m-d/{q'}} \bigg(\frac{R}{r}\bigg)^{\omega_{{q'}}-1}\end{align*}
and so by invertibility of~$L$,
\begin{equation}\label{eqn:E:Y:norm}
\|\vec E^L_{X,j,Z_0,r,q}\|_{Y^{m,q}(\RR^d)}
\leq C R^{m-d/{q'}} \bigg(\frac{R}{r}\bigg)^{\omega_{{q'}}-1} .\end{equation}
By Lemma~\ref{SFbound}, if $|\beta|<m-d/q$ and $|Y-Z_0|<R$ then
\begin{equation*}
|\partial_Y^\beta\vec E^L_{X,j,Z_0,r,q}(Y)|\leq CR^{2m-d-|\beta|}
\biggl(\frac{R}{r}\biggr)^{\kappa}
.\end{equation*}
Integration yields the bound~\eqref{eqn:E:2bound} in this case (for all ${{p}}\in [1,\infty]$).

By Remark~\ref{rmk:BMO}, if $|\beta|=m-d/q$ and $|\widetilde Q|=4R$ with $\widetilde Q$ centered at~$Z_0$, then 
\begin{equation*}\biggl(\int_{\widetilde Q}\abs{\partial_Y^\beta \vec E^L_{X,j,Z_0,r,q}(Y)}^{{p}}\,dY\biggr)^{1/{{p}}}
\leq CR^{d/{{p}}+m-d/{q'}} \biggl(\frac{R}{r}\biggr)^{\kappa}
.\end{equation*}
Because $|\beta|=m-d/q=m-d+d/{q'}$, the bound~\eqref{eqn:E:2bound} is valid in this case (for all ${{p}}\in [1,\infty)$).

We are left with the case $|\beta|>m-d/q$.
If $q\geq {{p}}$ and $m-d/q<|\beta|$, or if $q<{{p}}$ and $m-d/q<|\beta|
\leq m-d/q+d/{{p}}$,
then by formula~\eqref{pk} we have that ${{p}}\leq q_\beta<\infty $. By the bound~\eqref{eqn:E:Y:norm} and
H\"older's inequality,
\begin{equation*}
\biggl(\int_{Q} |\partial^\beta\vec E^L_{X,j,Z_0,r,q}|^{{p}}\biggr)^{1/{{p}}}
\leq
CR^{2m-d+d/{{p}}-|\beta|}
\bigg(\frac{R}{r}\bigg)^{\kappa}
.\end{equation*}

Finally, suppose that $q<{{p}}$ and that $m-d/q+d/{{p}}<|\beta|\leq m$. If $q< {{p}}$ then ${q'}>{{p}}'$,
and so by Lemma~\ref{SFbound}
\begin{equation*}
\|T_{X,j,Z_0,r,q}\|_{Y^{-m,{{p}}}(\RR^d)} \leq C R^{m-d/{{p}}'} \bigg(\frac{R}{r}\bigg)^{\omega_{{q'}}-1}.
\end{equation*}
By compatible invertibility of $L:Y^{m,{{p}}}(\RR^d)\to Y^{-m,{{p}}}(\RR^d)$, we have that 
\begin{equation}
\|\vec E^L_{X,j,Z_0,r,q}\|_{Y^{-m,{{p}}}(\RR^d)} \leq C R^{m+d-d/{{p}}} \bigg(\frac{R}{r}\bigg)^{\omega_{{q'}}-1}.\end{equation}
If $|\beta|>m-d/q+d/{{p}}$ and ${{p}}\leq 2q$ then $|\beta|>m-d/{{p}}$ and so this provides a Lebesgue space bound on $\partial^\beta \vec E^L_{X,j,Z_0,r,q}$. By H\"older's inequality,
\begin{equation*}\biggl(\int_{Q} |\partial^\beta\vec E^L_{X,j,Z_0,r,q}|^{{p}}\biggr)^{1/{{p}}}
\leq
CR^{2m-d+d/{{p}}-|\beta|}
\bigg(\frac{R}{r}\bigg)^{\kappa}
\end{equation*}
which is the bound~\eqref{eqn:E:2bound}.

In any case, the bound~\eqref{eqn:E:2bound} holds.
\end{proof}

\subsection{Mixed derivatives of the fundamental solution}
\label{sec:FS:mixed}

Recall that $\vec E^L_{X,j,Z_0,r,q}(Y)$ is a function of both $X$ and~$Y$. We may control derivatives in~$Y$ using Theorem~\ref{thm:derivs}, and derivatives in~$X$ using formula~\eqref{eqn:E:sym} and Theorem~\ref{thm:derivs} applied to $\vec E^{L^*}_{Y,k,Z_0,r,{q'}}$. We will also wish to control mixed derivatives, that is, derivatives in both $X$ and~$Y$. This subsection will consist of the following theorem and its proof.

\begin{theorem}\label{thm:fundamental}
Let $L$ be an operator of the form~\eqref{dfn:L} with $2\in\Upsilon_L\cap\Pi_L$, and let $q\in\Upsilon_L\cap\Pi_L$ with $1-m/d<1/q<m/d$, where $\Pi_L$ and $\Upsilon_L$ are as in Definitions \ref{dfn:ppm} and~\ref{dfn:compatible}. Then $L$ and $q$ satisfy the conditions of Definition~\ref{dfn:E:} and Theorem~\ref{thm:derivs} for all ${{p}}\in\Upsilon_L\cap \Pi_L\cap(1,2q]$ with $1-m/d<1/{{p}}<m/d$.

Let ${{p}}\in \Upsilon_L\cap\Pi_L\cap(1,2q]$ with $1-m/d<1/{{p}}<m/d$. Suppose that the Caccioppoli-Meyers inequality
\begin{multline}
\label{eqn:Meyers:FS}
\sum_{j=0}^m
|Q|^{j/d}
\biggl(\int_Q |\nabla^j\vec u|^{{p}}\biggr)^{1/{{p}}}
\\\leq
C|Q|^{1/{{p}}-1/2}
\biggl(\int_{2Q} |\vec u|^2\biggr)^{1/2}
+C|Q|^{m/d}\|L\vec u\|_{Y^{-m,{{p}}}(2Q)}
\end{multline}
holds whenever $Q\subset\RR^d$ is a cube with sides parallel to the coordinate axes and whenever $\vec u$ is a representative of an element of $Y^{m,{{p}}}(2Q)$, with $C$ independent of $\vec u$ and~$Q$. Suppose in addition this statement is valid with ${{p}}$ replaced by~$2$.

Suppose that $\alpha$ is a multiindex with $0\leq |\alpha|\leq m$. 

Then for every compact set $K\subseteq\R^d$, the function $\partial_X^\alpha \vec E^L_{X,j,Z_0,r,q}$ is in $Y^{m,{{p}}}(K)$ for almost every $X\in\R^d\setminus K$.
If $|\alpha|<\min(m-d/{{p}}',m-d/2)$ then $\partial_X^\alpha \vec E^L_{X,j,Z_0,r,q}\in Y^{m,{{p}}}(K)$ for almost every $X\in\R^d$.
Furthermore, we have the bound
\begin{align}\label{eqn:E:Bochner}
\int_{\Gamma}
\|\partial_X^\alpha{\vec E^L_{X,j,Z_0,r,q}}\|_{Y^{m,{{p}}}(Q)}^2\,dX
&\leq
CR^{2m-d+2d/{{p}}-2|\alpha|} \biggl(\frac{R}{\min(r,|Q|^{1/d})}\biggr)^\kappa
\end{align}
whenever $\Gamma$ and $Q$ are cubes with $|\Gamma|=|Q|$, $\Gamma\subset 8Q$, and either $\Gamma\subset 8Q\setminus 4Q$ or $|\alpha|<m-d/{{p}}'$.
Here $R=\max(r,|Q|^{1/d},\dist(Z_0,Q))$ and $\kappa$ is a positive constant depending on the standard parameters.

In particular, if the Caccioppoli inequality~\eqref{eqn:Meyers:FS} is valid for ${{p}}=2$, then for all multiindices $\beta$ with $0\leq |\beta|\leq m$, the mixed partial derivative $\partial_X^\alpha \partial_Y^\beta \vec E^L_{X,j,Z_0,r,q}(Y)$ exists as a locally $L^2$ function defined on $\RR^d\times\RR^d\setminus\{(X,X):X\in\RR^d\}$. Furthermore, if $Q$, $\Gamma\subset\RR^d$ are two cubes with $|Q|=|\Gamma|$ and $\Gamma\subset 8Q\setminus 4Q$, then
\begin{align}
\label{eqn:E:gradbd}
\int_{\Gamma}\int_{Q}|\partial^\alpha_X \partial^\beta_Y \vec E^L_{X,j,Z_0,r,q}(Y)|^2 dY\,dX&\leq C\biggl(\frac{R}{\min(|Q|^{1/d},r)}\biggr)^\kappa\,
R^{4m-2|\alpha|-2|\beta|}
\end{align}

If $|\alpha|<m-d/2$, then $\partial_X^\alpha \partial_Y^\beta \vec E^L_{X,j,Z_0,r,q}(Y)$ exists as a locally $L^2$ function on all of $\RR^d\times\RR^d$. Furthermore, if $Q\subset\RR^d$ is a cube, then
\begin{align}
\label{eqn:E:gradbd:lower}
\int_{Q}\int_{Q}|\partial^\alpha_X \partial^\beta_Y \vec E^L_{X,j,Z_0,r,q}(Y)|^2dY\,dX&\leq C\biggl(\frac{R}{\min(|Q|^{1/d},r)}\biggr)^\kappa\,
R^{4m-2|\alpha|-2|\beta|}
\end{align}
where $R=\max(r,|Q|^{1/d},\dist(Z_0,Q))$ and $\kappa$ is a positive constant depending on the standard parameters. If the Caccioppoli inequality is valid for~$L^*$, that is, if the bound~\eqref{eqn:Meyers:FS} is valid with ${{p}}=2$ and $L$ replaced by~$L^*$, then the bound~\eqref{eqn:E:gradbd:lower} is valid whenever $|\beta|<m-d/2$ even if $m-d/2\leq |\alpha|\leq m$.

\end{theorem}

The remainder of this subsection will be devoted to the proof of Theorem~\ref{thm:fundamental}. We remark that if $L$ is an operator of the form~\eqref{dfn:L} associated to coefficients $A$ that satisfy the Gårding inequality~\eqref{gi} and either the bound~\eqref{eqn:intro:bound} or~\eqref{eqn:intro:bound:Bochner}, then by Theorem~\ref{umpm} the condition~\eqref{eqn:Meyers:FS} is valid for $p\in \Upsilon_L\cap\Pi_L$ with $p\geq 2$. Thus the above theorem gives the bound~\eqref{eqn:E:Bochner} only for $p\geq 2$. 

Let $\alpha$ be a multiindex with $|\alpha|\leq m$. Let $1\leq j\leq N$.

Let $\eta$ be a nonnegative real-valued smooth cutoff function supported in $B(0,1)$ and integrating to~$1$ and define $\eta_\varepsilon(X')=\frac{1}{\varepsilon^d} \eta\bigl(\frac{1}{\varepsilon}X'\bigr)$ for $\varepsilon>0$. Define
\begin{equation}
\label{eqn:u:E:conv}
\vec u_{\varepsilon,\alpha,X}(Y)
= \int_{B(X,\varepsilon)}
\partial^\alpha \eta_\varepsilon(X-X') \,\vec E^L_{X',j,Z_0,r,q}(Y)\,dX'.
\end{equation}

By the weak definition of derivative and the symmetry relation~\eqref{eqn:E:sym},
\begin{equation}\label{eqn:u:dual}
(\vec u_{\varepsilon,\alpha,X}(Y))_k = \eta_\varepsilon*\overline{(\partial^\alpha \vec E^{L^*}_{Y,k,Z_0,r,{q'}})_j}( X).
\end{equation}
We now investigate $\vec u_{\varepsilon,\alpha,X}$.

\begin{lemma} With the above construction and under the conditions of Theorem~\ref{thm:derivs}, if $Q\subset\R^d$ is a cube, then $\vec u_{\varepsilon,\alpha,X} \in W^{m,{{p}}}(2Q)$,  and if $|\beta|\leq m$ then
\begin{equation}
\label{eqn:u:deriv}
\partial^\beta u_{\varepsilon,\alpha,X}(Y)
=\int \partial^\alpha \eta_\varepsilon(X-X') \,
\partial^\beta\vec E^L_{X',j,Z_0,r,q}(Y)\,dX'
.\end{equation}
\end{lemma}

\begin{proof}
If $0\leq|\beta|\leq m$ and $\vec\varphi\in C^\infty _0(B(Y_0,2\rho))$, then
\begin{align*}\int\partial^\beta\vec\varphi\cdot\vec u_{\varepsilon,\alpha,X}
&=
\int\partial^\beta\vec\varphi(Y)\cdot \int
\partial^\alpha \eta_\varepsilon(X-X') \, \vec E^L_{X',j,Z_0,r,q}(Y)\,dX' \,dY
.\end{align*}
By Theorem~\ref{thm:derivs}, $\vec E^L_{X',j,Z_0,r,q}$ and $\partial^\beta \vec E^L_{X',j,Z_0,r,q}$ are locally square integrable and thus locally integrable.
By Fubini's theorem and the definition of weak derivative,
\begin{align*}\int \partial^\beta\vec\varphi\cdot\vec u_{\varepsilon,\alpha,X}
&=
\int\partial^\alpha \eta_\varepsilon(X-X')\int \partial^\beta\vec\varphi(Y) \cdot
\vec E^L_{X',j,Z_0,r,q}(Y) \,dY\,dX'
\\&=
(-1)^{|\beta|}
\int\partial^\alpha \eta_\varepsilon(X-X') \int \vec\varphi(Y)\cdot
\partial^\beta\vec E^L_{X',j,Z_0,r,q}(Y) \,dY\,dX'
\\&=
(-1)^{|\beta|}
\int \vec\varphi(Y)\cdot
\int \partial^\alpha \eta_\varepsilon(X-X') \,
\partial^\beta\vec E^L_{X',j,Z_0,r,q}(Y)\,dX' \,dY
.\end{align*}
This is true for all test functions $\vec \varphi$, and so we have that formula~\eqref{eqn:u:deriv} is valid. Another application of Fubini's theorem yields that
\begin{align*}\int \partial^\beta\vec\varphi\cdot \vec u_{\varepsilon,\alpha,X}
&=
(-1)^{|\beta|}
\int \partial^\alpha \eta_\varepsilon(X-X')
\int
\vec\varphi(Y)\cdot\partial^\beta\vec E^L_{X',j,Z_0,r,q}(Y) \,dY\,dX'
.\end{align*}
By the bound~\eqref{eqn:E:2bound}, $\partial^\beta\vec E^L_{X',j,Z_0,r,q}\in L^{{p}}(2Q)$ and so there is some large constant $K_\varepsilon$ depending on~$\varepsilon$, $r$, $R$, and other parameters such that
\begin{align*}\biggl|\int \partial^\beta\vec\varphi\cdot\vec u_{\varepsilon,\alpha,X}
\biggr|
&\leq
K_\varepsilon\|\vec\varphi\|_{L^{{{p}}'}(2Q)}
.\end{align*}
So by the Hahn-Banach theorem and the Riesz representation theorem there is a function $\vec U\in L^{{p}}(2Q)$ with
\begin{equation*}\int_{B(Y_0,\rho)}(-1)^{|\beta|}\partial^\beta\vec\varphi(Y)\cdot\vec u_{\varepsilon,\alpha,X}(Y) \,dY
=\int_{B(Y_0,\rho)}\vec\varphi(Y)\cdot\vec U(Y) \,dY\end{equation*}
for all $\vec\varphi\in C^\infty _0(2Q)$.
By the weak definition of derivative,
$\vec U=\partial^\beta\vec u_{\varepsilon,\alpha,X}$,  so $\vec u_{\varepsilon,\alpha,X}\in W^{m,{{p}}}(2Q)$.
\end{proof}

We will need a better bound on $\|\partial^\beta\vec u_{\varepsilon,\alpha,X}\|_{L^2(Q)}$. We seek to apply the Caccioppoli (and Meyers) inequalities; we will need to compute $L\vec u_{\varepsilon,\alpha,X}$.

\begin{lemma}\label{lem:E:Lu}
With the above construction, if $L$ is of the form~\eqref{dfn:L} and $q\in\Pi_L$ with $1-m/d<1/q<m/d$ and with $L:Y^{m,q}(\R^d)\to Y^{-m,q}(\R^d)$ invertible, and if $Q\subset\R^d$ is a cube, then for all $\varepsilon\in (0,\frac{1}{2}|Q|^{1/d})$ and all ${{p}}$ with $1>1/{{p}}>\max(0,1-m/d)$, if $X\in 9Q$ and either $X\notin 3Q$ or $|\alpha|<m-d/{{p}}'$, then
\begin{align}
\label{eqn:E:Lu}
\|L\vec u_{\varepsilon,\alpha,X}\|_{Y^{-m,{{p}}}(2Q)}
&
\leq
CR^{m-d/{{p}}'-|\alpha|}\left(\frac{R}{r}\right)^\kappa
\end{align}
where $R=\max(r,|Q|^{1/d},\dist(Z_0,Q))$ and $C$ and $\kappa$ are constants depending on the standard parameters.
\end{lemma}

\begin{proof} Let $\vec\Phi\in C^\infty_0(2Q)$.
By the bound~\eqref{eqn:E:Y:norm} and the definition of $\Pi_L$, $\langle L\vec E^L_{X',j,Z_0,r,q},\vec\Phi\rangle$ denotes an absolutely convergent integral whenever $\vec\Phi\in Y^{m,q'}(\R^d)$, and furthermore, the integrand has uniform $L^1$ norm. Thus we may apply Fubini's theorem to the integral
\begin{equation*}\int \partial^\alpha\eta_\varepsilon(X-X')
\overline{\langle L\vec E^L_{X',j,Z_0,r,q},\vec\Phi\rangle}
\,dX'
\end{equation*}
and compute that
\begin{equation*}\overline{\langle L\vec u_{\varepsilon,\alpha,X},\vec\Phi\rangle}
=\int \partial^\alpha\eta_\varepsilon(X-X')
\overline{\langle L\vec E^L_{X',j,Z_0,r,q},\vec\Phi\rangle}
\,dX'.\end{equation*}
By formula~\eqref{eqn:E:Linv},
\begin{equation*}\overline{\langle L\vec E^L_{X',j,Z_0,r,q},\vec\Phi\rangle}
=\langle {L^*}\vec\Phi,\vec E^L_{X',j,Z_0,r,q}\rangle
=\overline{((((L^*)^{-1}){L^*}\vec\Phi)_j)_{Z_0,r,{q'}}(X)}
=\overline{(\Phi_j)_{Z_0,r,{q'}}(X')}.\end{equation*}
Thus,
\begin{equation*}\langle L\vec u_{\varepsilon,\alpha,X},\vec\Phi\rangle
=\eta_\varepsilon*(\partial^\alpha(\Phi_j)_{Z_0,r,{q'}})(X).
\end{equation*}
Recall that $(\Phi_j)_{Z_0,r,{q'}}=\Phi_j+P$ for some polynomial $P$ of degree at most $m-d/{q'}$ satisfying $P(Z_0+rH_i)=-\Phi_j(Z_0+rH_i)$.
As in the proof of Lemma~\ref{SFbound}, if $P(X)=\sum_{|\gamma|\leq m-d/{q'}} p_\gamma \left(\frac{X-Z_0}{r}\right)^\gamma$, then
\begin{equation*}|p_\gamma|\leq h\sup_i \Bigl|\sum_{|\gamma|\leq m-d/{q'}} p_\gamma H_i^\gamma\Bigr|
=h\sup_i |P(Z_0+rH_i)|
=h\sup_i |\Phi_j(Z_0+rH_i)|.\end{equation*}
Because $\vec\Phi\in C^\infty _0(2Q)$, we have that $\vec \Phi=0$ outside of $B(Z_0,(1+2\sqrt{d})R)$. Thus, $\vec\Phi=\vec\Phi_{Z_0,CR,{q'}}=\vec\Phi_{Z_0,CR,{{p}}'}$ because $|H_i|>1/2$ for all~$i$. Thus
\begin{equation*}|p_\gamma|\leq h\sup_i |\vec\Phi_{Z_0,CR,{{p}}'}(Z_0+rH_i)|\end{equation*}
and by Lemma~\ref{SFbound}, since ${{p}}'>d/m$,
\begin{equation*}|p_\gamma|\leq CR^{m-d/{{p}}'}\|\vec\Phi\|_{Y^{m,{{p}}'}(\RR^d)}.\end{equation*}
Thus
\begin{align*}|\langle L\vec u_{\varepsilon,\alpha,X},\vec\Phi\rangle|
&=
|\eta_\varepsilon*(\partial^\alpha P)(X)+\eta_\varepsilon*(\partial^\alpha \Phi_j)(X)|
\\&\leq  |\eta_\varepsilon*(\partial^\alpha \Phi_j)(X)|
+
CR^{m-d/{{p}}'-|\alpha|}\left(\frac{R}{r}\right)^\kappa \|\vec\Phi\|_{Y^{m,{{p}}'}(\RR^d)}
.\end{align*}
If $X\notin {3Q}$ and $0<\varepsilon<\frac{1}{2}|Q|^{1/d}=\dist(2Q,\RR^d\setminus 3Q)$, then $\eta_\varepsilon*(\partial^\alpha \Phi_j)(X)=0$. If $|\alpha|<m-d/{{p}}'$, then again by Lemma~\ref{SFbound} applied to $\Phi_j=(\Phi_j)_{Z_0,CR,{{p}}'}$, if $0<\varepsilon<R$ then
\begin{align*}|\langle L\vec u_{\varepsilon,\alpha,X},\vec\Phi\rangle|
&
\leq
CR^{m-d/{{p}}'-|\alpha|}\left(\frac{R}{r}\right)^\kappa \|\vec\Phi\|_{Y^{m,{{p}}'}(\RR^d)}
.\end{align*}
This completes the proof.
\end{proof}

We have established that $\vec u_{\varepsilon,\alpha,X}\in W^{m,{{p}}}(2Q)$ and have a bound on~$L\vec u_{\varepsilon,\alpha,X}$. We will now bound the derivatives of $\vec u_{\varepsilon,\alpha,X}$.

\begin{lemma}\label{lem:u:Bochner}
Let $L$, $q$, and ${{p}}$ satisfy the conditions of Theorem~\ref{thm:derivs}. Suppose in addition that the conclusion~\eqref{eqn:E:Lu} of Lemma~\ref{lem:E:Lu} is valid (under the given conditions on $\varepsilon$, $X$ and~$\alpha$). Let $Q\subset\R^d$ be a cube.
Suppose further that the Caccioppoli-Meyers estimate \eqref{eqn:Meyers:FS}
is valid in~$Q$ for all $\vec u\in Y^{m,{{p}}}(2Q)$. Let $\Gamma\subset 8Q$ be a cube with $|\Gamma|=|Q|$.

Then for all $\varepsilon\in (0,\frac{1}{2}|Q|^{1/d})$, if either $\Gamma\subset 8Q\setminus 4Q$ or  $|\alpha|<m-d/{{p}}'$, then
\begin{align*}\int_{\Gamma}
\|\vec u_{\varepsilon,\alpha,X}\|_{Y^{m,{{p}}}(Q)}^2\,dX
&\leq
C|Q|^{2/{{p}}-1-2m/d}
R^{4m-d-2|\alpha|} \biggl(\frac{R}{r}\biggr)^\kappa
\end{align*}
where $R=\max(r,|Q|^{1/d},\dist(Z_0,Q))$ and $C$ and $\kappa$ are constants depending on the standard parameters.

In particular, if ${{p}}=2$ then for all $\beta$ with $|\beta|\leq m$ we have that
\begin{align*}
\int_{\Gamma}
\int_{Q} |\partial^\beta\vec u_{\varepsilon,\alpha,X}(Y)|^{2}\,dY\,dX
&\leq
C|Q|^{-2|\beta|/d}
R^{4m-d-2|\alpha|} \biggl(\frac{R}{r}\biggr)^\kappa
.\end{align*}
\end{lemma}

\begin{proof}
Applying the bounds \eqref{eqn:E:Lu} and~\eqref{eqn:Meyers:FS} and Lemma~\ref{evlem} to $\vec u=\vec u_{\varepsilon,\alpha,X}$ yields
\begin{equation*}
\|\vec u_{\varepsilon,\alpha,X}\|_{Y^{m,{{p}}}(Q)}
\leq
C|Q|^{1/{{p}}-1/2-m/d}
\biggl(\int_{2Q} |\vec u_{\varepsilon,\alpha,X}|^2\biggr)^{1/2}
+CR^{m-d/{{p}}'-|\alpha|}\biggl(\frac{R}{r}\biggr)^\kappa
.\end{equation*}
By formula~\eqref{eqn:u:dual}, the $L^2$ boundedness of convolution, and the bound~\eqref{eqn:E:2bound}, if $\varepsilon$ is small enough, $\Gamma\subset 8Q$ and $Y\in 2Q$, then
\begin{align*}\int_{\Gamma} |\vec u_{\varepsilon,\alpha,X}(Y)|^2\,dX
&\leq \sup_{1\leq k\leq N} \int_{2\Gamma} |\partial^\alpha \vec E^{L^*}_{Y,k,Z_0,r}( X)|^2\,dX
\leq
C\biggl(\frac{R}{r}\biggr)^\kappa R^{4m-d-2|\alpha|}.
\end{align*}
Combining the above bounds completes the proof.
\end{proof}

We now prove Theorem~\ref{thm:fundamental}. The assumptions of Theorem~\ref{thm:fundamental} include the assumptions of Theorem~\ref{thm:derivs}, Lemma~\ref{lem:E:Lu} and Lemma~\ref{lem:u:Bochner} with ${{p}}=2$; we will use only
the conclusions of Lemma~\ref{lem:u:Bochner} and the definitions \eqref{eqn:u:E:conv} and~\eqref{eqn:u:dual} of $\vec u_{\varepsilon,\alpha,X}$.

The Lebesgue space $L^2(\Gamma\times Q)$ is weakly sequentially compact.
Thus, because $\{\vec u_{\varepsilon,\alpha,X}\}_{0<\varepsilon<\frac{1}{2}|Q|^{1/d}}$ is a bounded set in $L^2(\Gamma\times Q)$, if $0\leq |\beta|\leq m$, there is a function $\vec E_{\alpha,\beta,j}$ with
\begin{align*}
\int_\Gamma\int_Q |\vec E_{\alpha,\beta,j}(X,Y)|^{2}\,dY\,dX
\leq CR^{4m-2|\alpha|-2|\beta|}\biggl(\frac{R}{\min(r,|Q|^{1/d})}\biggr)^\kappa
\end{align*}
and a sequence of positive numbers $\varepsilon_i$ with $\varepsilon_i\to 0$ and such that, for all  $\vec\varphi\in L^2(\Gamma\times Q)$, we have that
\begin{align*}
\int_\Gamma \int_Q \vec\varphi(X,Y)\cdot\vec E_{\alpha,\beta,j}(X,Y)
\,dY\,dX
&=
\lim_{i\to\infty }\int_\Gamma \int_Q \vec\varphi(X,Y)\cdot
\partial^\beta\vec u_{\varepsilon_i,\alpha,X}(Y)
\,dY\,dX
.\end{align*}
Integrating by parts and applying formula~\eqref{eqn:u:dual}, we see that if $\vec\varphi$ is smooth and compactly supported then
\begin{multline*}
\int_\Gamma \int_Q  \varphi_k(X,Y)\,(\vec E_{\alpha,\beta,j}(X,Y))_k
\,dY\,dX
\\\begin{aligned}&=
(-1)^{|\beta|}
\lim_{i\to\infty }\int_\Gamma \int_Q  \partial_Y^\beta\varphi_k(X,Y)\,
(\vec u_{\varepsilon_i,\alpha,X}(Y))_k
\,dY\,dX
\\&=
(-1)^{|\beta|}
\lim_{i\to\infty }\int_\Gamma \int_Q  \partial_Y^\beta\varphi_k(X,Y)\,
\eta_{\varepsilon_i}*\overline{\partial^\alpha (\vec E^{L^*}_{Y,k,Z_0,r,q'})_j}( X)\,dY\,dX
.\end{aligned}\end{multline*}
Using properties of convolutions, we see that
\begin{multline*}
\int_\Gamma \int_Q  \varphi_k(X,Y)\,(\vec E_{\alpha,\beta,j}(X,Y))_k
\,dY\,dX
\\=
(-1)^{|\beta|+|\alpha|}
\lim_{i\to\infty }\int_\Gamma \int_Q  \eta_{\varepsilon_i}*_X\partial_X^\alpha \partial_Y^\beta\varphi_k(X,Y)\,
\overline{(\vec E^{L^*}_{Y,k,Z_0,r,q'})_j}( X)\,dY\,dX
\end{multline*}
where $*_X$ denotes convolution in the $X$ variable only. By the dominated convergence theorem,
\begin{multline*}
\int_\Gamma \int_Q  \varphi_k(X,Y)\,(\vec E_{\alpha,\beta,j}(X,Y))_k
\,dY\,dX
\\=
(-1)^{|\beta|+|\alpha|}
\int_\Gamma \int_Q  \partial_X^\alpha \partial_Y^\beta\varphi_k(X,Y)\,
\overline{(\vec E^{L^*}_{Y,k,Z_0,r,q'})_j( X)}\,dY\,dX
\end{multline*}
and so $(\vec E_{\alpha,\beta,j}(X,Y))_k = \partial_X^\alpha\partial_Y^\beta \overline{(\vec E^{L^*}_{Y,k,Z_0,r,q'})_j( X)} = \partial_X^\alpha\partial_Y^\beta\overline{(\vec E^L_{X,j,Z_0,r,q})_k( Y)}$ in the weak sense. Furthermore, we may derive bounds on $\vec E_{\alpha,\beta,j}(X,Y)$ from our bounds on $\vec u_{\varepsilon_i,\alpha,X}$. Thus, by Lemma~\ref{lem:u:Bochner}, we have the bound~\eqref{eqn:E:gradbd} and the bound~\eqref{eqn:E:gradbd:lower} in the case $|\alpha|<m-d/2$.

Suppose $|\beta|<m-d/2$ and the Caccioppoli inequality (\eqref{eqn:Meyers:FS} with~${{p}}=2$) holds for~$L^*$. By Remark~\ref{rmk:L:sym}, we may thus apply the above results to~$\vec E^{L^*}$. By the  bound~\eqref{eqn:E:gradbd:lower} for~$\vec E^{L^*}$, if $|\beta|<m-d/2$ then
\begin{equation*}\int_{Q}\int_{Q}|\partial^\alpha_Y \partial^\beta_X \vec E^{L^*}_{Y,k,Z_0,r,q'}(X)|^2 dX\,dY
\leq C\biggl(\frac{R}{\min(|Q|^{1/d},r)}\biggr)^\kappa \,
R^{4m-2|\alpha|-2|\beta|}.
\end{equation*}
Applying formula~\eqref{eqn:E:sym} yields the bound~\eqref{eqn:E:gradbd:lower} in the case $|\beta|<m-d/2$.

The space $L^2(\Gamma;L^{{{p}}_\beta}(Q))$ is a Bochner space, and so is a reflexive Banach space with dual $L^2(\Gamma;L^{({{p}}_\beta)'}(Q))$. By Lemma~\ref{lem:u:Bochner}, we have that if $\vec\varphi\in L^2(\Gamma,Q)$ then
\begin{multline*}\biggl|\int_\Gamma\int_Q \vec\varphi(X,Y)\cdot
\partial_X^\alpha\partial_Y^\beta{\vec E^L_{X,j,Z_0,r,q}( Y)}
\,dY\,dX\biggr|
\\\begin{aligned}
&=\lim_{i\to\infty }
\biggl|\int_\Gamma\int_Q \vec\varphi(X,Y)\cdot
\partial_Y^\beta\overline{\vec u_{\varepsilon_i,\alpha,X}( Y)}
\,dY\,dX\biggr|
\\&\leq
\biggl(\int_\Gamma
\biggl(\int_Q|\vec\varphi(X,Y)|^{({{p}}_\beta)'}dY\biggr)^{2/({{p}}_\beta)'}\,dX\biggr)^{1/2}
CR^{\theta} \biggl(\frac{R}{\min(r,|Q|^{1/d})}\biggr)^\kappa
\end{aligned}
\end{multline*}
where $\theta=m-d/2+d/{{p}}-|\alpha|$.
The space $L^2(\Gamma\times Q)$ is dense in $L^2(\Gamma;L^{({{p}}_\beta)'}(Q))$. Thus, this bound is valid for all $\vec\varphi\in L^2(\Gamma;L^{({{p}}_\beta)'}(Q))$, and so $\partial_X^\alpha\partial_Y^\beta{\vec E^L_{X,j,Z_0,r,q}( Y)}\in L^2(\Gamma;L^{{{p}}_\beta}(Q))$ and satisfies the bound~\eqref{eqn:E:Bochner}.

\subsection{Extraneous parameters}
\label{sec:FS:indep}

The fundamental solution $E^L_{X,j,Z_0,r,q}(Y)$ of Definition~\ref{dfn:E:} depends on the parameters $Z_0$, $r$, and~$q$ in a somewhat artificial way: they are used only to normalize $T_{X,j,Z_0,r,q}$ and $E^L_{X,j,Z_0,r,q}$. We would like (to the extent possible) to remove the dependencies on $Z_0$, $r$, and~$q$. The following lemma will allow us to remove (or at least reduce) these dependencies. 

\begin{lemma}\label{lem:indep}
Let $q_1$, $q_2\in (1,\infty)$.
Let $L$ satisfy the conditions of Definition~\ref{dfn:E:} for both $q=q_1$ and~$q=q_2$. Suppose that $L$ is compatible in the sense that if $S\in Y^{-m,q_1}(\R^d)\cap Y^{-m,q_2}(\R^d) $, then $L^{-1}S\in Y^{m,q_1}(\R^d)\cap Y^{m,q_2}(\R^d)$.

Suppose that $\alpha$ and $\beta$ are multiindices such that
\begin{align*}
\max(m-d/{q_1'},m-d/{q_2'})<|\alpha|&\leq m
,\\
\max(m-d/{q_1},m-d/{q_2})<|\beta|&\leq m.
\end{align*}
Let $1\leq j\leq N$, $r_1>0$, $r_2>0$, $Z_1\in\RR^d$, and $Z_2\in\RR^d$. Suppose that, for $i\in\{1,2\}$, the mixed derivate $\partial^\alpha_X \partial^\beta_Y \vec E^{L}_{X,j,Z_i,r_i,{q_i}}(Y)$ exists almost everywhere and is locally integrable on $\R^d\times\R^d\setminus\{(X,X):X\in\R^d\}$.

Then we have that
\begin{align}
\label{eqn:E:indep}
\partial^\alpha_X \partial^\beta_Y \vec E^{L}_{X,j,Z_1,r_1,{q_1}}(Y)
&= \partial^\alpha_X \partial^\beta_Y \vec E^{L}_{X,j,Z_2,r_2,{q_2}}(Y)
\end{align}
for almost every $(X,Y)\in \RR^d\times\RR^d$.
\end{lemma}

As noted after Theorem~\ref{thm:derivs}, if $\Upsilon_L$ is as in Definition~\ref{dfn:compatible} and $q_1$, $q_2\in\Upsilon_L$, then $L$, $q_1$, and $q_2$ satisfy the conditions of the lemma.

Under the conditions of Theorem~\ref{thm:fundamental}, existence and local integrability of the mixed partial derivative is valid. Furthermore, under these conditions we may combine formulas~\eqref{eqn:E:indep} and~\eqref{eqn:E:gradbd} to see that if $\alpha$ and $\beta$ are multiindices with $m-d/{q'}<|\alpha|\leq m$ and $m-d/q<|\beta|\leq m$, then by choosing $Z_0$ and $r$ appropriately, we have that if $\rho=|X_0-Y_0|/8$, then
\begin{equation}\label{Egradrho}
\int_{B(X_0,\rho)}\int_{B(Y_0,\rho)}|\partial^\alpha_X \partial^\beta_Y \vec E^{L}_{X,j,Z_0,r,q}(Y)|^2dY\,dX\leq C\rho^{4m-2|\alpha|-2|\beta|}
.\end{equation}

\begin{proof}[Proof of Lemma~\ref{lem:indep}]
Fix some such $j$, $\alpha$, and~$\beta$.

Let $\eta$ and $\varphi$ be smooth functions with disjoint compact support. Let $T$ be given by
\begin{equation*}\langle T,\vec\Phi\rangle = \int_{\R^d} \Phi_k(Y) \,\partial^\beta \eta(Y)\,dY = (-1)^{|\beta|}\int_{\R^d} \partial^\beta \Phi_k(Y) \,\eta(Y)\,dY.\end{equation*}
Because $|\beta|>m-d/q_i$, we have that if $\vec\Phi\in Y^{m,q_i}(\R^d)$ then $\partial^\beta \Phi_k$ is well defined as a $L^{(q_i)_\beta}(\RR^d)$-function (that is, up to sets of measure zero, not up to polynomials), and so $T\in Y^{-m,q_i'}(\RR^d)$ with no normalization necessary. 

By formula~\eqref{eqn:E:Linv},
\begin{equation*}(((L^*)^{-1}T)_j)_{Z_i,r_i,q_i'}(X) =\overline{\langle T,\vec E^{L}_{X,j,Z_i,r_i,{q_i}}\rangle}.\end{equation*}
By duality, if $T\in Y^{-m,q_1'}(\R^d)\cap Y^{-m,q_2'}(\R^d)$, then $(L^*)^{-1}T=(L^{-1})^*T\in Y^{m,q_1'}(\R^d)\cap Y^{m,q_2'}(\R^d)$. That is, the inverses are identical whether we consider ${L^*}:Y^{m,{q_1'}}(\RR^d)\to Y^{-m,{q_1'}}(\RR^d)$ or ${L^*}:Y^{m,{q_2'}}(\RR^d)\to Y^{-m,{q_2'}}(\RR^d)$.
Furthermore, $|\alpha|>m-d/q_i'$ and so $\partial^\alpha((L^*)^{-1}T)_j$ is a well defined locally integrable function that does not depend on $Z_i$, $r_i$, or~$q_i$. Thus
\begin{multline*}
\iint {\partial^\alpha \varphi(X)}\, \partial^\beta \eta(Y)\,(\vec E^{L}_{X,j,Z_1,r_1,q_1}(Y))_k \,dY\,dX
\\
=\int {\partial^\alpha \varphi(X)}\,\langle T,\vec E^{L}_{X,j,Z_1,r_1,q_1}\rangle\,dX
=
\int {\partial^\alpha \varphi(X)}\, \overline{((L^*)^{-1}T)_j(X)}\,dX
\\
=
\iint {\partial^\alpha \varphi(X)}\, \partial^\beta \eta(Y)\,(\vec E^{L}_{X,j,Z_2,r_2,q_2}(Y))_k \,dY\,dX
.\end{multline*}
Applying the definition of weak derivative, we see that
\begin{multline*}\iint {\varphi(X)} \,\eta(Y)\,(\partial_Y^\alpha \partial_Y^\beta \vec E^{L}_{X,j,Z_1,r_1,{q_1}}(Y))_k \, dY\,dX
\\=
\iint {\varphi(X) } \,\eta(Y)\,(\partial_Y^\alpha \partial_Y^\beta \vec E^{L}_{X,j,Z_2,r_2,{q_2}}(Y))_k \, dY\,dX
\end{multline*}
for any smooth functions with disjoint compact support. By the Lebesgue differentiation theorem, formula~\eqref{eqn:E:indep} is valid for almost every $(X,Y)\in\RR^d\times\RR^d$.
\end{proof}

We now consider the dependency of $\vec E^L_{X,j,Z_0,r,q}$ on $q$ in more detail. Define
\begin{equation*}\Xi_q=\{(\alpha,\beta):\alpha,\beta\text{ are multiindices, } m-d/q'<|\alpha|\leq m,\text{ and }m-d/q<|\beta|\leq m\}.\end{equation*}
$\Xi_q$ is illustrated in Figure~\ref{fig:zeta:xi}.
By Lemma~\ref{lem:indep}, if $(\alpha,\beta)\in \Xi_q$, then $\partial_X^\alpha\partial_Y^\beta \vec E^{L}_{X,j,Z_0,r,q}(Y)$ is independent of~$Z_0$ and~$r$. Thus, we may largely ignore the dependency on $Z_0$ and~$r$. 

\begin{figure}

\def\offset{0.1}
\def\bigoffset{0.2}

\begin{tikzpicture}[scale=0.6]
\def\m{6}
\def\d{5}
\def\horistart{4}
\def\vertstart{4}
\def\overq{0.54}
\node at (0,\m) [left] {$m$};
\node at (0,\m-\d) [left] {$m-d$};
\node at (\m,0) [below] {$\vphantom{d}m$};

\draw [white!50!black] (\m-\d,\m)--(\m,\m-\d);
\foreach \k in {0,...,\m}{
	\draw [white!50!black] (0,\k)--(\m,\k);
	\draw [white!50!black] (\k,0)--(\k,\m);
}
\draw [->] (-1,0)--(\m+1,0) node [below] {$|\alpha|$};
\draw [->] (0,-1)--(0,\m+1) node [left] {$|\beta|$};

\fill [white!50!black]
	(\m-\d*\overq+\offset,\m-\d+\d*\overq+\offset)
	rectangle (\m,\m);
\foreach \k in {\vertstart,...,\m}{
\foreach \j in {\horistart,...,\m}{
	\fill [black] (\j,\k) circle (1.5pt);
}}

\draw [black]
	(\m-\d*\overq,\m-\d+\d*\overq) circle (2pt);
\node at (\m-\d*\overq,\m-\d+\d*\overq) [left] {$(m-\frac{d}{q'},m-\frac{d}{q})$};

\end{tikzpicture}
\begin{tikzpicture}[scale=0.6]
\def\m{6}
\def\d{4}
\def\horistart{5}
\def\vertstart{5}
\def\overq{0.5}
\node at (0,\m) [left] {$m$};
\node at (0,\m-\d) [left] {$m-d$};
\node at (\m,0) [below] {$\vphantom{d}m$};

\draw [white!50!black] (\m-\d,\m)--(\m,\m-\d);
\foreach \k in {0,...,\m}{
	\draw [white!50!black] (0,\k)--(\m,\k);
	\draw [white!50!black] (\k,0)--(\k,\m);
}
\draw [->] (-1,0)--(\m+1,0) node [below] {$|\alpha|$};
\draw [->] (0,-1)--(0,\m+1) node [left] {$|\beta|$};

\fill [white!50!black]
	(\m-\d*\overq+\offset,\m-\d+\d*\overq+\offset)
	rectangle (\m,\m);
\foreach \k in {\vertstart,...,\m}{
\foreach \j in {\horistart,...,\m}{
	\fill [black] (\j,\k) circle (1.5pt);
}}

\draw [black]
	(\m-\d*\overq,\m-\d+\d*\overq) circle (2pt);
\node at (\m-\d*\overq,\m-\d+\d*\overq) [left] {$(m-\frac{d}{2},m-\frac{d}{2})$};

\end{tikzpicture}

\begin{tikzpicture}[scale=0.6]
\def\m{6}
\def\d{4}
\def\horistart{5}
\def\vertstart{4}
\def\overq{0.4}
\node at (0,\m) [left] {$m$};
\node at (0,\m-\d) [left] {$m-d$};
\node at (\m,0) [below] {$\vphantom{d}m$};

\draw [white!50!black] (\m-\d,\m)--(\m,\m-\d);
\foreach \k in {0,...,\m}{
	\draw [white!50!black] (0,\k)--(\m,\k);
	\draw [white!50!black] (\k,0)--(\k,\m);
}
\draw [->] (-1,0)--(\m+1,0) node [below] {$|\alpha|$};
\draw [->] (0,-1)--(0,\m+1) node [left] {$|\beta|$};

\fill [white!50!black]
	(\m-\d*\overq+\offset,\m-\d+\d*\overq+\offset)
	rectangle (\m,\m);
\foreach \k in {\vertstart,...,\m}{
\foreach \j in {\horistart,...,\m}{
	\fill [black] (\j,\k) circle (1.5pt);
}}

\draw [black]
	(\m-\d*\overq,\m-\d+\d*\overq) circle (2pt);
\node at (\m-\d*\overq,\m-\d+\d*\overq) [left] {$(m-\frac{d}{q'},m-\frac{d}{q})$};

\end{tikzpicture}
\begin{tikzpicture}[scale=0.6]
\def\m{6}
\def\d{4}
\def\horistart{4}
\def\vertstart{5}
\def\overq{0.6}
\node at (0,\m) [left] {$m$};
\node at (0,\m-\d) [left] {$m-d$};
\node at (\m,0) [below] {$\vphantom{d}m$};

\draw [white!50!black] (\m-\d,\m)--(\m,\m-\d);
\foreach \k in {0,...,\m}{
	\draw [white!50!black] (0,\k)--(\m,\k);
	\draw [white!50!black] (\k,0)--(\k,\m);
}
\draw [->] (-1,0)--(\m+1,0) node [below] {$|\alpha|$};
\draw [->] (0,-1)--(0,\m+1) node [left] {$|\beta|$};

\fill [white!50!black]
	(\m-\d*\overq+\offset,\m-\d+\d*\overq+\offset)
	rectangle (\m,\m);
\foreach \k in {\vertstart,...,\m}{
\foreach \j in {\horistart,...,\m}{
	\fill [black] (\j,\k) circle (1.5pt);
}}

\draw [black]
	(\m-\d*\overq,\m-\d+\d*\overq) circle (2pt);
\node at (\m-\d*\overq,\m-\d+\d*\overq) [left] {$(m-\frac{d}{q'},m-\frac{d}{q})$};

\end{tikzpicture}

\caption{
$\Xi_q$ denotes the set of lattice points in the gray rectangle (including the top and right edges, but not the bottom or left edges). In odd dimensions (upper left), $\Xi_q=\Xi_2$ if $q$ is sufficiently close to~$2$. In even dimensions, $\Xi_2$ (upper right) is a proper subset of $\Xi_q$ for all $q$ sufficiently close to~$2$ but either less than~$2$ (bottom left) or greater than~$2$ (bottom right).
}
\label{fig:zeta:xi}

\end{figure}
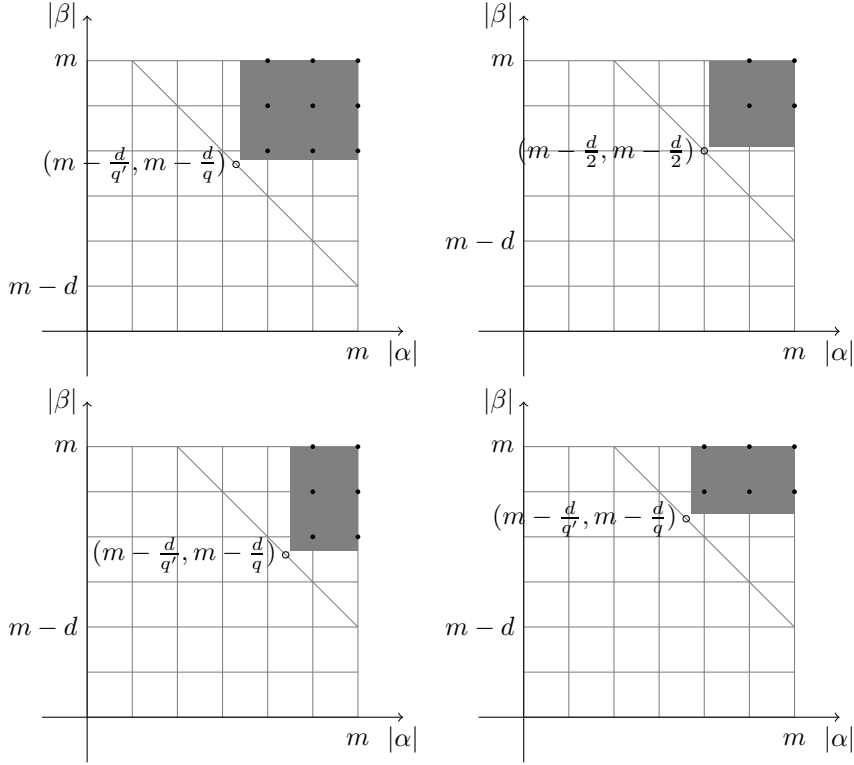

However, the range $\Xi_q$ of acceptable derivatives does depend on~$q$. We would like to discuss this dependency in more detail.

\subsubsection{Odd dimensions}\label{sec:odd}
In odd dimensions, we will let our fundamental solution be $\vec E^L_{X,j}(Y)=\vec E^{L}_{X,j,Z_0,r,2}(Y)$.
In light of the Gårding inequality~\eqref{gi} and the Lax-Milgram lemma, and their consequence Lemma~\ref{lem:L:invertible}, $q=2$ is the most natural value.
A straightforward computation yields that if the dimension $d$ is odd, then $\Xi_q=\Xi_2$ whenever $\frac{2d}{d+1}\leq q\leq \frac{2d}{d-1}$, that is, for all $q$ sufficiently close to~$2$.

Note that for general rough coefficients, it may be that $q\in \Upsilon_L$ and so $q$ satisfies the conditions of Definition~\ref{dfn:E:} only for $q$ very close to~$2$ (and in particular may not satisfy these conditions for any $q$ outside of~$[\frac{2d}{d+1},\frac{2d}{d-1}]$); thus, we cannot in general expect to improve upon~$\vec E^{L}_{X,j,Z_0,r,2}(Y)$ in terms of the number of derivatives independent of $Z_0$,~$r$.

\subsubsection{Even dimensions} The situation in even dimensions is more complicated. In this case, if $\frac{2d}{d+2}<q<\frac{2d}{d-2}$ and $q\neq 2$, then $\Xi_q\supsetneq \Xi_2$; that is, $\vec E^{L}_{X,j,Z_0,r,q}(Y)$ has strictly more derivatives independent of $Z_0$, $r$ than $\vec E^{L}_{X,j,Z_0,r,2}(Y)$. See Figure~\ref{fig:zeta:xi:2}. However, if $\frac{2d}{d+2}<q<2<s<\frac{2d}{d-2}$ and $m\geq d/2$, then $\Xi_q$ and $\Xi_s$ are not equal; indeed we have both of the two noninclusions  $\Xi_q\not\subseteq\Xi_s$ and $\Xi_s\not\subseteq\Xi_q$. Thus, neither of the functions $\vec E^{L}_{X,j,Z_0,r,q}$ and $\vec E^{L}_{X,j,Z_0,r,s}$ is entirely satisfactory; we thus wish to define a new fundamental solution $\vec E^L_{X,j,Z_0,r}(Y)$ with the correct derivatives for all multiindices in either $\Xi_s$ or~$\Xi_q$.

\begin{theorem}\label{thm:renormalize}
Let $d\geq 2$ be an even integer and let $m\in \NN$.
Let $L$ be such that there exists an open neighborhood $\widetilde\Upsilon_L$ of $2$ such that if $q$, $q_1$, $q_2\in \widetilde\Upsilon_L$, then $L$ and $q$ satisfy the conditions of Definition~\ref{dfn:E:}, the bound~\eqref{eqn:E:gradbd} is valid, and formula~\eqref{eqn:E:indep} is true whenever $(\alpha,\beta) \in \Xi_{q_1}\cap\Xi_{q_2}$.

Then there exists a function $\vec E^{L}_{X,j}(Y)$ such that if $q\in \widetilde\Upsilon_L\cap \bigl(\frac{2d}{d+2},\frac{2d}{d-2}\bigr)$ (or $q\in \widetilde\Upsilon_L\cap (1,\infty)$ if $d=2$), then
\begin{equation}
\label{eqn:norm:even}
\partial_X^\alpha\partial_Y^\beta \vec E^{L}_{X,j,Z_0,r,q}(Y)= \partial_X^\alpha\partial_Y^\beta \vec E^{L}_{X,j}(Y)
\quad
\text{for all $(\alpha,\beta)\in \Xi_q$.}
\end{equation}
Furthermore, $(\alpha,\beta)\in\Xi_q$ for some such $q$ if and only if
\begin{align}
\label{cond:zeta}m-d/2\leq |\alpha|\leq m, \qquad
%\\\label{cond:xi}
m-d/2\leq |\beta|\leq m,\qquad
%\\\label{cond:zeta:xi}
2m-d<|\alpha|+|\beta|.
\end{align}
\end{theorem}

\begin{figure}

\def\offset{0.1}
\def\bigoffset{0.4}

\begin{tikzpicture}[scale=0.6]
\def\m{6}
\def\d{4}
\def\horistart{4}
\def\vertstart{5}
\def\overq{0.5}
\draw (0,\m)-- (-0.2,\m) node [left] {$m$};
\draw (0,\m-\d)-- (-0.2,\m-\d) node [left] {$m-d$};
\draw (\m,0)-- (\m,-0.2) node [below] {$\vphantom{d}m$};
\draw (\m-\d,0)-- (\m-\d,-0.2) node [below] {$\vphantom{d}m-d$};

\draw [white!50!black] (\m-\d,\m)--(\m,\m-\d);
\foreach \k in {0,...,\m}{
	\draw [white!50!black] (0,\k)--(\m,\k);
	\draw [white!50!black] (\k,0)--(\k,\m);
}
\draw [->] (-1,0)--(\m+1,0) node [below] {$|\alpha|$};
\draw [->] (0,-1)--(0,\m+1) node [left] {$|\beta|$};

\fill [white!50!black]
	(\m-\d/2-1+\bigoffset+\offset,\m) --
	(\m-\d/2-1+\bigoffset+\offset,
		\m-\d/2+1-\bigoffset+\offset) --
	(\m-\d/2+1-\bigoffset+\offset,
		\m-\d/2-1+\bigoffset+\offset) --
	(\m,\m-\d/2-1+\bigoffset+\offset) --
	(\m,\m)--cycle;
\foreach \k in {\vertstart,...,\m}{
	\fill [black] (\k,\horistart) circle (1.5pt);
\foreach \j in {\horistart,...,\m}{
	\fill [black] (\j,\k) circle (1.5pt);
}}

\end{tikzpicture}

\caption{
The set of points that satisfy Conditions~(\ref{cond:zeta}--\ref{cond:zeta:xi}).
}
\label{fig:zeta:xi:2}

\end{figure}
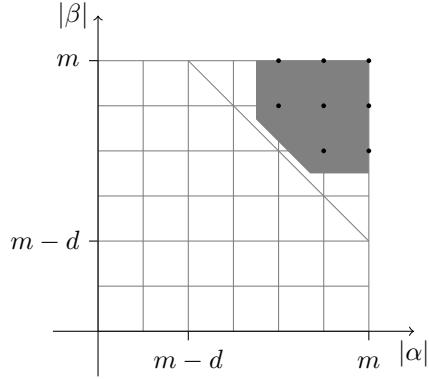

\begin{proof}
If $q$, $\widetilde q\in(\frac{2d}{d+2},2)$, then $\Xi_q=\Xi_{\widetilde q}$  and so if $q$, $\widetilde q\in\widetilde\Upsilon_L$ then $\partial_X^\alpha\partial_Y^\beta \vec E^{L}_{X,j,Z_0,r,q}(Y)=\partial_X^\alpha\partial_Y^\beta \vec E^{L}_{X,j,Z_0,r,\widetilde q}(Y)$ for all $(\alpha,\beta)\in\Xi_{q}$. 
The same is true if $q$, $\widetilde q \in(2,\frac{2d}{d-2})\cap\widetilde\Upsilon_L$. 
Thus, it suffices to find a function $\vec E^L_{X,j}$ such that the condition~\eqref{eqn:norm:even} is valid for a single $q\in(\frac{2d}{d+2},2)\cap\widetilde\Upsilon_L$ and a single $q\in(2,\frac{2d}{d-2})\cap\widetilde\Upsilon_L$.

Fix $q$, $s\in \widetilde\Upsilon_L$ such that $\frac{2d}{d+2}<q<2<s<\frac{2d}{d-2}$. By assumption, some such $q$ and~$s$ exist. An elementary computation shows that $(\alpha,\beta)\in\Xi_q\cup\Xi_s$ if and only if Condition~\eqref{cond:zeta} is true.

For each $W\in\R^d$ and each $\zeta$ with $|\zeta|=m-d/2$, define
\begin{equation*}\vec G_{j,\zeta,Y}(W) = \partial_W^\zeta \vec E^{L}_{W,j,Z_0,r,s}(Y)-\partial_W^\zeta \vec E^{L}_{W,j,Z_0,r,q}(Y)
.\end{equation*}
By Remark~\ref{rmk:L:sym}, formula~\eqref{eqn:E:sym}, and the bound~\eqref{eqn:E:Y:norm} applied to $\vec E^{L^*}$, for each $Y\in\RR^d$, $\vec G_{j,\zeta,Y}$ is a locally integrable function.

Now, fix some $W_0\in\RR^d$ and some $\rho>0$, and define
\begin{equation*}\vec E^{L}_{X,j}(Y)=\vec E^{L}_{X,j,Z_0,r,q}(Y) - \sum_{|\zeta|=m-d/2} \frac{1}{\zeta!}\,X^\zeta \fint_{B(W_0,\rho)}\vec G_{j,\zeta,Y}(W)\,dW.\end{equation*}

If $(\alpha,\beta)\in \Xi_q$, then because $q<2$ we have that $|\alpha|>m-d/2$. Thus, $\partial_X^\alpha X^\zeta=0$ and so
\begin{equation*}\partial_X^\alpha\partial_Y^\beta \vec E^{L}_{X,j}(Y)=\partial_X^\alpha\partial_Y^\beta \vec E^{L}_{X,j,Z_0,r,q}(Y).\end{equation*}

Now, suppose that $(\alpha,\beta)\in\Xi_q$. If $(\alpha,\beta)\in\Xi_s$ then
\begin{equation*}\partial_X^\alpha\partial_Y^\beta \vec E^{L}_{X,j,Z_0,r,s}(Y)=\partial_X^\alpha\partial_Y^\beta \vec E^{L}_{X,j,Z_0,r,q}(Y)=\partial_X^\alpha\partial_Y^\beta \vec E^{L}_{X,j}(Y).\end{equation*}
We thus need only consider the case $(\alpha,\beta)\in\Xi_q\setminus \Xi_s$; this implies that $|\beta|>m-d/2$ and $|\alpha|=m-d/2$. We then compute that
\begin{equation*}
\partial_X^\alpha\partial_Y^\beta \vec E^{L}_{X,j}(Y)
=\partial_X^\alpha\partial_Y^\beta \vec E^{L}_{X,j,Z_0,r,q}(Y)
- \partial_Y^\beta\fint_{B(W_0,\rho)}\vec G_{j,\alpha,Y}(W)\,dW.\end{equation*}
Let $\vec e_k$ be a unit coordinate vector and let $\xi=\alpha+\vec e_k$. Then $m\geq|\xi|>m-d/2$ and so $(\beta,\xi)\in \Xi_q\cap \Xi_s$. Thus
\begin{equation*}\partial_{W_k} \partial_Y^\beta \vec G_{j,\alpha,Y}(W) = \partial_W^\xi \partial_Y^\beta \vec E^{L}_{W,j,Z_0,r,q}(Y)-\partial_W^\xi \partial_Y^\beta \vec E^{L}_{W,j,Z_0,r,s}(Y)
=0\end{equation*}
as locally integrable functions; that is, for almost every $Y\in\R^d$ we have that
$\partial_Y^\beta \vec G_{j,\alpha,Y}(W)$ is a constant and so $\fint_{B(W_0,\rho)}\vec G_{j,\zeta,Y}(W)\,dW$ does not depend on $W_0$ or~$\rho$. By changing $W_0$ and $\rho$ appropriately and using the bound~\eqref{eqn:E:gradbd}, we see that
\[\int_Q \fint_{B(W_0,\rho)}
|\partial_Y^\beta\vec G_{j,\zeta,Y}(W)|\,dW\,dY<\infty\]
for any cube~$Q\subseteq\R^d$. Thus by Fubini's theorem
\[\partial_Y^\beta\fint_{B(W_0,\rho)}\vec G_{j,\alpha,Y}(W)\,dW
=
\fint_{B(W_0,\rho)}\partial_Y^\beta\vec G_{j,\alpha,Y}(W)\,dW\]
for almost every $Y\in\R^d$, and so 
\begin{align*}
\partial_X^\alpha\partial_Y^\beta \vec E^{L}_{X,j}(Y)
&=\partial_X^\alpha\partial_Y^\beta \vec E^{L}_{X,j,Z_0,r,q}(Y)
- \fint_{B(W_0,\rho)}\partial_Y^\beta\vec G_{j,\alpha,Y}(X)\,dW
\\&=\partial_X^\alpha\partial_Y^\beta \vec E^{L}_{X,j,Z_0,r,q}(Y)
-\partial_X^\alpha\partial_Y^\beta \vec E^{L}_{X,j,Z_0,r,q}(Y)
+\partial_X^\alpha\partial_Y^\beta \vec E^{L}_{X,j,Z_0,r,s}(Y)
\\&=\partial_X^\alpha\partial_Y^\beta \vec E^{L}_{X,j,Z_0,r,s}(Y)
\end{align*}
as desired.
\end{proof}

\subsection{Derivatives of $({L^*})^{-1}$}
\label{sec:FS:Fubini}

Recall from formula~\eqref{eqn:E:Linv} that, if $T\in Y^{-m,{q'}}(\RR^d)$, then
\begin{equation*}(((L^*)^{-1}T)_j)_{Z_0,r,{q'}}(X) = \overline{\langle T, \vec E^{L}_{X,j,Z_0,r,q}\rangle}.\end{equation*}
By the Hahn-Banach theorem, if $T\in Y^{-m,{q'}}(\RR^d)$, then there exist functions $\vec F_{\xi}$ with
\begin{equation*}
\sum_{m-d/q<|\xi|\leq m} \|\vec F_{\xi}\|_{L^{(q_\xi)'}(\RR^d)}<\infty \end{equation*} and where
\begin{equation*}\langle T,\vec\varphi\rangle
= \sum_{m-d/{q'}<|\xi|\leq m}  \int_{\RR^d} \partial^\xi \vec\varphi(Y)\cdot\overline{\vec F_{\xi}(Y)}\,dY.\end{equation*}
Thus,
\begin{equation}
\label{eqn:pi:E:repeat}
(((L^*)^{-1}T)_j)_{Z_0,r,{q'}}(X)
=
\sum_{m-d/q<|\xi|\leq m}  \int_{\RR^d} \overline{\partial_Y^\xi \vec E^{{L}}_{X,j,Z_0,r,{q}}(Y)}\cdot{\vec F_{\xi}(Y)}\,dY
.\end{equation}
We would like a similar integral formula for the derivatives of $(L^*)^{-1}T$.

\begin{theorem}\label{thm:deriv:Linv}
Let ${L}$ and ${q}$ satisfy the conditions of Definition~\ref{dfn:E:}. Assume that the bound~\eqref{eqn:E:Bochner} in Theorem~\ref{thm:fundamental} is valid for ${{p}}={q}$. 

Let $T=T_{\vec F,\xi}\in Y^{-m,{q'}}(\R^d)$ be a linear operator defined by
\begin{equation*}\langle T_{\vec F,\xi},\vec\varphi\rangle =  \int_{\RR^d} \partial^\xi \vec\varphi(Y)\cdot\overline{\vec F(Y)}\,dY\end{equation*}
for some $\xi$ and~$\vec F$ such that $m-d/q<|\xi|\leq m$ and $\vec F\in L^{(q_\xi)'}(\R^d)$ is compactly supported.

If $|\alpha|>m-d/{q'}$, and if $|\alpha|<m$ or $|\xi|<m$, then
\begin{equation}
\label{hihigradpi}
\partial^\alpha ((L^*)^{-1}T_{\vec F,\xi})_j(X)
=
\int_{\RR^d} \overline{\partial_X^\alpha\partial_Y^\xi \vec E^{{L}}_{X,j,Z_0,r,{q}}(Y)}\cdot{\vec F(Y)}\,dY
\end{equation}
and the integral converges absolutely for almost every $X\in\RR^d$. If $|\alpha|=|\xi|=m$, this formula is true for almost every $X\notin\supp F$.
\end{theorem}

\begin{proof}
By the bound~\eqref{eqn:E:Y:norm} and H\"older's inequality,
\begin{align*}
\int_{\RR^d} |\partial_Y^\xi {\vec E^{{L}}_{X,j,Z_0,r,{q}}(Y)}|\,|{\vec F(Y)}|\,dY<\infty .
\end{align*}
Let $Q_0\subset\R^d$ be a cube. We begin with the case where $\overline{Q_0}$ and $\supp F$ are disjoint. If $|\alpha|\leq m$, and if $\supp F$ is compact, then a covering argument combined with the bound~\eqref{eqn:E:Bochner} yields
\begin{align}
\label{eqn:alpha:bound}
\int_{Q_0}\int_{\RR^d} |\partial_X^\alpha\partial_Y^\xi {\vec E^{{L}}_{X,j,Z_0,r,{q}}(Y)}|\,|{\vec F(Y)}|\,dY\,dX<\infty
.\end{align}

By Fubini's theorem, if $\varphi\in C^\infty _0(Q_0)$ then
\begin{multline*}
\int_{Q_0}\partial^\alpha \varphi(X)
\int_{\RR^d}
\overline{\partial_Y^\xi \vec E^{{L}}_{X,j,Z_0,r,{q}}(Y)}\cdot{\vec F(Y)}\,dY
\,dX
\\=
(-1)^\alpha
\int_{Q_0}\varphi(X)
\int_{\RR^d} \overline{\partial_X^\alpha\partial_Y^\xi \vec E^{{L}}_{X,j,Z_0,r,{q}}(Y)}\cdot{\vec F(Y)}\,dY
\,dX
\end{multline*}
and so
\begin{equation*}
\partial^\alpha
\int_{\RR^d} \partial_Y^\xi \overline{\vec E^{L}_{X,j,Z_0,r,{q}}(Y)}\cdot{\vec F(Y)}\,dY
=
\int_{\RR^d} \overline{\partial_X^\alpha\partial_Y^\xi \vec E^{{L}}_{X,j,Z_0,r,{q}}(Y)}\,{\vec F(Y)}\,dY
\end{equation*}
as $L^1(Q_0)$ functions. Combining this result with formula~\eqref{eqn:pi:E:repeat} yields that
\begin{equation}
\label{higradpi}
\partial^\alpha((L^*)^{-1}T_{F,k,\xi})_{Z_0,r,{q'}}(X)
=
\int_{\RR^d} \overline{\partial_X^\alpha\partial_Y^\xi \vec E^{{L}}_{X,j,Z_0,r,{q}}(Y)}\cdot{\vec F(Y)}\,dY
\end{equation}
for almost every $X\notin \supp \vec F$. If $|\alpha|>m-d/{q'}$ then
\begin{equation*}\partial^\alpha ((L^*)^{-1}T_{F,k,\xi})=\partial^\alpha((L^*)^{-1}T_{F,k,\xi})_{Z_0,r,{q'}}\end{equation*}
and so formula~\eqref{hihigradpi} is valid for almost every $X\notin\supp \vec F$.

\begin{remark} If $|\alpha|<\min(m-d/2,m-d/{q'})$, then the bound~\eqref{eqn:E:Bochner} yields the bound~\eqref{eqn:alpha:bound} even if $\overline{Q_0}$ and $\supp F$ are not disjoint, and so in this case formula~\eqref{higradpi} is valid for almost every $X\in\RR^d$. 
\end{remark}

We are left with the case where $X\in\supp F$ and $|\alpha|+|\xi|<2m$.
We will show that the bound~\eqref{eqn:alpha:bound} is still valid; the argument given above then yields formula~\eqref{higradpi} and thus formula~\eqref{hihigradpi}.

Since $F$ has compact support, we may assume that $Q_0$ is large enough that $\supp F\subseteq Q_0$. Let $G_a$ be a grid of $2^{ad}$ pairwise-disjoint dyadic open subcubes of $Q_0$ of measure $2^{-ad}|Q_0|$ whose union (up to a set of measure zero) is~$Q_0$.  If $X\in Q_0$, let $Q_a(X)$ be the cube that satisfies $X\in Q_a(X)\in G_a$.  If $Q\in G_{a+1}$, let $P(Q)$ be the dyadic parent of the cube $Q$, that is, the unique cube with $Q\subset P(Q)\in G_a$.  Then by the monotone convergence theorem,
\begin{multline*}
\int_{Q_0}\int_{Q_0}|{\partial_X^\alpha\partial_Y^\xi \vec E^{{L}}_{X,j,Z_0,r,{q}}(Y)}|\,|{\vec F(Y)}|\,dX\,dY
\\
\begin{aligned}
&= \int_{Q_0}\sum_{a=0}^\infty \int_{4Q_a(Y)\setminus4Q_{a+1}(Y)} |{\partial_X^\alpha\partial_Y^\xi \vec E^{{L}}_{X,j,Z_0,r,{q}}(Y)}|\,|{\vec F(Y)}|\,dX\,dY
\\
&=\sum_{a=0}^\infty \sum_{Q\in G_{a+1}}\int_Q\int_{4P(Q)\setminus4Q} |{\partial_X^\alpha\partial_Y^\xi \vec E^{{L}}_{X,j,Z_0,r,{q}}(Y)}|\,|{\vec F(Y)}|\,dX\,dY.
\end{aligned}
\end{multline*}
By the bound~\eqref{eqn:E:Bochner} and Fubini's theorem we may interchange the order of integration. Applying H\"older's inequality first in $Q$ and then in sequence spaces,
\begin{multline*}
\int_{Q_0}\int_{Q_0}|{\partial_X^\alpha\partial_Y^\xi \vec E^{{L}}_{X,j,Z_0,r,{q}}(Y)}|\,|{\vec F(Y)}|\,dX\,dY
\\
\begin{aligned}
&=
\sum_{a=0}^\infty \sum_{Q\in G_{a+1}}\int_{4P(Q)\setminus4Q} \biggl(\int_Q|{\partial_X^\alpha\partial_Y^\xi \vec E^{{L}}_{X,j,Z_0,r,{q}}}|^{q_\xi}\biggr)^{1/q_\xi}\,dX
\biggl(\int_Q|{F}|^{(q_\xi)'}\biggr)^{1/(q_\xi)'}
\\&\leq
\sum_{a=0}^\infty
\biggl(
	\sum_{Q\in G_{a+1}}
	\biggl(\int_{4P(Q)\setminus4Q} \biggl(\int_Q|{\partial_X^\alpha\partial_Y^\xi \vec E^{{L}}_{X,j,Z_0,r,{q}}(Y)}|^{q_\xi}\,dY\biggr)^{1/q_\xi}\,dX\biggr)^{q_\xi}
\biggr)^{1/q_\xi}
\\&\qquad\times
\biggl(
	\sum_{Q\in G_{a+1}}
	\int_Q|{F}|^{(q_\xi)'}
\biggr)^{1/(q_\xi)'}
.
\end{aligned}
\end{multline*}
The final term is $\|F\|_{L^{(q_\xi)'}(Q_0)}=\|F\|_{L^{(q_\xi)'}(\RR^d)}<\infty $, so we need only bound the previous term.

If $a\geq 0$ is an integer and $Q\in G_{a+1}$, then by the bound~\eqref{eqn:E:Bochner}, and applying Lemma~\ref{lem:indep} to change $Z_0$ and $r$ as desired, we have that
\begin{align*}
\int_{4P(Q)\setminus 4Q}
\biggl(\int_Q |\partial_X^\alpha\partial_Y^\xi{\vec E^L_{X,j,Z_0,r,{q}}(X)}|^{q_\xi}\,dY\biggr)^{2/q_\xi}
\,dX
&\leq
C|Q|^{2m/d-1+2/q-2|\alpha|/d}
.\end{align*}
By H\"older's inequality,
\begin{multline*}
\int_{4P(Q)\setminus4Q} \biggl(\int_Q|{\partial_X^\alpha\partial_Y^\xi \vec E^{{L}}_{X,j,Z_0,r,{q}}(Y)}|^{q_\xi}\,dY\biggr)^{1/q_\xi}\,dX
\\\leq
\biggl(\int_{4P(Q)\setminus4Q} 	\biggl(\int_Q|{\partial_X^\alpha\partial_Y^\xi \vec E^{{L}}_{X,j,Z_0,r,{q}}(Y)}|^{q_\xi}\,dY\biggr)^{2/q_\xi}\,dX
\biggr)^{1/2}C|Q|^{1/2}
\\\leq
C|Q|^{m/d+1/q-|\alpha|/d}
.\end{multline*}
Thus, recalling that there are $2^{d(a+1)}$ cubes $Q$ in $G_a$ each satisfying $|Q|=2^{-(a+1)d}|Q_0|$,
\begin{multline*}
\int_{Q_0}\int_{Q_0}|{\partial_X^\alpha\partial_Y^\xi \vec E^{{L}}_{X,j,Z_0,r,{q}}(Y)}|\,|{\vec F(Y)}|\,dX\,dY
\\
\begin{aligned}
&\leq
C\|F\|_{L^{(q_\xi)'}(\RR^d)}
\sum_{a=0}^\infty
\biggl(
	\sum_{Q\in G_{a+1}}
	\biggl(
	|Q|^{m/d+1/q-|\alpha|/d}\biggr)^{q_\xi}
\biggr)^{1/q_\xi}
\\&=
C\|F\|_{L^{(q_\xi)'}(\RR^d)}
	|Q_0|^{m/d+1/q-|\alpha|/d}
\sum_{a=0}^\infty
	2^{ad/q_\xi}
	2^{-a(m+d/q-|\alpha|)}
.
\end{aligned}
\end{multline*}
Recall from formula~\eqref{pk} that $d/q_\xi=d/q-m+|\xi|$. Thus the final sum reduces to
\begin{equation*}\sum_{a=0}^\infty
	2^{-a(2m-|\xi|-|\alpha|)}
\end{equation*}
which converges provided $|\alpha|<m$ or $|\xi|<m$. This completes the proof.
\end{proof}

\subsection{The fundamental solution for operators of arbitrary order}\label{sec:FS:low}

In this section we show how use the fundamental solution for operators of high order to construct the fundamental solution for operators of arbitrary order.

We begin by defining a suitable higher order operator associated to each lower order operator and investigate its properties.
\begin{lemma}\label{lem:tilde:L}
Let $L:Y^{m,2}(\R^d)\to Y^{-m,2}(\R^d)$ be a bounded linear operator. Let $M$ be a nonnegative integer. Define
\begin{equation}
\label{eqn:tilde:L}
\widetilde L=
\Delta^M L\Delta^M,
\qquad \widetilde m=m+2M.
\end{equation}
Here $\Delta^M L\Delta^M$ is the operator given by
\[\langle (\Delta^M L\Delta^M)\vec\psi,\vec\varphi\rangle = \langle L(\Delta^M\vec\psi),\Delta^M\vec\varphi\rangle\text{ for all }\vec\varphi,\vec\psi\in Y^{\widetilde m,2}(\R^d).\]

Then:
\begin{enumerate}[label={\textup{(\alph*)}}]
\item If $1<q<\infty$ and $L$ is bounded or invertible $Y^{m,q}(\R^d)\to Y^{-m,q}(\R^d)$, then $\widetilde L$ is bounded or invertible $Y^{\widetilde m,q}(\R^d)\to Y^{-\widetilde m,q}(\R^d)$. If $L$ is invertible and in addition $M$ is large enough (depending on $d$, $m$, and~$q$), then $\widetilde L$ and $q$ satisfy the conditions of Definition~\ref{dfn:E:}.

\item If $L$ is bounded and invertible $Y^{m,2}(\R^d)\to Y^{-m,2}(\R^d)$, then $\Upsilon_{\widetilde L}= \Upsilon_L$, where $\Upsilon_L$ is as in Definition~\ref{dfn:compatible}.

\item If $L$ is of the form~\eqref{dfn:L}, then so is~$\widetilde L$, and
\begin{equation*}
\widetilde m-\mathfrak{a}_{\widetilde L}=m-\mathfrak{a}_L,
\qquad
\widetilde m-\mathfrak{b}_{\widetilde L}=m-\mathfrak{b}_L,
\qquad
\Pi_{\widetilde L}\supseteq\Pi_L,
\end{equation*}
where $\mathfrak{a}_L$ and $\mathfrak{b}_L$ are as in formulas \eqref{eqn:tildealpha} and~\eqref{eqn:tildebeta} and $\Pi_L$ is as in Definition~\ref{dfn:ppm}.

\item If $T\in Y^{-m,q}(\R^d)$, define $\widetilde T\in Y^{-\widetilde m,q}(\R^d)$ by
\begin{equation*}\langle \widetilde T,\vec\psi\rangle=\langle T,\Delta^M\vec\psi\rangle\text{ for all }\vec\psi\in Y^{\widetilde m,q'}(\R^d).\end{equation*}
If $L$ is invertible $Y^{m,q}(\R^d)\to Y^{-m,q}(\R^d)$, then
\begin{equation}
\label{eqn:tilde:L:inverse}
\Delta^M (\widetilde L^{-1} \widetilde T) = L^{-1}T.\end{equation}
\end{enumerate}

\end{lemma}

\begin{proof}
We need only consider the case $M>0$.
The polylaplacian is obviously bounded $\Delta^M:Y^{\widetilde m,p}(\R^d)\to Y^{ m,p}(\R^d)$ for any $1<p<\infty$ (in particular, for both $p=q$ and $p=q'$), and so if $L$ is bounded $Y^{\widetilde m,q}(\R^d)\to Y^{-\widetilde m,q}(\R^d)$ then $\widetilde L$ is bounded $Y^{\widetilde m,q}(\R^d)\to Y^{-\widetilde m,q}(\R^d)$.

It is well known (see, for example, \cite[Section 5.2.3]{Tri83}) that the Laplacian is a bounded and invertible operator $\dot W^{s,p}(\RR^d)\to \dot W^{s-2,p}(\RR^d)$ for any $1<p<\infty$ and any $-\infty<s<\infty$. Recall that there is a natural isomorphism between $Y^{m,p}(\R^d)$ and $\dot W^{m,p}(\R^d)$. 

Thus $L:Y^{m,q}(\R^d)\to Y^{-m,q}(\R^d)$ is bounded if and only if $\widetilde L:Y^{\widetilde m,q}(\R^d)\to Y^{-\widetilde m,q}(\R^d)$ is bounded, and  $L:Y^{m,q}(\R^d)\to Y^{-m,q}(\R^d)$ is invertible if and only if $\widetilde L:Y^{\widetilde m,q}(\R^d)\to Y^{-\widetilde m,q}(\R^d)$ is invertible.

If in addition $M>(d/2)\max(1/q,1/q')-m/2$, then $1-\widetilde m/d<1/q<\widetilde m/d$ and so $\widetilde L$ and $q$ satisfy the conditions of Definition~\ref{dfn:E:}.

Furthermore, $\Delta^{-1}$ is compatible, and so $(\widetilde L)^{-1}=\Delta^{-M}L^{-1}\Delta^{-M}$ is compatible if and only if $L^{-1}$ is compatible. Thus, $\Upsilon_L=\Upsilon_{\widetilde L}$.

There are real nonnegative constants $\kappa_\zeta$ such that $\Delta^M =\sum_{|\zeta|=M} \kappa_\zeta \partial^{2\zeta}$.
If $\vec u$ and $\vec\varphi$ lie in suitable function spaces, and $L$ is of the form~\eqref{dfn:L}, we have that
\begin{align*}
\langle \widetilde L\vec{u},\vec \varphi\rangle
&=
\langle L\Delta^M\vec{u},\Delta^M\vec{\varphi}\rangle
\\&=\int
\sum_{j,k=1}^N\sum_{|\alpha|\leq m}\sum_{|\beta|\leq m}
\sum_{|\xi|=M}\sum_{|\zeta|=M}
\partial^{\alpha+2\zeta}\varphi_j\,
\overline{\kappa_\zeta \kappa_\xi A^{j,k}_{\alpha,\beta} \,\partial^{\beta+2\xi} u_k}
.\end{align*}
We may rearrange our order of summation to see that $\widetilde L$ is an operator of the form~\eqref{dfn:L} of order $2\widetilde m$ with coefficients
\begin{equation}
\label{eqn:tilde:L:coeffs}
\widetilde A^{j,k}_{\nu,\varpi}=
\sum_{\substack{|\xi|=M\\2\xi<\varpi}}
\sum_{\substack{|\zeta|=M\\2\zeta<\nu}}
\kappa_\zeta \kappa_\xi A^{j,k}_{(\nu-2\zeta),(\varpi-2\xi)}
.\end{equation}
Furthermore,
\begin{equation*}
\sum_{j,k=1}^N\sum_{|\nu|\leq \widetilde m} \sum_{|\varpi|\leq \widetilde m}
\partial^{\nu}\varphi_j\,
\overline{\widetilde A^{j,k}_{\nu,\varpi} \,\partial^{\varpi }  u_k}
=
\sum_{j,k=1}^N\sum_{|\alpha|\leq m}\sum_{|\beta|\leq m}
\partial^{\alpha}\Delta^M\varphi_j\,
\overline{A^{j,k}_{\alpha,\beta} \,\partial^{\beta} \Delta^M u_k}
.\end{equation*}
If $\vec\varphi\in Y^{\widetilde m,q'}(\R^d)$ and $\vec u\in Y^{\widetilde m,q}(\R^d)$, then $\Delta^M\vec\varphi\in Y^{ m,q'}(\R^d)$ and $\Delta^M\vec u\in Y^{ m,q}(\R^d)$. Thus, if $q\in \Pi_L$ then the right hand side represents a $L^1(\R^d)$ function, and so $q\in \Pi_{\widetilde L}$. 

Finally, recall that $\Delta^M$ is invertible $Y^{\widetilde m,q}(\R^d)\to Y^{m,q}(\R^d)$. Thus if $\vec \Phi\in Y^{m,q}(\R^d)$, and $L:Y^{ m,q}(\R^d)\to Y^{m,q}(\R^d)$ and $\widetilde L:Y^{\widetilde m,q}(\R^d)\to Y^{\widetilde m,q}(\R^d)$ are bounded and invertible, then
\begin{align*}
\langle L(\Delta^M(\widetilde L)^{-1}\widetilde T),\vec\Phi\rangle
&=
\langle L(\Delta^M(\widetilde L)^{-1}\widetilde T),\Delta^M \Delta^{-M}\vec\Phi\rangle
\\&=
\langle \widetilde L((\widetilde L)^{-1}\widetilde T), \Delta^{-M}\vec\Phi\rangle
\\&=
\langle \widetilde T,\Delta^{-M}\vec\Phi\rangle
\\&=
\langle  T,\Delta^M\Delta^{-M}\vec\Phi\rangle
\\&=
\langle  T,\vec\Phi\rangle
\end{align*}
and so $\Delta^M(\widetilde L)^{-1}\widetilde T=(L)^{-1}T$. This completes the proof.
\end{proof}

Thus, natural conditions on~$L$ guarantee that $\widetilde L$ has a fundamental solution.

We now use $\vec E^{\widetilde L}$ to construct $\vec E^L$ for operators of arbitrary order. Theorem~\ref{Fslowth} (with $E^L_{j,k}(Y,X)=(\vec E^L_{X,k})_j(Y)$ and $L$ and $L^*$ interchanged as needed) comprises most of Theorem~\ref{thm:intro:fundamental}; the remaining property cited in Theorem~\ref{thm:intro:fundamental} (the uniqueness of the fundamental solution) will be addressed in Section~\ref{sec:FS:unique}.

\begin{theorem}\label{Fslowth}
Let $L$ be an operator of order $2m$ of the form~\eqref{dfn:L} that satisfies the ellipticity condition~\eqref{gi} such that $2\in\Pi_L$, where $\Pi_L$ is  the interval of Definition~\ref{dfn:ppm}. Let $M$ be the smallest nonnegative integer with $m+2M>d/2$.
Let $\widetilde L$ be given by formula~\eqref{eqn:tilde:L}.

Suppose in addition that the Caccioppoli-Meyers inequality for $\widetilde L$ holds, that is, that there is an interval $ S_{\widetilde L}$ with $2\in S_{\widetilde L} \subseteq [2,4]\cap\Pi_L$ such that if $p\in S_{\widetilde L}$, if $Q\subset\RR^d$ is a cube with sides parallel to the coordinate axes, and if $\vec u$ is a representative of an element of $Y^{\widetilde m,p}(2Q)$, then we have the estimate
\begin{equation}
\label{eqn:Meyers:FS:2}
\sum_{j=0}^{\widetilde m}
|Q|^{j/d}
\|\nabla^j\vec u\|_{L^p(Q)}
\\\leq
C|Q|^{1/p-1/2}\|\vec u\|_{L^2(2Q)}
+C|Q|^{\widetilde m/d}\|\widetilde L\vec u\|_{Y^{-\widetilde m,p}(2Q)}
.\end{equation}
If $L$ satisfies either the bound \eqref{eqn:intro:bound} or the bound~\eqref{eqn:intro:bound:Bochner}, then this condition is true with $S_{\widetilde L} =\Pi_{\widetilde L}\cap  \Upsilon_L\cap[2,4]\supsetneq\{2\}$, with $\Upsilon_L$ given by formula~\eqref{eqn:SL}.

Then there exists some array of functions $\vec E^{L}_{X,j}(Y)$ with the following properties.

Suppose that $\alpha$ and $\beta$ are two multiindices with
$m-d/2\leq |\alpha|\leq m$, $m-d/2\leq |\beta|\leq m$, and $(|\alpha|,|\beta|)\neq (m-d/2,m-d/2)$.
If $\Pi_L$ does not contain a neighborhood of~$2$, then we impose the stronger condition $m-d/2<|\alpha|\leq m$, $m-d/2<|\beta|\leq m$.

Suppose further that $Q$ and $\Gamma$ are two cubes in $\R^d$ with $|Q|=|\Gamma|$ and $\Gamma\subset 8Q\setminus 4Q$. Then the partial derivative $\partial^\alpha_X \partial^\beta_Y \vec E^{L}_{X,j}(Y)$ exists as a locally $L^2(Q\times\Gamma)$ function and satisfies the bounds
\begin{align}\label{Elowgradbd}
\int_{Q}\int_{\Gamma}|\partial^\alpha_X \partial^\beta_Y \vec E^{L}_{X,j}(Y)|^2 &\leq C|Q|^{(4m-2|\alpha|-2|\beta|)/d}
,\\
\label{Elowgradbdmu}
\int_\Gamma \biggl(\int_{Q} |\partial_X^\alpha \partial_Y^\beta \vec E^{L}_{X,j}(Y)|^{p_\beta}\,dY\biggr)^{2/p_\beta}\,dX
&\leq C|Q|^{2m/d-1+2/p-2|\alpha|}
\end{align}
for all $p\in \Upsilon_L\cap S_{\widetilde L}$ with $m-d/p'<|\alpha|$, $m-d/p<|\beta|$.

Furthermore, we have the symmetry property
\begin{equation}\label{sym}
\partial^\alpha_X \partial^\beta_Y (\vec E^{L}_{X,j}(Y))_k=\overline{\partial^\alpha_X \partial^\beta_Y (\vec E^{L^*}_{Y,k}(X))_j}
\end{equation}
for almost every $X$, $Y\in\R^d\times\R^d$.

Finally, let $\Upsilon_L$ be as in Definition~\ref{dfn:compatible}. Suppose that $q\in \Upsilon_L\cap ((-\infty,2)\cup S_{\widetilde L})$ and $m-d/q<|\xi|\leq m$.
Let $T=T_{\vec F,\xi}\in Y^{-m,{q'}}(\R^d)$ be a linear operator defined by
\begin{equation*}\langle T_{\vec F,\xi},\vec\varphi\rangle =  \int_{\RR^d} \partial^\xi \vec\varphi(Y)\cdot\overline{\vec F(Y)}\,dY\end{equation*}
for some compactly supported $\vec F\in L^{(q_\xi)'}(\R^d)$.
Whenever $|\zeta|>m-d/q'$, we have that
\begin{equation}
\label{hihigradpi:lower}
\partial^\zeta ((L^*)^{-1}T_{\vec F,\xi})_j(X)
=
\int_{\RR^d} \overline{\partial_X^\zeta\partial_Y^\xi \vec E^{{L}}_{X,j}(Y)}\cdot{\vec F(Y)}\,dY
\end{equation}
and the integral converges absolutely for almost every $X\notin\supp F$. If in addition $|\zeta|<m$ or $|\xi|<m$ then formula~\eqref{hihigradpi:lower} is valid for almost every $X\in\RR^d$.
\end{theorem}

\begin{proof}
If $L$ satisfies either the bound \eqref{eqn:intro:bound} or the bound~\eqref{eqn:intro:bound:Bochner}, then by Lemma~\ref{lem:tilde:L} and formula~\eqref{eqn:tilde:L:coeffs}, so does~$\widetilde L$. By Lemma~\ref{lem:tilde:L}, $\Upsilon_{\widetilde L} =\Upsilon_L$. By Lemma~\ref{NPLp}, $\Upsilon_L$ contains a neighborhood of~$2$ and so $\Upsilon_L\cap [2,\infty)$ contains values greater than~$2$.
The inequality~\eqref{eqn:Meyers:FS:2} is valid for all $p\in\Pi_{\widetilde L}\cap  \Upsilon_{\widetilde L}\cap[2,\infty)$ by Theorem~\ref{umpm}.

By assumption and Lemma~\ref{lem:L:invertible}, $2\in\Upsilon\cap\Pi_L$ and $1-\widetilde m/d<1/2<\widetilde m/d$. Also observe that $\widetilde L$ satisfies the conditions of Definition~\ref{dfn:E:} and Theorem~\ref{thm:fundamental} for all $q\in \Upsilon_L\cap\Pi_L$ with $1-\widetilde m/d<1/q<\widetilde m/d$ and all $p\in S_{\widetilde L}\cap\Upsilon_{\widetilde L}\cap \Pi_{\widetilde L}\cap(1,2q]$ with $1-\widetilde m/d<1/p<\widetilde m/d$ (in particular, for $q=p=2$).
Consequently $\widetilde L$ satisfies the conditions of Lemma~\ref{lem:indep} for all $q_1$, $q_2\in\Upsilon_L\cap\Pi_L$ with $1/q_1$, $1/q_2\in (1-m/d,m/d)$.

If $\Pi_{ L}$ contains an open neighborhood of~$2$, then by Lemma~\ref{NPLp}, $\Upsilon_L$ also contains an open neighborhood of~$2$. Thus the conditions of Theorem~\ref{thm:renormalize} are valid whenever $d$ is even.

If $d$ is odd, or if $\Pi_L$ does not contain a neighborhood of~$2$, let $\vec E^{\widetilde L}_{X,j}=\vec E^L_{X,j,0,1,2}$ be as in Definition~\ref{dfn:E:}. 

If $d$ is even, and if $\Pi_L$ contains a neighborhood of~$2$, we let $\vec E^{\widetilde L}_{X,j}$ be as in Theorem~\ref{thm:renormalize}. 

In either case, by Theorem~\ref{thm:fundamental}, $\partial_X^\zeta\partial_Y^\xi\vec E^{\widetilde L}_{X,j}(Y)$ exists for almost every $(X,Y)\in\RR^d\times\RR^d$ and every $\xi$, $\zeta$ with $|\xi|$, $|\zeta|\in [0,\widetilde m]$. 
We define
\begin{equation*}
\vec E^{L}_{X,j}(Y)=\sum_{|\varpi|=M}\sum_{|\nu|=M} \kappa_\varpi \kappa_\nu\partial^{2\varpi}_X \partial^{2\nu}_Y \vec E^{\widetilde L}_{X,j}(Y)
.\end{equation*}

The bounds~\eqref{Elowgradbd} and~\eqref{Elowgradbdmu} follow from Theorem~\ref{thm:fundamental}, Lemma~\ref{lem:indep} and Lemma~\ref{lem:tilde:L}. The symmetry property~\eqref{sym} follows from the symmetry property~\eqref{eqn:E:sym} for $\vec E^{\widetilde L}_{X,j}$.

We are left with formula~\eqref{hihigradpi:lower}. This property follows from Theorem~\ref{thm:deriv:Linv} if $2m>d$ and so $M=0$. If $2m\leq d$, let $m-d/q<|\xi|\leq m$ and $\vec F$ satisfy the conditions given in the theorem statement. Let $T=T_{\vec F,\xi}$ and let $\widetilde T$ be as in formula~\eqref{eqn:tilde:L:inverse}. Observe that
\begin{equation*}\langle \widetilde T,\vec\psi\rangle
=
\langle T,\Delta^M\psi\rangle
= \sum_{|\nu|=M}\kappa_\nu
\int_{\RR^d} \partial^{\xi+2\nu} \vec\psi(Y)\cdot\overline{\vec F(Y)}\,dY\end{equation*}
and so $\widetilde T$ is a (linear combination of) operators as in Theorem~\ref{thm:deriv:Linv}.
By formula~\eqref{hihigradpi} and linearity, we have that if $|\widetilde\zeta|>\widetilde m-d/q$, then
\begin{equation*}
\partial^{\widetilde\zeta} ((\widetilde L^*)^{-1}\widetilde T)_j(X)
=
\sum_{|\nu|=M}\kappa_\nu \int_{\RR^d} \overline{\partial_X^{\widetilde\zeta} \partial_Y^{\xi+2\nu} \vec E^{{\widetilde L}}_{X,j}(Y)}\cdot{\vec F(Y)}\,dY
\end{equation*}
for almost every $X$ or almost every $X\notin\supp \vec F$.
In particular, if $m-d/q'<|\zeta|\leq m$ and $|\varpi|=M$, then $\widetilde m-d/q<|\zeta+2\varpi|\leq \widetilde m$, and so
\begin{align*}
\partial^{\zeta} (\Delta^M(\widetilde L^*)^{-1}\widetilde T)_j(X)
&=
\sum_{|\varpi|=M}
\sum_{|\nu|=M}\kappa_\varpi\kappa_\nu \int_{\RR^d} \overline{\partial_X^{2\varpi+\zeta} \partial_Y^{\xi+2\nu} \vec E^{{\widetilde L}}_{X,j}(Y)}\cdot{\vec F(Y)}\,dY
\\&=
\int_{\RR^d} \overline{\partial_X^{\zeta} \partial_Y^{\xi} \vec E^{{L}}_{X,j}(Y)}\cdot{\vec F(Y)}\,dY
\end{align*}
Observe that $\widetilde{(L^*)}=(\widetilde L)^*$. By formula~\eqref{eqn:tilde:L:inverse} with $L$ replaced by $L^*$, formula~\eqref{hihigradpi:lower} is valid.
\end{proof}

\begin{remark}
Theorem~\ref{Fslowth} involves conditions on~$\widetilde L=\Delta^ML\Delta^M$ for the smallest $M$ such that $\vec E^{\widetilde L}_{X,j,Z_0,r,2}$ exists. The fundamental solution also exists for larger values of~$M$. However, there is no loss of generality in Theorem~\ref{Fslowth} in taking the smallest available $M$; that is, we claim that if the Caccioppoli-Meyers inequality \eqref{eqn:Meyers:FS:2} is valid for $\widetilde L=\Delta^ML\Delta^M$, and if $L:Y^{m,p}(\Omega)\to Y^{-m,p}(\Omega)$ is bounded for all open sets~$\Omega$, then it is valid for $\widetilde L=\Delta^NL\Delta^N$ for any integer $N$ with $0\leq N\leq M$.

We now prove the claim.
Suppose that $p\geq 2$ and the Caccioppoli or Meyers inequality
\begin{multline*}
\sum_{j=0}^{m+2M}
|Q|^{j/d}
\biggl(\int_Q |\nabla^j\vec w|^p\biggr)^{1/p}
\\\leq
C|Q|^{1/p-1/2}
\biggl(\int_{2Q} |\vec w|^2\biggr)^{1/2}
+C|Q|^{(m+2M)/d}\|\Delta^M L\Delta^M\vec w\|_{Y^{-m-2M,p}(2Q)}
\end{multline*}
is valid for all $\vec w\in Y^{m+2M,p}(2Q)$ for some cube~$Q$. Let $0\leq N<M$.

It is well known (see \cite[Chapter~VI, Section~3]{Ste70}) that there is a bounded, linear extension operator $E$ such that for all $k\in\NN_0$ and all $1\leq p<\infty$ we have that $\|E\vec u\|_{W^{k,p}(\R^d)}\leq C_{k,p}\|\vec u\|_{W^{k,p}(2Q)}$. Recall that $\Delta^{M-N}$ is an isomorphism from $W^{k+2M-2N,p}(\R^d)$ to $W^{k,p}(\R^d)$.

Choose some $\vec u\in W^{m+2N,p}(2Q)$. Let $\vec v=\Delta^{-(M-N)}(E\vec u)$. Then $\vec v\in W^{m+2M,p}(\R^d)$ and also satisfies 
\[\|\nabla^{2M-2N}v\|_{L^2(2Q)}\leq \|\nabla^{2M-2N}v\|_{L^2(\R^d)}\leq C\|E\vec u\|_{L^2(\R^d)}\leq C^2 \|\vec u\|_{L^2(2Q)}.\] 
Let $\vec w=\vec v+\vec P$, where $\vec P$ is a polynomial of degree at most $2M-2N-1$ such that  $\int_{2Q}\partial^\gamma \vec w=0$ for all $|\gamma|\leq 2M-2N-1$. We have that $\Delta^{M-N}\vec w=\Delta^{M-N}\vec v=E\vec u=u$ in~$2Q$.
We compute
\begin{align*}
\sum_{j=0}^{m+2N}
|Q|^{j/d}
\biggl(\int_Q |\nabla^j\vec u|^p\biggr)^{1/p}
&=
\sum_{j=0}^{m+2N}
|Q|^{j/d}
\biggl(\int_Q |\nabla^j\Delta^{M-N}\vec w|^p \biggr)^{1/p }
\\&\leq
\sum_{k=2M-2N}^{m+2M}
|Q|^{(k-2M+2N)/d}
\biggl(\int_Q |\nabla^k \vec w|^p \biggr)^{1/p }
.\end{align*}
By the Meyers inequality for $\Delta^M L\Delta^M$,
\begin{align*}
\sum_{j=0}^{m+2N}
|Q|^{j/d}
\biggl(\int_Q |\nabla^j\vec u|^p \biggr)^{1/p }
&\leq
C|Q|^{1/p -1/2-(2M-2N)/d}
\biggl(\int_{2Q} |\vec w|^2\biggr)^{1/2}
\\&\qquad
+C|Q|^{(m+2N)/d}\|\Delta^M L\Delta^M\vec w\|_{Y^{-m-2M,p }(2Q)}
.\end{align*}
By the Poincar\'e inequality and because $\Delta^M\vec w=\Delta^N\vec u$,
\begin{align*}
\sum_{j=0}^{m+2N}
|Q|^{j/d}
\biggl(\int_Q |\nabla^j\vec u|^p \biggr)^{1/p }
&\leq
C|Q|^{1/p -1/2}
\biggl(\int_{2Q} |\nabla^{2M-2N}\vec w|^2\biggr)^{1/2}
\\&\qquad
+C|Q|^{(m+2N)/d}\|\Delta^{N} L\Delta^N\vec u\|_{Y^{-m-2N,p }(2Q)}
.\end{align*}
Finally, using the estimate $\|\nabla^{2M-2N}\vec w\|_{L^2(2Q)}
=\|\nabla^{2M-2N}\vec v\|_{L^2(2Q)}
\leq C \|\vec u\|_{L^2 (2Q)}$, we see that the Caccioppoli-Meyers estimate for $\Delta^NL\Delta^N$ is also valid.
\end{remark}

\subsection{Uniqueness}\label{sec:FS:unique}

We have constructed a fundamental solution; we now show that it is unique.

\begin{theorem}\label{thm:unique} Let $L:Y^{m,q}(\R^d)\to Y^{-m,q}(\R^d)$ be bounded and invertible. Suppose that $\vec \Psi_{X,j}$ and $\vec \Gamma_{X,j}$ are such that the bound~\eqref{Elowgradbd} and formula~\eqref{hihigradpi:lower} are valid with $\vec E^L$ replaced by either $\vec \Psi$ or $\vec \Gamma$.

Then $\partial_X^\alpha \partial_Y^\beta \vec \Psi_{X,j}(Y)=\partial_X^\alpha \partial_Y^\beta \vec \Gamma_{X,j}(Y)$ for almost every $(X,Y)\in\RR^d\times\R^d$ and all $\alpha$, $\beta$ as in Theorem~\ref{Fslowth}.
\end{theorem}

\begin{proof}
By the bound~\eqref{Elowgradbd}, we have that $\partial_X^\alpha \partial_Y^\beta \vec \Psi_{X,j}$ and $\partial_X^\alpha \partial_Y^\beta \vec \Gamma_{X,j}$ are locally integrable away from $Y=X$ for almost every $X\in\R^d$. By formula~\eqref{hihigradpi:lower},
\[\int_{\RR^d} \overline{\partial_X^\alpha\partial_Y^\beta \vec \Psi_{X,j}(Y)}\cdot{\vec F(Y)}\,dY
=
\int_{\RR^d} \overline{\partial_X^\alpha\partial_Y^\beta \vec \Gamma_{X,j}(Y)}\cdot{\vec F(Y)}\,dY
\]
for all sufficiently nice test functions~$\vec F$. The result follows from the Lebesgue differentiation theorem.
\end{proof}

\bibliographystyle{amsalpha}
\bibliography{bibli}
\end{document}